\documentclass[reqno]{amsart}
\usepackage{macros/preprintStyle}
\pdfoutput=1 %

\usepackage{bm}
\usepackage{comment}
\usepackage{adjustbox} 

\usepackage{algorithmic}

\usepackage{graphicx}
\usepackage{color}
\usepackage{circuitikz}
\usepackage{tikz}
\usetikzlibrary{calc,positioning,shapes}
\usetikzlibrary{patterns,decorations.pathmorphing,decorations.markings}
\usetikzlibrary{external}
\usepackage{pgfplots}
\pgfplotsset{compat=newest}
\usepackage[margin=10pt,font=small,labelfont=bf,labelsep=endash]{caption}
\usepackage{subcaption}

\usetikzlibrary{intersections, backgrounds}
\usetikzlibrary{circuits}
\usetikzlibrary{circuits.ee.IEC}
\usepackage{pgfplots}
\pgfplotsset{compat=newest}

\definecolor{color0}{rgb}{0.12156862745098,0.466666666666667,0.705882352941177}
\definecolor{color1}{rgb}{1,0.498039215686275,0.0549019607843137}
\definecolor{color2}{rgb}{0.172549019607843,0.627450980392157,0.172549019607843}
\definecolor{color3}{rgb}{0.83921568627451,0.152941176470588,0.156862745098039}
\definecolor{color4}{rgb}{0.580392156862745,0.403921568627451,0.741176470588235}
\definecolor{color5}{rgb}{0,0,0}

\definecolor{mycolor1}{rgb}{0.00000,0.44700,0.74100}
\definecolor{mycolor2}{rgb}{0.85000,0.32500,0.09800}
\definecolor{mycolor3}{rgb}{0.92900,0.69400,0.12500}
\definecolor{mycolor4}{rgb}{0.46600,0.67400,0.18800}
\definecolor{mycolor5}{rgb}{0.49400,0.18400,0.55600}

\usepackage{amssymb}

\usepackage{mathrsfs}

\usepackage{bm}
\usepackage{comment}
\usepackage{adjustbox} 

\usepackage{mathtools}

\usepackage{cleveref}

\usepackage{booktabs}
\usepackage{subcaption}

\usepackage{xspace}

\usepackage{graphicx}
\usepackage{color}
\usepackage{circuitikz}
\usepackage{tikz}
\usetikzlibrary{calc,positioning,shapes}
\usetikzlibrary{patterns,decorations.pathmorphing,decorations.markings}
\usetikzlibrary{external}

\usepackage{pgfplots}
\pgfplotsset{compat=newest}
\usepackage[margin=10pt,font=small,labelfont=bf,labelsep=endash]{caption}
\usepackage{subcaption}

\usetikzlibrary{intersections, backgrounds}
\usetikzlibrary{circuits}
\usetikzlibrary{circuits.ee.IEC}
\usepackage{pgfplots}
\pgfplotsset{compat=newest}

\definecolor{color0}{rgb}{0.12156862745098,0.466666666666667,0.705882352941177}
\definecolor{color1}{rgb}{1,0.498039215686275,0.0549019607843137}
\definecolor{color2}{rgb}{0.172549019607843,0.627450980392157,0.172549019607843}
\definecolor{color3}{rgb}{0.83921568627451,0.152941176470588,0.156862745098039}
\definecolor{color4}{rgb}{0.580392156862745,0.403921568627451,0.741176470588235}
\definecolor{color5}{rgb}{0,0,0}

\definecolor{mycolor1}{rgb}{0.00000,0.44700,0.74100}
\definecolor{mycolor2}{rgb}{0.85000,0.32500,0.09800}
\definecolor{mycolor3}{rgb}{0.92900,0.69400,0.12500}
\definecolor{mycolor4}{rgb}{0.46600,0.67400,0.18800}
\definecolor{mycolor5}{rgb}{0.49400,0.18400,0.55600}

\newcommand{\rev}[1]{{ #1}}

\newcommand{\smallO}{\raisebox{0.2ex}{\scalebox{0.7}{$\mathcal{O}$}}}
\newcommand{\bigO}{\cal{O}}

\DeclareMathOperator*{\argmin}{arg\,min}
\DeclareMathOperator*{\spanop}{span}

\newcommand{\bb}[1]{\mathbb{#1}}
\renewcommand{\rm}[1]{\mathrm{#1}}
\newcommand{\cal}[1]{\mathcal{#1}}

\newcommand{\tr}{\mathrm{tr}}

\newcommand{\ran}{\mathrm{ran}}

\newcommand{\rank}{\mathrm{rank}}

\newcommand{\RBspace}{\mathcal{V}}

\newcommand{\RBproj}{\boldsymbol{\Pi}}
\newcommand{\projDiff}{\boldsymbol{\Delta}}

\renewcommand{\d}[1]{\mathrm{d}#1}

\renewcommand{\i}{\mathrm{i}}

\renewcommand{\vec}[1]{\mathbf{#1}}
\newcommand{\gvec}[1]{\boldsymbol{#1}}  

\newcommand{\mat}[1]{\mathbf{#1}}

\newcommand{\prmtr}{\mu}
\newcommand{\prmtrVec}{\boldsymbol{\mu}}

\newcommand{\prmtrSet}{\mathcal{P}}

\newcommand{\EigVal}{\lambda}

\newcommand{\fdim}{N}
\newcommand{\rdim}{r}
\newcommand{\Pdim}{d}
\newcommand{\nrAffine}{Q}
\newcommand{\idxAffine}{q}

\newcommand{\Ord}{n}

\newcommand{\multiidx}{\gvec{\beta}}
\newcommand{\multiidxSec}{\gvec{\nu}}

\title[MOR for pEVPs via Taylor-RBM]{
Model Order Reduction for Parametric Hermitian Eigenvalue Problems: 
Local High-Order Approximation with Taylor-Reduced Basis Method
}
\author{Benjamin Stamm${}^\dagger$  \and Zhuoyao Zeng${}^\dagger$}

\address{${}^{\dagger}$ Institute of Applied Analysis and Numerical Simulation, University of Stuttgart, Pfaffenwaldring 57, 70569 Stuttgart, Germany}
\email{\{benjamin.stamm,zhuoyao.zeng\}@mathematik.uni-stuttgart.de}

\date{\today}

\begin{document}

\begin{abstract}
This paper is concerned with the Taylor-reduced basis method (Taylor-RBM) 
for the efficient approximation of eigenspaces of 
large-scale parametric Hermitian matrices.
The Taylor-RBM is a local model order reduction method, 
which constructs an approximation space 
\rev{by using the ranges of scaled derivatives 
of the spectral projector at a reference point in the parameter domain. 
For a prescribed order $n$, 
under local $\mathrm{C}^{n+1}$-regularity,} 
we prove local convergence of orders $n+1$ and $2n+2$ 
for the eigenspace and the Ritz-value sum, respectively, 
and present a computationally efficient procedure 
for assembling the Taylor reduced basis space.
Since this method is tightly connected 
to multivariate perturbation theory, 
we also analyze a projector approximation 
obtained from the finite Taylor polynomial 
of the eigenprojector and compare it 
with the Taylor-RBM. 
\end{abstract}

\maketitle



\section{Introduction}\label{sec:intro}

Let $\fdim\in\bb{N}$ with $\fdim\gg1$ and let
$\prmtrSet\subseteq\bb{R}^{\Pdim}$ be an open parameter domain.
We consider a Hermitian matrix-valued mapping
$\mat{A}:\prmtrSet\to\bb{C}^{\fdim\times\fdim}$ and fix a reference
parameter $\prmtrVec_0\in\prmtrSet$ 
{as well as a derivative order
$\Ord\in\bb{N}_0$.
Our standing assumption is that $\mat{A}$ is of class
$\mathrm{C}^{\Ord +1}$ in a neighborhood of $\prmtrVec_0$.
}

For practical relevance and simplicity, we work in the
finite-dimensional matrix setting.
There are no fundamental obstacles to formulating the results for
sufficiently smooth families of unbounded self-adjoint operators once
the relevant technical details, such as the operator domains, the type of
spectrum, the absence of spectral pollution and the matrix
representations, have been specified.

Denote the smallest $K+1$ \emph{distinct} eigenvalues of
$\mat{A}(\prmtrVec_0)$ by
\begin{equation}\label{eq:target_eigenvalues}
\EigVal_1<\cdots<\EigVal_K<\EigVal_{K+1},
\end{equation}
and let $\mat{P}_{\!\!j}(\prmtrVec_0)$ be the orthogonal eigenprojector
associated with $\EigVal_j$.
The multiplicity of each $\EigVal_j$ is
$m_i\coloneqq\rank\mat{P}_{\!\!j}(\prmtrVec_0)$, so that
\begin{equation*}
\ker (\mat{A}(\prmtrVec_0) - \EigVal_j \mat{I}) 
\,=\,
\ran \mat{P}_{\!\!j}(\prmtrVec_0) 
\qquad \text{for } j=1,\ldots,K. 
\end{equation*}
We set
\begin{equation}\label{eq:target_projector}
\mat{P}(\prmtrVec_0)
\coloneqq\sum_{j=1}^K\mat{P}_{\!j}(\prmtrVec_0)
\qquad\text{and}\qquad
M
\coloneqq 
(\rank\mat{P}(\prmtrVec_0))
=
\sum_{j=1}^K m_i.
\end{equation}
{The strict inequality
$\EigVal_K<\EigVal_{K+1}$ isolates this
rank-$M$ target cluster from the rest of the spectrum.
By standard spectral perturbation theory \cite{Kato76,Bau84}, 
this isolated cluster determines 
a unique rank-\(M\) total eigenprojector \(\mat P(\prmtrVec)\) 
of class \(\mathrm C^{\Ord +1}\) locally around \(\prmtrVec_0\).

For integers $1\leq m\leq N$, we denote by
\begin{equation}\label{eq:def_grassmannian_projectors}
\operatorname{Gr}(m,N)
\coloneqq
\bigl\{
\mat{Q}\in\bb{C}^{N\times N}:
\mat{Q}^2=\mat{Q}=\mat{Q}^*,\ 
\tr\mat{Q}=m
\bigr\}
\end{equation}
the orthogonal-projector representation of the complex Grassmannian of
$m$-dimensional subspaces of $\bb{C}^N$.
For any $\mat{Q}\in\operatorname{Gr}(M,\fdim)$,
the \emph{Ritz value sum} of $\mat{A}(\prmtrVec)$
associated with $\ran\mat{Q}$ is given by 
\(
\tr (\mat{A}(\prmtrVec)\mat{Q} ).
\)
By the Ky Fan minimum principle \cite{Fan49},
this is uniquely minimized by the target projector, 
i.e., 
\[
\mat{P}(\prmtrVec)
\; = \!
\argmin_{\mat{Q}\in\operatorname{Gr}(M,\fdim)}
\tr (\mat{A}(\prmtrVec)\mat{Q} )
\qquad\text{for }\prmtrVec\text{ near }\prmtrVec_0.
\]
We aim to approximate both $\mat{P}(\prmtrVec)$ and
$\tr(\mat{A}(\prmtrVec)\mat{P}(\prmtrVec))$
by solving a low-dimensional Rayleigh--Ritz problem,
while guaranteeing high-order accuracy for
$\prmtrVec$ near $\prmtrVec_0$.
}

{
The \emph{Taylor reduced basis method} (Taylor-RBM) constructs its
reduced space from the ranges of the scaled derivatives of the target
projector, i.e., 
$\RBspace_{\!\Ord}(\prmtrVec_0)
\coloneqq
\spanop\bigl\{\ran\mat{P}^{(\multiidx)}:
|\multiidx|\leq\Ord\bigr\}.$
For each query parameter, 
the Taylor-RBM applies the \emph{Rayleigh--Ritz method}
in this space and determines
\begin{equation} 
\label{eq:reduced_sol_intro}
\mat{Q}_{\rm{rb}}(\prmtrVec)
\; \in \!
\argmin_{\substack{
\mat{Q}\in\operatorname{Gr}(M,\fdim)\\
\ran\mat{Q}\subseteq\RBspace_{\!\Ord}(\prmtrVec_0)}}
\tr(\mat{A}(\prmtrVec)\mat{Q}).
\end{equation}
The rationale and mechanism underlying the Taylor-RBM can be outlined as follows.
The multivariate Taylor theorem
\cite[Thm.~3.11]{Kna16}, applied componentwise to $\mat{P}$, gives
a finite polynomial
$\mat{P}^{[\Ord]}(\prmtrVec)$ that approximates
$\mat{P}(\prmtrVec)$ with order $\Ord+1$ 
for \(\prmtrVec\) sufficiently close to \(\prmtrVec_0\).
Although $\mat{P}^{[\Ord]}(\prmtrVec)$ is generally not a projector, 
one can extract from \(\mat P^{[\Ord]}(\prmtrVec)\) a rank-\(M\) orthogonal projector \(\mat Q_{\rm pt}(\prmtrVec)\) satisfying
$\ran(\mat{Q}_{\rm{pt}}(\prmtrVec))
\subseteq
\ran(\mat{P}^{[\Ord]}(\prmtrVec))
$
and 
\(
\|\mat{Q}_{\rm{pt}}(\prmtrVec)-\mat{P}(\prmtrVec)\|
\lesssim
\|\mat{P}^{[\Ord]}(\prmtrVec)-\mat{P}(\prmtrVec)\|
=\bigO(|\prmtrVec-\prmtrVec_0|^{\Ord+1}).
\)
However, building $\mat{P}^{[\Ord]}(\prmtrVec)$ 
and $\mat{Q}_{\mathrm{pt}}(\prmtrVec)$ for many different $\prmtrVec$ 
could be expensive,
as the calculation of~$\mat{P}^{[\Ord]}(\prmtrVec)$ 
involves dense $\fdim \times \fdim$ matrices. 
Alternatively, 
observe that $\ran(\mat{P}^{[\Ord]}(\prmtrVec)) \subseteq \RBspace_\Ord(\prmtrVec_0)$,
or equivalently \({\RBproj}_{\Ord}^{\perp}\mat P^{[\Ord]}(\prmtrVec)=\mat0\),
with ${\RBproj}_{\Ord}^\perp$ the orthogonal projector 
with  $\RBspace_\Ord(\prmtrVec_0)$ as kernel.
Hence, the Rayleigh-Ritz procedure \eqref{eq:reduced_sol_intro},
which can be reformulated as an eigenvalue problem
of dimension \(\dim\RBspace_\Ord(\prmtrVec_0)\ll \fdim\), 
should yield a "reasonable" approximation $\mat{Q}_{\rm{rb}}(\prmtrVec)$
 of~$\mat{P}(\prmtrVec)$. 
Using an error bound from Rayleigh-Ritz approximation,
we can indeed establish 
\[
\begin{aligned}
\|\mat{Q}_{\rm{rb}}(\prmtrVec)-\mat{P}(\prmtrVec)\|
\lesssim
\|{\RBproj}_{\Ord}^{\perp}\mat{P}(\prmtrVec)\|
\leq
\|\mat{P}^{[\Ord]}(\prmtrVec)-\mat{P}(\prmtrVec)\| 
=
\bigO\bigl(|\prmtrVec-\prmtrVec_0|^{\Ord+1}\bigr).
\end{aligned}
\]
The corresponding Ritz-value sum
$\tr(\mat{A}(\prmtrVec)\mat{Q}_{\rm{rb}}(\prmtrVec))$
approximates
\(\tr(\mat A(\prmtrVec)\mat P(\prmtrVec))\)
with the squared order $2\Ord+2$.
Regarding the assembly of $\RBspace_\Ord(\prmtrVec_0)$,
the scaled derivatives of $\mat{P}(\prmtrVec)$
need \emph{not} be formed: 
the Taylor-RB space $\RBspace_\Ord(\prmtrVec_0)$ 
can be assembled recursively from thin
invariant-subspace basis matrices, 
requiring only the scaled derivatives of the system, 
and the computed eigenvalues and eigenvectors at $\prmtrVec_0$.
}

The motivation of our study is twofold.
Taylor reduced bases are established in model order reduction (MOR) for
parametric linear source problems, while closely related
derivative-based reduced spaces have emerged independently in physics.
\rev{Our goal is to develop a rigorous mathematical framework 
for Taylor-RBM in multivariate parametric EVPs, 
including possibly multiple eigenvalues. 
While conventional RBM approaches often aim at approximating solutions over a broad parameter domain, 
here we focus on constructing a reduced basis that guarantees high-order local convergence near a fixed reference parameter.}
The main applications we have in mind arise in computational physics
and chemistry, although the method also applies to engineering
problems.
Before stating our contributions, 
we summarize the related literature.

\subsection{Literature review and state-of-the-art}

MOR is a vast research area with numerous applications in science and engineering \cite{BTGetal20total,Ant05, SVR08}. 
The general idea of MOR for parametric systems is that given a large parametric full-order model (FOM) that is computationally expensive to solve, 
one may construct in the \emph{offline phase} a reliable and efficient reduced order model (ROM) that approximates the FOM well in a certain sense,
and then use the ROM in the \emph{online phase} for rapid repeated evaluations \cite{BOPRU17}.
Among the various MOR techniques, we focus on \emph{projection-based MOR}, 
which can be divided into two categories: 
\emph{System-theoretic} MOR addresses more on control problems, see \cite{BenGW15} for a detailed survey;
the \emph{reduced basis method} (RBM) focuses on reducing the dimensionality of the state space through Galerkin-projection \cite{BGTetal20}.

RBM is a class of projection-based MOR techniques, 
proposed originally for solving parametric source problems, 
i.e., parametric partial differential equations (PDEs) \cite{NooP80,PorL87, Rhe93}.
The key idea is to use information of the computed true solutions at selected parameter values, 
also called \emph{snapshot solutions}, 
to construct a low-dimensional RB space, such that the Galerkin projection onto this space provides a good approximation of the solution set. 
Following \cite{Haa17}, 
there are two main ideas of constructing the RB space: 
\begin{itemize}
\item 
\emph{Lagrange-RBM} constructs the RB space by collecting snapshot solutions at selected parameters. 
There are two common strategies: 
The \emph{proper orthogonal decomposition} (POD) refers to computing snapshot solutions of many preselected parameters and then performing a singular value decomposition (SVD) of the resulting large snapshot matrix. 
Alternatively, one can use a \emph{greedy method} to iteratively select the most ``representative'' snapshots from a training set of parameters based on a-posteriori error estimators. 
In both cases, solid mathematical framework and convergence analysis have been established for linear, coercive continuous and affinely decomposable problems \cite{PruRVP02, HesRS16, BMPetal12}.
\item 
\emph{Taylor-RBM} constructs the RB space by collecting the Taylor coefficients of the parametric solution at a given parameter up to a certain order, 
provided that the parametric solution is differentiable w.r.t.\! the parameter. 
Convergence analysis of Taylor-RBM for parametric source problems has been conducted in \cite{Fin83}.  
\end{itemize}

Parametric eigenvalue problems (pEVPs) arise naturally in numerical linear algebra \cite{KreV14, SirK16, KanMMM18, MMMV17, PraB24}, 
in infinite dimensional setting (especially within finite element approximations of parametric PDE EVP) \cite{MacMOPR00, AndS12, AlgBB25, FMPV16, HWD17},
in uncertainty quantification \cite{CJT21, ShiA72, DoeE24, BenR88},
as well as in quantum physics \cite{BHWRS23, HSWR22, FHIetal18, MDFGZ22, DEFK23}.
Compared to parametric source problems, 
pEVPs demonstrate a different nature.
The main challenges include 
(i) \rev{the loss of the Galerkin orthogonality, 
in the sense that the Rayleith-Ritz method does not enforce orthogonality 
of the difference between a true and an approximate eigenvector
 to the subspace, 
}
(ii) the absence of a reliable residual based error estimator, and 
(iii) the complicated regularity behavior of the eigenvalues and eigenvectors w.r.t.\! the parameter, especially when eigenvalue crossings occur; 
see \cite[Sec.~2.1]{MSZ25} for more details. 

The first two challenges are intrinsic problem of subspace methods for eigenvalue problems \cite{Bof10,Gou12, BabO91, Wei74, Par98, KMX13}.
These challenges obstruct direct extensions of the existing RBM theory for parametric source problems to pEVPs.
Nonetheless, certified greedy approximation procedures for the smallest eigenvalue and its corresponding eigenspace have been proposed in \cite{SirK16, MSZ25}. 

\rev{In the case $\prmtrSet\subseteq\bb{R}$, challenge (iii) is
addressed by the classical analytic perturbation theory of linear
operators \cite{Rel54,Kato76}: the eigenvalues and corresponding
eigenvectors of a real-analytic Hermitian family can be parameterized as
real-analytic functions.}
Exploiting this analyticity, 
the quantum physics community has proposed the \emph{eigenvector continuation} (EC) method,
which constructs an RB space by collecting derivatives of the eigenvectors \cite{FHIetal18, DEFK23}. 
An involved mathematical convergence analysis of this univariate case has been recently provided in \cite{GarS24}.   

The regularity of eigenvalues and eigenvectors in the multivariate parameter case is more complicated,
essentially because \emph{individual} eigenvalues are generally only continuous and the \emph{individual} eigenvectors may even be discontinuous.
This phenomenon is well-illustrated by the classical example of the $2 \times 2$ matrix family
\begin{equation} \label{eq:kato_example}
\begin{bmatrix}
\prmtr_1 & \prmtr_2 \\
\prmtr_2 & -\prmtr_1
\end{bmatrix} 
\qquad \text{for }  
(\prmtr_1, \prmtr_2)^{\!\top} \in \bb{R}^2
\end{equation}
from \cite[Exp.~5.12 in Chap.~2]{Kato76},
which has eigenvalues $\pm {(\prmtr_1^2 + \prmtr_2^2)^{1/2}}$ that are continuous but not differentiable at the origin,
and the corresponding eigenvectors \emph{cannot} be defined continuously at the origin.
The classical works \cite{Kato76,Bau84} and recent works
\cite{AndS12,GSH23} suggest considering the total eigenprojector
associated with a cluster of eigenvalues instead of individual
eigenvalues and eigenvectors. 
\rev{If the cluster remains separated 
from the rest of the spectrum, 
this projector inherits 
the finite differentiability of the matrix family, 
and is real-analytic whenever
the matrix family is real-analytic.}
To the best of our knowledge, a multivariate Taylor-RBM based on these
total eigenprojector derivatives has not previously been analyzed,
nor has an efficient recursive assembly procedure been provided.


\subsection{Contributions}
{This article provides the following contributions:}
\begin{itemize}
\item 
For the Rayleigh--Ritz method, we derive an a priori eigenspace
error bound in terms of the best-approximation error, the spectral gap
and the residual of the entire approximation space, cf.~\Cref{thm:Cea_EVP}. 
We also establish an exact trace identity that
reveals the squaring effect for the Ritz-value sum in
\Cref{lem:energy_difference_general}.
\item 
\rev{Under local \(\mathrm C^{\Ord +1}\)-regularity, 
we develop a finite-order perturbation analysis 
of total eigenprojectors. 
We derive explicit formulae, 
study algebraic properties of their scaled derivatives, 
construct a rank-\(M\) orthogonal approximation \(\mat Q_{\rm pt}\) 
from the finite Taylor polynomial by spectral projection, 
and prove eigenprojector and total Ritz-value error orders 
\(\Ord+1\) and \(2\Ord+2\), 
respectively, cf.~\Cref{sec:analyticity}. 
For real-analytic matrix families, 
the construction applies to every finite order.}
\item
We introduce the multivariate Taylor-RBM in
\Cref{def:taylor_rbm}, prove its well-posedness and the corresponding
eigenprojector and Ritz-value convergence rates in
\Cref{thm:Taylor_rbm_projector_error,thm:energy_difference_rb}, and
derive an efficient recursive assembly procedure in
\Cref{alg:assembly_taylor_rbm_space}. 
\rev{We further discuss its offline cost 
and practical online realization.}
\end{itemize}

\subsection{Organisation of the manuscript}
\Cref{sec:rayleigh_ritz} develops the Rayleigh--Ritz results used
in the subsequent analysis.
\rev{\Cref{sec:analyticity} establishes the finite-order regularity and
derivative structure of total eigenprojectors and analyzes a spectrally
corrected finite Taylor approximation.}
\Cref{sec:Taylor_rbm} introduces the Taylor-RBM, proves its convergence
and develops an efficient recursive assembly procedure 
\rev{together with offline and online implementation considerations.}
\Cref{sec:num_exp} investigates the dimension growth of the Taylor-RB
space, validates the error estimates and compares the method with
classical perturbation theory.
Finally, \Cref{sec:conclusion} concludes the study.

\subsection{General notations}
$\bb{N}$ denotes natural numbers without $0$, and $\bb{N}_0 \coloneqq \{0\} \cup \bb{N}$. 
Bold letters, such as $\prmtrVec$, $\vec{y}$ or $\mat{M}$, indicate quantities that are in general vectors or matrices.
$\mat{I}_n$ denotes the identity matrix in $\bb{C}^{n \times n}$ for any $n\in \bb{N}$,
and for convenience, we denote $\mat{I} \coloneqq \mat{I}_\fdim$.
The $j$-th canonical unit vector in $\bb{R}^\Pdim$, which has $1$ in the $j$-th position and $0$ elsewhere, is denoted by $\vec{e}_j$.

For parameters $\prmtrVec = (\prmtr_1, \ldots, \prmtr_\Pdim) \in \bb{R}^\Pdim$, we denote its Euclidean norm by $|\prmtrVec| \coloneqq \big({\prmtr_1^2 + \cdots + \prmtr_\Pdim^2}\big)^{1/2}$.
{
For a multi-index
$\multiidx=(\beta_1,\ldots,\beta_{\Pdim})\in\bb{N}_0^{\Pdim}$ and a
parameter $\prmtrVec=(\prmtr_1,\ldots,\prmtr_{\Pdim})$, we write
\[
|\multiidx|\coloneqq\beta_1+\cdots+\beta_{\Pdim},
\qquad
\multiidx!\coloneqq\beta_1!\cdots\beta_{\Pdim}!,
\qquad
\prmtrVec^{\multiidx}
\coloneqq\prmtr_1^{\beta_1}\cdots
\prmtr_{\Pdim}^{\beta_{\Pdim}}.
\]
The mutual notation $|\cdot|$ for the Euclidean norm of a parameter and the length of
a multi-index will not cause confusion.
}

{
For $|\multiidx|\leq\Ord+1$, the scaled derivatives of $\mat{A}$ at
$\prmtrVec_0$ are denoted by
\begin{equation}\label{eq:scaled_derivatives_A}
\mat{A}^{\!(\multiidx)}
\coloneqq
\frac{1}{\multiidx!}\partial^{\multiidx}\mat{A}(\prmtrVec_0)
=
\frac{1}{\beta_1!\cdots\beta_{\Pdim}!}
\frac{\partial^{|\multiidx|}\mat{A}}
{\partial\prmtr_1^{\beta_1}\cdots
\partial\prmtr_{\Pdim}^{\beta_{\Pdim}}}(\prmtrVec_0).
\end{equation}
The scaled derivatives of other parameter-dependent matrix- or vector-valued
functions are denoted analogously.}

The spectral norm (or induced Euclidean norm) of a matrix $\mat{M} \in \bb{C}^{\fdim \times \fdim}$ is denoted by $\|\mat{M}\|$,
and the spectrum of $\mat{M}$ is denoted by $\sigma(\mat{M})$.
Given an orthogonal projector $\RBproj = {\RBproj}^2 = {\RBproj}^*$,
we denote by ${\RBproj}^\perp \coloneqq \mat{I} - {\RBproj}$ its complementary projector.
We denote by $\mat{M}\!\!\upharpoonright_{\RBproj} : \ran{\RBproj} \to \ran{\RBproj}, \vec{y} \mapsto {\RBproj}\mat{M}\vec{y}$ the restriction of ${\RBproj}\mat{M}$ to $\ran {\RBproj}$, 
which is also known as the \emph{compression} of $\mat{M}$ to $\ran {\RBproj}$.
    
To avoid clumsy brackets in formulae, 
we will mostly neglect brackets when taking traces, ranks, kernels and ranges of matrix sums or products from now on, 
e.g., we mostly write $\tr \, \mat{M}_1 \mat{M}_2$ instead of $\tr (\mat{M}_1 \mat{M}_2)$, or $\rank \, \mat{M}_1 - \mat{M}_2$ instead of $\rank (\mat{M}_1 - \mat{M}_2)$.

\section{Selected aspects of the Rayleigh-Ritz method for eigenspace approximation}\label{sec:rayleigh_ritz}

Throughout this section, we consider a (non-parametric) Hermitian matrix $\mat{A} \in \bb{C}^{\fdim \times \fdim}$.
Let $\mat{P} \in \bb{C}^{\fdim \times \fdim}$ be an orthogonal projector of rank $M$
{that is a sum of spectral projectors of $\mat{A}$}.
{Consequently, $\mat{P}$ commutes with $\mat{A}$, 
and the spectra of
$\mat{A}\!\!\upharpoonright_\mat{P}$ and
$\mat{A}\!\!\upharpoonright_{\mat{P}^{\perp}}$ are disjoint.}

Given a subspace $\cal{V} \subseteq \bb{C}^{\fdim}$ with associated orthogonal projector ${\RBproj}$ and $\rdim \coloneqq \dim \cal{V} = \rank {\RBproj} \ll \fdim$ as well as $\rdim \geq M$, 
the Rayleigh-Ritz method, as mentioned in \Cref{sec:intro},
refers to approximating the eigenvalues and eigenspaces of $\mat{A}$ by those of its compression $\mat{A}\!\!\upharpoonright_{\RBproj}$.
In this context, 
the eigenspaces of $\mat{A}\!\!\upharpoonright_{\RBproj}$ are viewed as subspaces of $\cal{V} \subseteq \bb{C}^\fdim$, 
and so its eigenprojectors can be naturally compared with $\mat{P}$. 
Let $\mat{Q}\in \bb{C}^{\fdim \times \fdim}$ be an orthogonal projector of rank $M$, 
such that $\ran \, \mat{Q} \subseteq \cal{V}$ and $\mat{Q}$ commutes with ${\RBproj} \mat{A} {\RBproj}$, 
i.e., $\ran \mat{Q}$ is an invariant subspace of $\mat{A}\!\!\upharpoonright_{\RBproj}$.
We are interested in how well $\ran \mat{Q}$ approximates $\ran \mat{P}$,
or equivalently, how well $\mat{Q}$ approximates $\mat{P}$.

The following example illustrates some issues of the subspace method in approximating eigenspaces.
\begin{example}
[Pathological case of Rayleigh-Ritz method for eigenspace approximation]\label{ex:rayleigh_ritz_pathological}
Consider $\mat{A} \coloneqq \rm{diag}(1,2,2,3)$ and $\cal{V} \coloneqq \ran {\RBproj} \coloneqq \rm{span}\{\vec{e}_1 + \vec{e}_4, \vec{e}_2, \vec{e}_3\}$,
where $\vec{e}_j$ denotes the $j$-th standard basis vector.
Then, we have $\mat{A}\!\!\upharpoonright_{\RBproj} = \rm{diag}(2,2,2)$. 
Let us discuss two cases of $\mat{P}$:
\begin{enumerate}
\item \label{ex:pathological_item1}
If $\mat{P} = \vec{e}_1\vec{e}_1^*$, i.e. $M=1$, 
the choice of $\mat{Q}$ is not unique, 
as any 1-dimensional subspace of~$\cal{V}$ is an invariant subspace of $\mat{A}\!\!\upharpoonright_{\RBproj}$, 
and all approximations are bad, 
since $\|{\RBproj}^\perp \mat{P}\| = 2^{-1/2}$. 
\item 
If $\mat{P} = \vec{e}_2\vec{e}_2^* + \vec{e}_3\vec{e}_3^*$, i.e. $M=2$,
the choice of $\mat{Q}$ is also not unique, 
as any 2-dimensional subspace of $\cal{V}$ is an invariant subspace of $\mat{A}\!\!\upharpoonright_{\RBproj}$.
We also see that any choice of $\mat{Q}$ satisfies $\tr \mat{AQ} = \tr \mat{AP} =4$, 
{since every Ritz value in $\cal{V}$ equals $2$}.
In other words, the subspace~$\cal{V}$ introduces additional {``spurious''} eigenvalues which coincide with the target eigenvalues. 
However, we observe $\ran \mat{P} \subseteq \cal{V}$, so the best choice of $\mat{Q}$ is certainly $\mat{Q} = \mat{P}$.
\end{enumerate}
\end{example}

For our problem of interest and for completeness, 
we comment on the case where $\mat{P}$ is characterized variationally.
\begin{remark}[Variational approximation]
{In the variational setting relevant to the Taylor-RBM, $\mat{P}$ is the
projector onto the eigenspaces associated with the $M$ smallest eigenvalues,
and $\mat{Q}$ is chosen according to
\begin{align} \label{eq:ky_fan_min_principle_subspace}
\mat{P} 
\,=\!
\argmin_{\mat{M}\in\operatorname{Gr}(M,\fdim)}
\tr \, \mat{A} \mat{M}
\qquad \text{and} \qquad
\mat{Q} 
\,\in\!
\argmin_{\substack{
\mat{M}\in\operatorname{Gr}(M,\fdim)\\
\ran \mat{M} \subseteq \ran {\RBproj}}}
\tr \, \mat{A}  \mat{M},
\end{align}
respectively.
As mentioned, 
the first minimization is the Ky Fan minimum principle \cite{Fan49};
see also \cite[Thm.~1 in Sec.~10.2]{BirS87} and
\cite[Prp.~1.3.4]{Tao12}.
Its solution $\mat{P}$ is unique, whereas the restricted minimization
need not have a unique solution; see
\Cref{ex:rayleigh_ritz_pathological}~\ref{ex:pathological_item1}.
If $\ran\mat{P}\subseteq\ran{\RBproj}$, however, $\mat{P}$ is feasible
for the restricted problem, and uniqueness of the unrestricted minimizer
implies $\mat{Q}=\mat{P}$.
The second minimization is the non-parametric form of \eqref{eq:reduced_sol_intro},
and this exactness property will be used in \Cref{sec:Taylor_rbm}.
The analysis below does not otherwise require the variational setting
\eqref{eq:ky_fan_min_principle_subspace}.
}
\end{remark}


It is natural to measure the approximation quality of $\mat{Q}$ to $\mat{P}$ by $\| \mat{Q} - \mat{P} \|$. 
The following result provides a flexibility in measuring this approximation error under a rank assumption.

\begin{lemma}[Equivalence of subspace approximation error]\label{lem:proj_diff_equiv}
Let $\mat{P}$ {and} $\mat{Q}$ be two orthogonal projectors of the same rank.
{Then, it holds that}
{$\| \mat{Q} - \mat{P} \| = \| \mat{Q}^\perp \mat{P} \| = \| \mat{P}^{\perp} \mat{Q} \|.$}
\end{lemma}
Note that in general $\mat{Q} - \mat{P} \neq \mat{Q}^\perp \mat{P} \neq \mat{P}^{\perp} \mat{Q}$ even though their norms are equal.

An operator theoretic proof 
can be found in \cite[Thm.~6.34 in {Chap.}~1]{Kato76}; 
a more geometric treatment linking to the principal angles 
between subspaces can be found 
in \cite[Sec.~2]{Wed83}.

Next, we relate the approximation error $\| \mat{Q}^\perp \mat{P} \|$
and the projection error $\| {\RBproj}^\perp \mat{P} \|$.
The following theorem has a similar form to the quasi-optimal property of the Petrov-Galerkin solution in finite element methods \cite[Lem.~3.7]{Bra07},
and is {a specialization of \cite[Thm.~2]{Ovt06I}}. 
Our formulation emphasizes the deviation of the entire subspace being an invariant subspace, 
and avoids individual eigenvectors. 
This is more suitable for the parametric setting, as will be discussed later in \Cref{rem:individual_eigenvalues}.

\begin{theorem}[Bounding eigenspace approximation error by eigenspace projection error]\label{thm:Cea_EVP}
{Assume that $\rdim>M$ and 
we have a positive spectral gap in the sense that}
\rev{\begin{equation*}
\begin{aligned}
\gamma 
\coloneqq \min \bigl\{ |\xi-\hat{\lambda}|:
\ \xi\in
\bigl[\min\{\sigma(\mat{A}\!\!\upharpoonright_{\mat{P}})\},
\max\{\sigma(\mat{A}\!\!\upharpoonright_{\mat{P}})\}\bigr], 
\ \hat{\lambda}\in
\sigma(\mat{A}\!\!\upharpoonright_{\mat{Q}^{\perp}{\RBproj}})
\bigr\}>0.
\end{aligned}
\end{equation*}
}
Then, it holds that 
\begin{align*}
\| \mat{Q}^\perp \mat{P}  \|^2
\; \leq \;
\bigl(1+\gamma^{-2}\rm{Res}^2(\mat{A},\RBproj)\bigr)
\| {\RBproj}^\perp \mat{P} \|^2,
\end{align*}
where $\rm{Res}^2(\mat{A}, \RBproj) \coloneqq \|{\RBproj} \mat{A} {\RBproj}^\perp\|^2$ 
is the squared residual of the entire subspace $\cal{V} = \ran {\RBproj}$ 
{with respect to} $\mat{A}$.
\end{theorem}

This theorem states that the error of the eigenspace approximation by the subspace method,
{under the stated spectral separation},
can be controlled by how well the target eigenspace is approximated by the subspace $\cal{V}$,
up to a factor depending on the spectral gap and on {how close $\cal{V}$ is to being} an invariant subspace of $\mat{A}$.

\begin{proof}
First, we note that $\mat{Q}^\perp {\RBproj}$ is also an orthogonal projector satisfying ${\RBproj} = \mat{Q} + \mat{Q}^\perp {\RBproj}$ by \cite[Thm.~3 in Sec.~2.8]{BirS87}. 
Since $\ran \mat{Q} \subseteq \ran {\RBproj}$, we have $\ran {\RBproj}^\perp \subseteq \ran \mat{Q}^\perp$, and so $\mat{Q}^\perp {\RBproj}^\perp = {\RBproj}^\perp$. 
Using these facts and $\mat{I} = {\RBproj}^\perp + {\RBproj}$, we obtain 
$\mat{Q}^\perp \mat{P} 
= \mat{Q}^\perp ({\RBproj}^\perp + {\RBproj}) \mat{P}
= {\RBproj}^\perp \mat{P} + \mat{Q}^\perp {\RBproj} \mat{P}.
$
Since ${\RBproj}^\perp \mat{P}$ and $\mat{Q}^\perp {\RBproj} \mat{P}$ have orthogonal ranges,
we deduce
\begin{align} \label{eq:cea_evp_proof_1}
\| \mat{Q}^\perp \mat{P} \|^2
{\leq} 
\| {\RBproj}^\perp \mat{P} \|^2 + \| \mat{Q}^\perp {\RBproj} \mat{P} \|^2.
\end{align}
To bound $\| \mat{Q}^\perp {\RBproj} \mat{P} \|$, 
we rewrite it as the solution of a {Sylvester equation}.
{Multiplying both sides of $\mat{P} \mat{A} = \mat{A} \mat{P}$ from the left by $\mat{Q}^\perp {\RBproj}$} and using again $\mat{I} = {\RBproj} + {\RBproj}^\perp$, we get
\begin{align} \label{eq:cea_evp_proof_2}
\mat{Q}^\perp {\RBproj} \mat{P} \mat{A} 
\,=\,
\mat{Q}^\perp {\RBproj} \mat{A} \mat{P} 
\,=\,
\mat{Q}^\perp {\RBproj} \mat{A} ({\RBproj} + {\RBproj}^\perp) \mat{P} 
\,=\,
\mat{Q}^\perp {\RBproj} \mat{A} {\RBproj} \mat{P} 
+
\mat{Q}^\perp {\RBproj} \mat{A} {\RBproj}^\perp \mat{P}.
\end{align}
Since $\mat{P} = \mat{P}^2$ commutes with $\mat{A}$, 
the {leftmost} term in \eqref{eq:cea_evp_proof_2} can be rewritten as
\begin{align} \label{eq:cea_evp_proof_3}
\mat{Q}^\perp {\RBproj} \mat{P} \mat{A} 
\,=\,
\mat{Q}^\perp {\RBproj}  \mat{P}^2 \mat{A} 
\,=\,
\mat{Q}^\perp {\RBproj} \mat{P} \mat{A} \mat{P} 
\,=\,
\mat{Q}^\perp {\RBproj} \mat{P} \, \mat{A}\!\!\upharpoonright_\mat{P}.
\end{align}
By construction, $\mat{Q}$ commutes with ${\RBproj} \mat{A} {\RBproj}$. 
Since ${\RBproj} = \mat{Q} + \mat{Q}^\perp{\RBproj}$, 
we deduce that $\mat{Q}^\perp {\RBproj} = (\mat{Q}^\perp {\RBproj})^2$ also commutes with ${\RBproj} \mat{A} {\RBproj}$. 
Using this fact and $\mat{Q}^\perp, {\RBproj}$ being projectors, 
we rewrite the first term on the {right-hand} side of \eqref{eq:cea_evp_proof_2} as
\begin{align} \label{eq:cea_evp_proof_4}
\mat{Q}^\perp {\RBproj} \mat{A} {\RBproj} \mat{P} 
\,=\,
(\mat{Q}^\perp {\RBproj})^2 {\RBproj} \mat{A} {\RBproj}^2 \mat{P} 
\, = \, 
\mat{Q}^\perp {\RBproj} \mat{A} {\RBproj} \mat{Q}^\perp  {\RBproj} \mat{P} 
\,=\,
 \mat{A}\!\!\upharpoonright_{\mat{Q}^\perp {\RBproj}} \, \mat{Q}^\perp {\RBproj} \mat{P}.
\end{align}
{Substituting} \eqref{eq:cea_evp_proof_3} and \eqref{eq:cea_evp_proof_4} 
into the left- and {right-hand} sides of \eqref{eq:cea_evp_proof_2} and rearranging the terms,
we arrive at the {Sylvester equation}
\begin{align*}
\mat{A}\!\!\upharpoonright_{\mat{Q}^\perp  {\RBproj}} \, \mat{Q}^\perp {\RBproj} \mat{P}
- \mat{Q}^\perp {\RBproj} \mat{P} \, \mat{A}\!\!\upharpoonright_\mat{P} 
\,=\,
- \mat{Q}^\perp {\RBproj} \mat{A} {\RBproj}^\perp \mat{P}.
\end{align*}
\rev{Since
$\sigma(\mat{A}\!\!\upharpoonright_{\mat{Q}^{\perp}{\RBproj}})$ has
distance at least $\gamma$ from 
$\bigl[\min\{\sigma(\mat{A}\!\!\upharpoonright_{\mat{P}})\},
\max\{\sigma(\mat{A}\!\!\upharpoonright_{\mat{P}})\}\bigr]$, 
the Sylvester estimate in \cite[Sec.~10]{BhaR97} gives}
\begin{align*}
\| \mat{Q}^\perp {\RBproj} \mat{P} \|
\; \leq \;
{\gamma}^{-1}{\| \mat{Q}^\perp {\RBproj} \mat{A} {\RBproj}^\perp \mat{P} \|}
\; \leq \;
{\gamma}^{-1}{\| {\RBproj} \mat{A} {\RBproj}^\perp \| \, \| {\RBproj}^\perp \mat{P} \|},
\end{align*}
where the second step follows from the {submultiplicativity} 
of the spectral norm combined with the facts that $\|\mat{Q}^\perp\|=1$ and ${\RBproj}^\perp = ({\RBproj}^\perp)^2$.
Inserting this bound into \eqref{eq:cea_evp_proof_1} completes the proof. 
\end{proof}

The last issue in this section concerns the relation between the approximation error of the eigenspaces
and the difference of the corresponding Ritz-value sums.
{The error in the corresponding Ritz-value sum is commonly known to be of
\emph{squared} order relative to the eigenspace approximation error.}
However, the formulation of this squaring effect in the literature usually relies on $\|\mat{Q}^\perp\mat{P}\|$ and involves the spread of $\mat{A}$ or some auxiliary operators, 
see{, e.g.,} \cite[Thm.~3.2]{Sun91} or \cite[Lem.~8.2]{GarS24}.  
Here, we present a concise trace identity, which reveals the squaring effect directly. 

\begin{lemma}[{Squaring effect for the Ritz-value sum}]\label{lem:energy_difference_general}
For ${{\projDiff}} \coloneqq \mat{Q} - \mat{P}$,
it holds that 
$$
\tr \,  \mat{A} {{\projDiff}} 
\,=\, 
\tr \, \mat{A} (\mat{P}^{\perp} -\mat{P}) {{\projDiff}}^{\!2}
\,=\, 
\tr \, \mat{A} (\mat{I} - 2\mat{P}) {{\projDiff}}^{\!2}.
$$
\end{lemma}

Note that $\mat{P}^{\perp} -\mat{P} = \mat{I} - 2\mat{P}$ is the generalized Householder reflector,
which reflects across $\ran\mat{P}^{\perp}$. 

\begin{proof}
{Since $\mat{I} = \mat{P} + \mat{P}^{\perp}$ and $\mat{P}$ and $\mat{P}^{\perp}$ are orthogonal projectors commuting with $\mat{A}$,}
we know $\mat{P}\mat{A}\mat{P}^{\perp} = \mat{P}^{\perp}\mat{A}\mat{P} = \mat{0}$. 
Using these facts and the linearity as well as cyclicity of the trace, we obtain 
\begin{align} \label{eq:energy_diff_first_manipulation}
\tr \, \mat{A} {{\projDiff}} 
\, = \, 
\tr \, (\mat{P} + \mat{P}^{\perp}) \mat{A} (\mat{P} + \mat{P}^{\!\perp}) {{\projDiff}}  
\, = \, 
\tr \, \mat{A} \mat{P} {{\projDiff}} \, \mat{P} 
+ \tr \, \mat{A} \mat{P}^{\perp} {{\projDiff}} \, \mat{P}^{\perp}.
\end{align}
By the definition of ${{\projDiff}}$, we have $\mat{Q} = \mat{P} + {{\projDiff}}$.
Since $\mat{Q}^2 = \mat{Q}$, we deduce $(\mat{P} + {{\projDiff}})^2 = \mat{P} + {{\projDiff}}$.
Expanding $(\mat{P} + {{\projDiff}})^2$, 
using the idempotent property of $\mat{P}$, multiplying both sides by $\mat{P}$ and $\mat{P}^{\perp}$ from the left and right, respectively,
we get 
\begin{align*}
\mat{P} {{\projDiff}} \, \mat{P} 
= \; - \mat{P} {{\projDiff}}^{\!2} \, \mat{P}
\qquad \text{and} \qquad
\mat{P}^{\perp} {{\projDiff}}  \, \mat{P}^{\perp} 
= \;  \mat{P}^{\perp} {{\projDiff}}^{\!2} \, \mat{P}^{\perp}.
\end{align*}
Applying these two relations to \eqref{eq:energy_diff_first_manipulation}, 
we arrive at
\begin{align*}
\tr \, \mat{A} {{\projDiff}} 
\,= \, 
- \tr \, \mat{A} \mat{P} {{\projDiff}}^{\!2} \, \mat{P} 
+ \tr \, \mat{A} \mat{P}^{\perp} {{\projDiff}}^{\!2} \, \mat{P}^{\perp} 
\,= \, 
- \tr \, \mat{P} \mat{A} \mat{P} {{\projDiff}}^{\!2}  
+ \tr \, \mat{P}^{\perp}  \mat{A} \mat{P}^{\perp} {{\projDiff}}^{\!2} 
\,= \, 
\tr \,  \mat{A} (\mat{P}^{\perp} -\mat{P}) {{\projDiff}}^{\!2},
\end{align*}
where the second step follows from the cyclicity of the trace 
and the third step relies on the fact that $\mat{P}$ and $\mat{P}^{\perp}$ are orthogonal projectors commuting with $\mat{A}$.
\end{proof}

\section{{Regularity and finite-order approximation of total eigenprojectors}}\label{sec:analyticity}

{This section develops a finite-order perturbation analysis of the total
eigenprojector associated with a possibly multiple eigenvalue.
After recalling the unperturbed reduced resolvent in
\Cref{subsec:unperturbed}, we derive and analyze the scaled projector
derivatives in
\Cref{subsec:analyticity_total_proj,subsec:prop_proj_series}.
Finally, \Cref{subsec:truncated_projector_series} constructs a spectrally
corrected finite Taylor approximation and estimates its subspace and total
Ritz-value errors.}

\subsection{Useful properties of the unperturbed case}\label{subsec:unperturbed}
Consider an arbitrary $\prmtrVec_0 \in \prmtrSet$ and denote for brevity $\mat{A}_0 \coloneqq \mat{A}(\prmtrVec_0)$.
Let us recall from \cite[Chap.1~§5]{Kato76} some useful results regarding $\mat{A}_0$ and its resolvent $\mat{R}_0(\zeta) \coloneqq (\mat{A}_0 - \zeta \, \mat{I})^{-1}$ 
which is defined for $\zeta \notin \sigma(\mat{A}_0)$. 

Let $\EigVal_0 \in \{\EigVal_1,\dots, \EigVal_K\}$ be an eigenvalue of $\mat{A}_0$ 
with multiplicity $m_0 \in \bb{N}$,
cf.~\eqref{eq:target_eigenvalues}.
We denote its corresponding orthogonal eigenprojector by~$\mat{P}_{\!0}$, 
i.e., $\mat{P}_{\!0} \in \{\mat{P}_{\!1}(\prmtrVec_0),\dots, \mat{P}_{\!\!K}(\prmtrVec_0)\}$ 
from \eqref{eq:target_projector}, satisfying
$\mat{P}_{\!0}^2 = \mat{P}_{\!0} = \mat{P}_{\!0}^*$, 
$\rank (\mat{P}_{\!0}) = m_0$ and  
$\mat{A}_0 \mat{P}_{\!0} =  \mat{P}_{\!0} \mat{A}_0  = \EigVal_0 \mat{P}_{\!0}$.

{Following \cite[Chap.1~§5.3]{Kato76},}
we expand the meromorphic function $\mat{R}_0(\zeta)$ in a Laurent-series around its singularity~$\EigVal_0$, 
and obtain 
\begin{align} \label{eq:laurent_series}
\mat{R}_0(\zeta) 
&= 
\rev{- (\zeta - \EigVal_0)^{-1} \mat{P}_{\!0} +} 
\sum_{t=\rev{0}}^{\infty} (\zeta - \EigVal_0)^t \, \mat{S}_0^{\rev{t+1}}, 
\end{align}
where the coefficient matrices $\mat{P}_{\!0}$ and $\mat{S}_0$ 
are calculated via the contour integrals
\begin{equation} \label{eq:laurent_coefficients}
\rev{\mat{P}_{\!0} =  
-\frac{1}{2 \pi \i} 
\oint_{\Gamma} 
\mat{R}_0(\zeta) \, \d{\zeta} 
\qquad \text{and} \qquad
\mat{S}_0 =  
\frac{1}{2 \pi \i} 
\oint_{\Gamma} 
(\zeta - \EigVal_0)^{-1} \, 
\mat{R}_0(\zeta) \, \d{\zeta}, }
\end{equation}
with $\Gamma$ being a {counterclockwise oriented} closed contour in the complex plane 
encircling only one eigenvalue $\EigVal_0$ of $\mat{A}_0$.
The matrix $\mat{S}_0$ is called the \emph{reduced resolvent} of $\mat{A}_0$ at $\EigVal_0$, 
which is characterized by 
\begin{align}
\mat{S}_0 \, \mat{P}_{\!0} &= \mat{P}_{\!0} \, \mat{S}_0 = \mat{0}, \label{eq:reduced_resolvent_orthogonality} \\
(\mat{A}_0 - \EigVal_0 \, \mat{I}) \, \mat{S}_0 &= \mat{S}_0 \, (\mat{A}_0 - \EigVal_0 \, \mat{I}) = \mat{I} - \mat{P}_{\!0}. \label{eq:reduced_resolvent}
\end{align}
\rev{In other words, 
$\mat{S}_0$ is the Moore-Penrose inverse of $\mat{A}_0 - \EigVal_0\mat{I}$}.
Note that the {principal part of} the Laurent-series of $\mat{R}_0(\zeta)$ terminates at $t=-1$, 
since $\mat{A}_0$ is Hermitian \cite[Chap.1~§5.4]{Kato76}. 
For detailed derivations and proofs of the above properties,
we refer the reader to \cite[Chap.1~§5]{Kato76}.

\subsection{\rev{Regularity and finite derivatives of the total eigenprojector}} \label{subsec:analyticity_total_proj}
{For $\prmtrVec$ sufficiently close to $\prmtrVec_0$, 
the contour $\Gamma$ from \Cref{subsec:unperturbed} lies 
in the resolvent set of $\mat{A}(\prmtrVec)$ 
and encloses precisely $m_0$ eigenvalues, 
counted \emph{with} multiplicity.
We define
\begin{equation}\label{eq:projector_contour_integral}
\mat{R}(\zeta;\prmtrVec)
\coloneqq
(\mat{A}(\prmtrVec)-\zeta\mat{I})^{-1}
\qquad \text{and} \qquad
\mat{P}_{\!0}(\prmtrVec)
\coloneqq
-\frac{1}{2\pi\i}\oint_\Gamma
\mat{R}(\zeta;\prmtrVec)\,\d{\zeta}.
\end{equation}
The smoothness of matrix inversion and differentiation under the contour integral imply that $\mat{P}_{\!0}$ is 
of class $\mathrm{C}^{\Ord +1}$ near $\prmtrVec_0$, 
cf.~\cite[Chap.2~§5.4]{Kato76}.
Moreover, $\mat{P}_{\!0}(\prmtrVec)$ is the rank-$m_0$ orthogonal projector onto the invariant subspace associated with the eigenvalues enclosed by $\Gamma$.

For $|\multiidx|\leq\Ord+1$, differentiation under the contour integral
gives
\begin{equation}\label{eq:scaled_derivatives_resolvent_projector}
\mat{P}_{\!0}^{(\multiidx)}
=
-\frac{1}{2\pi\i}\oint_\Gamma
\mat{R}_0^{\!(\multiidx)}(\zeta)\,\d{\zeta},
\end{equation}
where the derivatives of $\mat{R}(\zeta;\cdot)$ are taken with respect
to the parameter at $\prmtrVec_0$.
In particular,
$\mat{R}_0^{\!(\vec{0})}(\zeta)=\mat{R}_0(\zeta)$ and
$\mat{P}_{\!0}^{(\vec{0})}=\mat{P}_{\!0}$.

We recall for two sufficiently smooth matrix-valued functions $\mat{F}$ and $\mat{G}$ 
of suitable sizes 
that their scaled derivatives satisfy the \emph{multivariate Leibniz rule}
\begin{equation}\label{eq:scaled_multivariate_leibniz}
(\mat{F}\mat{G})^{(\multiidx)}
=
\sum_{\vec{0}\leq\multiidxSec\leq\multiidx}
\mat{F}^{(\multiidxSec)}
\mat{G}^{(\multiidx-\multiidxSec)},
\end{equation}
where the inequality $\vec{0}\leq\multiidxSec\leq\multiidx$ 
is understood componentwise.
Applying \eqref{eq:scaled_multivariate_leibniz} to
\[
(\mat{A}(\prmtrVec)-\zeta\mat{I})
\mat{R}(\zeta;\prmtrVec)=\mat{I}
\]
yields 
\begin{equation}\label{eq:resolvent_derivative_recursion}
\mat{R}_0^{\!(\multiidx)}(\zeta)
=
-\mat{R}_0(\zeta)
\sum_{\substack{\vec{0}\neq\multiidxSec\leq\multiidx}}
\!\!
\mat{A}^{\!(\multiidxSec)}
\mat{R}_0^{\!(\multiidx-\multiidxSec)}(\zeta)
\qquad
\text{for }1\leq|\multiidx|\leq\Ord+1.
\end{equation}
Iterating \eqref{eq:resolvent_derivative_recursion} gives}
\begin{equation}\label{eq:resolvent_coefficients}
\begin{aligned}
\mat{R}_0^{\!(\multiidx)}(\zeta)
={}&
\sum_{p=1}^{|\multiidx|}(-1)^p 
\!\!\!\!\!
\sum_{\substack{
\multiidxSec_1,\ldots,\multiidxSec_p\in
\bb{N}_0^{\Pdim}\setminus\{\vec{0}\}\\
\multiidxSec_1+\cdots+\multiidxSec_p=\multiidx}}
\!\!\!\!
\mat{R}_0(\zeta)\mat{A}^{\!(\multiidxSec_1)}
\mat{R}_0(\zeta) \cdots
\mat{A}^{\!(\multiidxSec_p)}
\mat{R}_0(\zeta)
\qquad
{\text{for }1\leq|\multiidx|\leq\Ord+1.}
\end{aligned}
\end{equation}
{To state the resulting projector formula compactly 
while retaining the decomposition in \eqref{eq:laurent_series}, 
we write 
$\mat{L}_0^{(0)} \coloneqq -\mat{P}_{\!0}$ and 
$\mat{L}_0^{(k)} \coloneqq \mat{S}_0^k$ for $k\in\bb{N}$,
i.e., 
each $\mat{L}_0^{(k)}  $ is the coefficient of $(\zeta-\EigVal_0)^{k-1}$ in the Laurent expansion of $\mat{R}_0$.
Substituting \eqref{eq:resolvent_coefficients} 
into \eqref{eq:scaled_derivatives_resolvent_projector} 
and evaluating the contour integrals using \eqref{eq:laurent_series},
we obtain for $1\leq|\multiidx|\leq\Ord+1$ the explicit formula 
\begin{equation}\label{eq:projector_coefficients}
\begin{aligned}
\mat{P}_{\!0}^{(\multiidx)}
={}&
-\sum_{p=1}^{|\multiidx|}(-1)^p
\!\!\!\!
\sum_{\substack{
\multiidxSec_1,\ldots,\multiidxSec_p\in
\bb{N}_0^{\Pdim}\setminus\{\vec{0}\}\\
\multiidxSec_1+\cdots+\multiidxSec_p=\multiidx}}
\;\sum_{\substack{
k_1,\ldots,k_{p+1}\in\bb{N}_0\\
k_1+\cdots+k_{p+1}=p}}
\!\!\!
\mat{L}_0^{(k_1)}
\mat{A}^{\!(\multiidxSec_1)}
\mat{L}_0^{(k_2)}
\cdots
\mat{L}_0^{(k_p)}
\mat{A}^{\!(\multiidxSec_p)}
\mat{L}_0^{(k_{p+1})}.
\end{aligned}
\end{equation}

If $\mat{A}$ is real-analytic near $\prmtrVec_0$, the preceding construction applies for arbitrary $\Ord$.
In this case, the scaled derivatives are the coefficients of the corresponding locally convergent power series.
}

\begin{remark}[Connection to classical perturbation results] \label{rem:connection_classical_results}
In the univariate case $\Pdim=1$,
\eqref{eq:projector_coefficients} reduces to the classical
one-dimensional perturbation formula,
cf.~\cite[Chap.~2, §2.2]{Kato76}.
The multivariate extension follows from
\eqref{eq:scaled_multivariate_leibniz}.
\end{remark}

\subsection{\rev{Algebraic structure of projector derivatives}} \label{subsec:prop_proj_series}
{Recall that the total eigenprojector from \Cref{subsec:analyticity_total_proj} satisfies
\begin{equation} \label{eq:defining_eqn_total_proj}
\mat{P}_{\!0}(\prmtrVec)^2=\mat{P}_{\!0}(\prmtrVec),
\qquad
\mat{P}_{\!0}(\prmtrVec)^*=\mat{P}_{\!0}(\prmtrVec),
\qquad 
\mat{A}(\prmtrVec)\mat{P}_{\!0}(\prmtrVec)
=
\mat{P}_{\!0}(\prmtrVec)\mat{A}(\prmtrVec).
\end{equation}
Differentiating the second identity
with respect to the real parameters,
we obtain immediately: 
\begin{lemma}[Hermiticity of the scaled projector derivatives]\label{lem:Hermitian_proj_series_coeff}
For every $\multiidx\in\bb{N}_0^{\Pdim}$ with $|\multiidx|\leq\Ord+1$,
it holds that $\mat{P}_{\!0}^{(\multiidx)}
=
\bigl(\mat{P}_{\!0}^{(\multiidx)}\bigr)^*.$ 
\end{lemma}


Applying \eqref{eq:scaled_multivariate_leibniz} to the first identity
in \eqref{eq:defining_eqn_total_proj} yields
\begin{equation} \label{eq:projector_recursive_relation}
\mat{P}_{\!0}^{(\multiidx)}
\, =
\sum_{\substack{
\multiidxSec_1,\multiidxSec_2\in\bb{N}_0^{\Pdim}\\
\multiidxSec_1+\multiidxSec_2=\multiidx}}
\!\!
\mat{P}_{\!0}^{(\multiidxSec_1)}
\mat{P}_{\!0}^{(\multiidxSec_2)}
\qquad
\text{for }|\multiidx|\leq\Ord+1.
\end{equation}
This identity also controls the ranges and ranks of the scaled derivatives.

\begin{lemma}[Ranges and ranks of the scaled projector derivatives] \label{lem:rank_projector_coefficients}
Let $\multiidx \in\bb{N}_0^{\Pdim}$ satisfy $|\multiidx|\leq\Ord+1$.
Then, it holds that 
\begin{equation}\label{eq:range_projector_derivative}
\ran\mat{P}_{\!0}^{(\multiidx)}
\subseteq
\operatorname{span}
\bigl\{
\ran  \mat{P}_{\!0}^{(\multiidxSec)}\mat{P}_{\!0}  
:
\vec{0}\leq\multiidxSec\leq\multiidx
\bigr\},
\end{equation}
and more generally that  
\begin{equation}
\label{eq:range_projector_derivatives}
\operatorname{span}
\{
\ran\mat{P}_{\!0}^{(\multiidx)}
:
|\multiidx|\leq\ell
\}
=
\operatorname{span}
\{
\ran  \mat{P}_{\!0}^{(\multiidx)}\mat{P}_{\!0} 
:
|\multiidx|\leq\ell
\}
\qquad \text{for }
0 \leq \ell \leq \Ord+1. 
\end{equation}
Consequently, we have 
\begin{equation}\label{eq:rank_projector_derivative}
\rank\mat{P}_{\!0}^{(\multiidx)}
\leq \, 
m_0\prod_{j=1}^{\Pdim}(\beta_j+1).
\end{equation}
\end{lemma}

\begin{proof}
For a fixed multi-index $\multiidx$, denote the space on the right-hand side of
\eqref{eq:range_projector_derivative} by $\mathcal{V}_{\multiidx}$.
We prove \eqref{eq:range_projector_derivative} by induction on $|\multiidx|$.
The assertion is immediate for $\multiidx=\vec{0}$.
Suppose that $|\multiidx|\geq1$ and 
let $\vec{x}\in\mathcal{V}_{\multiidx}^{\perp}$,
i.e., $\vec{x} \perp \ran \, \mat{P}_{\!0}^{(\multiidxSec)}\mat{P}_{\!0} $,
or equivalently $(\mat{P}_{\!0}^{(\multiidxSec)}\mat{P}_{\!0})^*\vec{x} = \vec{0}$
for all $\vec{0}\leq\multiidxSec\leq\multiidx.$
Then, \Cref{lem:Hermitian_proj_series_coeff} yields
\begin{equation*}
\mat{P}_{\!0}\mat{P}_{\!0}^{(\multiidxSec)}\vec{x}
=
\bigl(\mat{P}_{\!0}^{(\multiidxSec)}\mat{P}_{\!0}\bigr)^*\vec{x}
=
\vec{0}
\qquad
\text{for every }\vec{0}\leq\multiidxSec\leq\multiidx.
\end{equation*}
In particular, $\mat{P}_{\!0}\vec{x}=\vec{0}$ and
$\mat{P}_{\!0}\mat{P}_{\!0}^{(\multiidx)}\vec{x}=\vec{0}$.
For every $\vec{0}\neq\multiidxSec\leq\multiidx$ with
$\multiidxSec\neq\multiidx$, we have
$\mathcal{V}_{\multiidx-\multiidxSec}\subseteq\mathcal{V}_{\multiidx}$
by their definitions.
Using the induction hypothesis, 
we then deduce that 
$\mat{P}_{\!0}^{(\multiidx-\multiidxSec)}\vec{x}=\vec{0}$.
Hence, \eqref{eq:projector_recursive_relation} gives
$\mat{P}_{\!0}^{(\multiidx)}\vec{x}=\vec{0}$.
This proves
$\mathcal{V}_{\multiidx}^{\perp}
\subseteq
\ker\mat{P}_{\!0}^{(\multiidx)}$,
or equivalently 
$\ran \bigl(\mat{P}_{\!0}^{(\multiidx)}\bigr)^* 
= \bigl( \ker \mat{P}_{\!0}^{(\multiidx)} \bigr)^\perp
\subseteq 
\mathcal{V}_{\multiidx} $. 
Hermiticity of $\mat{P}_{\!0}^{(\multiidx)}$ by \Cref{lem:Hermitian_proj_series_coeff} 
concludes the proof of \eqref{eq:range_projector_derivative}.

The inclusion "$\supseteq$" in \eqref{eq:range_projector_derivatives} is immediate,
and the other inclusion there follows from \eqref{eq:range_projector_derivative}.
Finally, 
the space on the right-hand side of \eqref{eq:range_projector_derivative}
is spanned by $\prod_{j=1}^{\Pdim}(\beta_j+1)$ subspaces of dimension at most $m_0$,
and this proves \eqref{eq:rank_projector_derivative}.
\end{proof}
}
The explicit formula \eqref{eq:projector_coefficients} is generally too expensive for the computation of higher-order derivatives.
The following recursion is the multivariate counterpart of the density matrix perturbation formula in \cite{McW62}; see also \cite[Sec.~10]{GarS24}.

\begin{proposition}[Multivariate recursion for the scaled projector derivatives]\label{prop:multivariate_DMPT}
{For every $\multiidx\in\bb{N}_0^{\Pdim}$ with
$1\leq|\multiidx|\leq\Ord+1$, it holds that}
\begin{align} \label{eq:recursive_relation_projector_coefficients_appendix}
\mat{P}_{\!0}^{(\multiidx)}
={}&
-\mat{P}_{\!0}\mat{Z}^{(\multiidx)}\mat{P}_{\!0}
+\mat{P}_{\!0}^{\perp}\mat{Z}^{(\multiidx)}\mat{P}_{\!0}^{\perp}
-\mat{S}_0\mat{W}^{(\multiidx)}\mat{P}_{\!0}
-\bigl(\mat{S}_0\mat{W}^{(\multiidx)}\mat{P}_{\!0}\bigr)^* 
\end{align}
with
\begin{align*}
\mat{Z}^{(\multiidx)}
\coloneqq
\!\!\!\!
\sum_{\substack{
\multiidxSec_1,\multiidxSec_2\in\bb{N}_0^{\Pdim}\setminus\{\vec{0}\}\\
\multiidxSec_1+\multiidxSec_2=\multiidx}}
\!\!\!\!
\mat{P}_{\!0}^{(\multiidxSec_1)}
\mat{P}_{\!0}^{(\multiidxSec_2)}
\qquad \text{and} \qquad
\mat{W}^{(\multiidx)}
\coloneqq
\!\!\!\!\!\!
\sum_{\substack{
\multiidxSec_1,\multiidxSec_2\in\bb{N}_0^{\Pdim}\\
\multiidxSec_1\neq\vec{0},\;
\multiidxSec_1+\multiidxSec_2=\multiidx}}
\!\!\!\!
\bigl(
\mat{A}^{\!(\multiidxSec_1)}
\mat{P}_{\!0}^{(\multiidxSec_2)}
-
\mat{P}_{\!0}^{(\multiidxSec_2)}
\mat{A}^{\!(\multiidxSec_1)}
\bigr).
\end{align*}
In particular,
the computations of $\mat{P}_{\!0}^{(\multiidx)}$
involve only the scaled derivatives
$\mat{P}_{\!0}^{(\multiidxSec)}$ with
$|\multiidxSec|<|\multiidx|$.
\end{proposition}

{\begin{proof}
Separating the two terms containing
$\mat{P}_{\!0}^{(\multiidx)}$ in
\eqref{eq:projector_recursive_relation} gives
\begin{equation}\label{eq:projector_idempotent_derivative_relation}
\mat{P}_{\!0}^{(\multiidx)}
=
\mat{P}_{\!0}\mat{P}_{\!0}^{(\multiidx)}
+
\mat{P}_{\!0}^{(\multiidx)}\mat{P}_{\!0}
+
\mat{Z}^{(\multiidx)}.
\end{equation}
Multiplying \eqref{eq:projector_idempotent_derivative_relation}
from the left and right by $\mat{P}_{\!0}$, and separately by
$\mat{P}_{\!0}^{\perp}$, 
yields the identities
\begin{align}\label{eq:diagonal_blocks_projector_derivative}
\mat{P}_{\!0}\mat{P}_{\!0}^{(\multiidx)}\mat{P}_{\!0}
=
-\mat{P}_{\!0}\mat{Z}^{(\multiidx)}\mat{P}_{\!0}
\qquad \text{and} \qquad 
\mat{P}_{\!0}^{\perp}\mat{P}_{\!0}^{(\multiidx)}
\mat{P}_{\!0}^{\perp}
=
\mat{P}_{\!0}^{\perp}\mat{Z}^{(\multiidx)}
\mat{P}_{\!0}^{\perp}.
\end{align}

For the off-diagonal blocks, take the scaled derivative of
$\mat{A}(\prmtrVec)\mat{P}_{\!0}(\prmtrVec)
=
\mat{P}_{\!0}(\prmtrVec)\mat{A}(\prmtrVec)$.
The multivariate Leibniz rule 
\eqref{eq:scaled_multivariate_leibniz} then gives
\begin{equation}\label{eq:commutation_relation_derivative}
\mat{A}_0\mat{P}_{\!0}^{(\multiidx)}
-\mat{P}_{\!0}^{(\multiidx)}\mat{A}_0
\, = \, 
-\mat{W}^{(\multiidx)}.
\end{equation}
Multiplying \eqref{eq:commutation_relation_derivative} by
$\mat{P}_{\!0}^{\perp}$ from the left and by $\mat{P}_{\!0}$ from the right
and using $\mat{A}_0\mat{P}_{\!0} = \EigVal_0\mat{P}_{\!0}$,
we get 
$(\mat{A}_0-\EigVal_0\mat{I})
\mat{P}_{\!0}^{\perp}
\mat{P}_{\!0}^{(\multiidx)}
\mat{P}_{\!0}
=
-\mat{P}_{\!0}^{\perp}
\mat{W}^{(\multiidx)}
\mat{P}_{\!0}.$
Then, we multiply this new equation 
by $\mat{S}_0$ from the left 
and use \eqref{eq:reduced_resolvent}
to deduce 
$\mat{P}_{\!0}^{\perp}
\mat{P}_{\!0}^{(\multiidx)}
\mat{P}_{\!0}
=
-\mat{S}_0\mat{W}^{(\multiidx)}\mat{P}_{\!0}$.
The remaining off-diagonal block in \eqref{eq:recursive_relation_projector_coefficients_appendix} 
is its adjoint by
\Cref{lem:Hermitian_proj_series_coeff}.
Combining the four blocks proves
\eqref{eq:recursive_relation_projector_coefficients_appendix}.
\end{proof}

\begin{remark}[Individual eigenvalues and eigenvectors] \label{rem:individual_eigenvalues}
For $m_0>1$,
the individual eigenvalues enclosed by $\Gamma$ 
can be labelled as continuous functions of the real parameters 
near $\prmtrVec_0$ and converge to~$\EigVal_0$ 
as $\prmtrVec\to\prmtrVec_0$, 
cf.~\cite[Chap.~2, \S5.1 and~\S5.7]{Kato76}.
However, the corresponding individual eigenvectors need not admit continuous choices in several parameters, as illustrated by \eqref{eq:kato_example}.
\end{remark}

\begin{remark}[Several eigenvalues] \label{rem:multiple_eigenvalues}
For each distinct eigenvalue $\lambda_1<\ldots<\lambda_K$ of $\mat{A}_0$,
choosing disjoint contours and applying the construction from
\Cref{subsec:analyticity_total_proj} to each eigenvalue gives
$\mathrm{C}^{\Ord +1}$ orthogonal eigenprojectors
$\mat{P}_{\!1}(\prmtrVec),\ldots,\mat{P}_{\!K}(\prmtrVec)$.
Their sum is the total eigenprojector associated with these eigenvalues
and its scaled derivatives are the sums of the corresponding scaled derivatives.
Consequently, the preceding algebraic properties remain applicable, 
with $m_0$ replaced by
the total rank, while \Cref{prop:multivariate_DMPT} can be applied to each
eigenprojector separately.
\end{remark}

The preceding results concern the derivatives 
of a single projector-valued mapping. 
We conclude this subsection 
with a comparison result for two such mappings, 
whose scaled derivatives agree through order $\Ord$. 
Concretely, 
let $\mat{Q}$ be another $\mathrm{C}^{\Ord +1}$ mapping from a
neighborhood of $\prmtrVec_0$ into
$\operatorname{Gr}(M,\fdim)$, 
and suppose that the scaled derivatives of $\mat{P}$ and $\mat{Q}$ 
at $\prmtrVec_0$ coincide up to order $\Ord$,
i.e., {$\mat{Q}^{(\multiidxSec)}=\mat{P}^{(\multiidxSec)}$ 
for $|\multiidxSec|\leq\Ord$}.
We also set
$\projDiff(\prmtrVec)\coloneqq
\mat{Q}(\prmtrVec)-\mat{P}(\prmtrVec)$ and
$\mat{P}^{\perp}(\prmtrVec)\coloneqq
\mat{I}-\mat{P}(\prmtrVec)$.

\begin{lemma}[First nonmatching derivatives of projector-valued mappings]
\label{lem:first_defect_projector_mappings}
For every $|\multiidx|=\Ord+1$,
$\projDiff^{\!(\multiidx)}$ is off-diagonal with respect to
$\ran\mat{P}(\prmtrVec_0)\oplus
\ran\mat{P}^{\perp}(\prmtrVec_0)$ and satisfies
\[
\|\projDiff^{\!(\multiidx)}\|
=
\|\projDiff^{\!(\multiidx)}\mat{P}(\prmtrVec_0)\|
\qquad\text{and}\qquad
\rank\projDiff^{\!(\multiidx)}\leq2M.
\]
\end{lemma}

\begin{remark}[Differential-geometric interpretation]
\label{rem:tangent_projector_mappings}
The off-diagonal structure, 
together with Hermiticity,
means that the first nonmatching derivative $\projDiff^{\!(\multiidx)}$
is a tangent vector to the Grassmann manifold $\operatorname{Gr}(M,\fdim)$
at the common projector
$\mat{P}(\prmtrVec_0)=\mat{Q}(\prmtrVec_0)$;
see, e.g., \cite[Prop.~1]{BenZA24}.
More generally,
whenever two $\mathrm{C}^{\Ord+1}$ mappings into
$\operatorname{Gr}(M,\fdim)$ share their scaled derivatives
up to order $\Ord$, their first nonmatching derivative
is a tangent vector to $\operatorname{Gr}(M,\fdim)$
at their common projector.
\end{remark}

\begin{proof}[Proof of \Cref{lem:first_defect_projector_mappings}]
Let $|\multiidx|=\Ord+1$. 
Taking the scaled derivative of order
$\multiidx$ in $\mat{Q}^2=\mat{Q}$ and $\mat{P}^2=\mat{P}$, 
subtracting the resulting identities, 
and using {$\mat{Q}^{(\multiidxSec)}=\mat{P}^{(\multiidxSec)}$ 
for $|\multiidxSec|\leq\Ord$}, 
we obtain 
\begin{equation}
\mat{P}(\prmtrVec_0)\projDiff^{\!(\multiidx)}
+\projDiff^{\!(\multiidx)}\mat{P}(\prmtrVec_0)
\,=\, 
\projDiff^{\!(\multiidx)}.
\label{eq:projector_difference_key_identity}
\end{equation}
Multiplying \eqref{eq:projector_difference_key_identity} 
from the left and right by
$\mat{P}(\prmtrVec_0)$ and separately by
$\mat{P}^{\perp}(\prmtrVec_0)$ gives
\[
\begin{aligned}
2\mat{P}(\prmtrVec_0)\projDiff^{\!(\multiidx)}
\mat{P}(\prmtrVec_0)
=
\mat{P}(\prmtrVec_0)\projDiff^{\!(\multiidx)}
\mat{P}(\prmtrVec_0)
\qquad \text{and} \qquad
\mat{0}
=
\mat{P}^{\perp}(\prmtrVec_0)\projDiff^{\!(\multiidx)}
\mat{P}^{\perp}(\prmtrVec_0),
\end{aligned}
\]
respectively.
Thus, the diagonal blocks of $\projDiff^{\!(\multiidx)}$ vanish, 
i.e., 
$$
\projDiff^{\!(\multiidx)}=
\mat{P}(\prmtrVec_0)\projDiff^{\!(\multiidx)}
\mat{P}^{\perp}(\prmtrVec_0)
+\mat{P}^{\perp}(\prmtrVec_0)\projDiff^{\!(\multiidx)}
\mat{P}(\prmtrVec_0).
$$
Since $\mat{P}$ and $\mat{Q}$ are Hermitian-valued,
$\projDiff^{\!(\multiidx)}$ is Hermitian, too. 
Hence, its two off-diagonal blocks are adjoints of each other.
Set
$\mat{F}\coloneqq
\mat{P}^{\perp}(\prmtrVec_0)\projDiff^{\!(\multiidx)}
\mat{P}(\prmtrVec_0)$.
With respect to the stated orthogonal decomposition, 
we can write 
\[
\projDiff^{\!(\multiidx)}
=
\begin{pmatrix}\mat{0}&\mat{F}^*\\\mat{F}&\mat{0}\end{pmatrix},
\qquad 
\projDiff^{\!2}=
\begin{pmatrix}
\mat{F}^*\mat{F}&\mat{0}\\\mat{0}&\mat{F}\mat{F}^*
\end{pmatrix},
\qquad
\projDiff^{\!(\multiidx)}\mat{P}(\prmtrVec_0)
=
\begin{pmatrix}\mat{0}&\mat{0}\\\mat{F}&\mat{0}\end{pmatrix}.
\]
This representation,
together with 
$\|\mat{M}\|^2=\|\mat{M}^*\mat{M}\|$ for any $\mat{M}$
and the fact that the spectral norm of a block-diagonal matrix is
the maximum of the norms of its diagonal blocks,
yields 
\[
\|\projDiff^{\!(\multiidx)}\|^2
=
\|\bigl(\projDiff^{\!(\multiidx)}\bigr)^2\|
=
\max\{\|\mat{F}^*\mat{F}\|,\|\mat{F}\mat{F}^*\|\}
=
\|\mat{F}\|^2
=
\|\projDiff^{\!(\multiidx)}\mat{P}(\prmtrVec_0)\|^2,
\]
which proves the norm identity in the statement. 
Finally,
each off-diagonal block has rank at most~$M$,
and so the rank bound follows.
\end{proof}
}

\subsection{\rev{Finite Taylor approximation and spectral correction}}\label{subsec:truncated_projector_series}
{We now return to 
the 
total eigenprojector $\mat{P}(\prmtrVec)$ satisfying
\eqref{eq:target_projector}.
By \Cref{rem:multiple_eigenvalues}, 
$\mat{P}$ is of class $\mathrm{C}^{\Ord +1}$ 
in a neighborhood of $\prmtrVec_0$.
We denote its Taylor polynomial of order $\Ord$ at
$\prmtrVec_0$ by
\begin{equation}\label{eq:series_P_Pn}
\mat{P}^{[\Ord]}(\prmtrVec)
\coloneqq
\sum_{|\multiidx|\leq\Ord}
(\prmtrVec-\prmtrVec_0)^{\multiidx}\mat{P}^{(\multiidx)}.
\end{equation}
Then, Taylor's theorem  gives
\begin{equation}\label{eq:finite_taylor_remainder_P}
\mat{P}(\prmtrVec)-\mat{P}^{[\Ord]}(\prmtrVec)
=
\sum_{|\multiidx|=\Ord+1}
(\prmtrVec-\prmtrVec_0)^{\multiidx}\mat{P}^{(\multiidx)}
+ \smallO \bigl(|\prmtrVec-\prmtrVec_0|^{\Ord+1}\bigr).
\end{equation}
In particular, 
$\|\mat{P}(\prmtrVec)-\mat{P}^{[\Ord]}(\prmtrVec)\|
=
\bigO\bigl(|\prmtrVec-\prmtrVec_0|^{\Ord+1}\bigr).$

In general, $\mat{P}^{[\Ord]}(\prmtrVec)$ is not an orthogonal projector,
but we can extract from it an orthogonal projector that approximates
$\mat{P}(\prmtrVec)$.
}

{\begin{lemma}[Spectral correction of the finite Taylor polynomial]
\label{lem:closedness_test_space}
There exist $\varepsilon>0$ and a real-analytic orthogonal projector-valued
mapping $\mat{Q}_{\rm{pt}}$,
which satisfies for
$|\prmtrVec-\prmtrVec_0|\leq\varepsilon$ that 
\begin{equation}\label{eq:range_Q_pt}
\ran\mat{Q}_{\rm{pt}}(\prmtrVec)
\subseteq
\ran\mat{P}^{[\Ord]}(\prmtrVec)
\subseteq
\operatorname{span}
\bigl\{\ran\mat{P}^{(\multiidx)}:|\multiidx|\leq\Ord\bigr\},
\end{equation}
$\rank\mat{Q}_{\rm{pt}}(\prmtrVec)=M$ and
$\mat{Q}_{\rm{pt}}(\prmtrVec_0)=\mat{P}(\prmtrVec_0)$.
Consequently, we have 
$\|\mat{Q}_{\rm{pt}}(\prmtrVec)-\mat{P}(\prmtrVec)\|
= \bigO(|\prmtrVec-\prmtrVec_0|^{\Ord+1})$ 
as $\prmtrVec\to\prmtrVec_0$, and
more quantitatively,
\begin{equation}\label{eq:Q_pt_asymptotic_prefactor}
\begin{aligned}
\limsup_{\prmtrVec\to\prmtrVec_0}
\frac{\|\mat{Q}_{\rm{pt}}(\prmtrVec)-\mat{P}(\prmtrVec)\|}
{|\prmtrVec-\prmtrVec_0|^{\Ord+1}}
\, \leq \, 
2\limsup_{\prmtrVec\to\prmtrVec_0}
\frac{\|\mat{P}^{[\Ord]}(\prmtrVec)-\mat{P}(\prmtrVec)\|}
{|\prmtrVec-\prmtrVec_0|^{\Ord+1}}
\, \leq \, 
2
\!\!\!
\sum_{|\multiidx|=\Ord+1}
\!\!\!
\|\mat{P}^{(\multiidx)}\|.
\end{aligned}
\end{equation}
\end{lemma}

\begin{proof}
Note that $\mat{P}^{[\Ord]}(\prmtrVec)$ is Hermitian by
\Cref{lem:Hermitian_proj_series_coeff} and \Cref{rem:multiple_eigenvalues}.
Choose an arbitrary $\eta\in(0,1/2)$.
Since $\mat{P}^{[\Ord]}(\prmtrVec_0)=\mat{P}(\prmtrVec_0)$,
\eqref{eq:finite_taylor_remainder_P} allows us to choose $\varepsilon>0$
such that
\begin{equation}\label{eq:P_taylor_eta_bound}
\|\mat{P}^{[\Ord]}(\prmtrVec)-\mat{P}(\prmtrVec)\|
\leq\eta
\qquad
\text{for }|\prmtrVec-\prmtrVec_0|\leq\varepsilon.
\end{equation}

Let $\{\tilde{\EigVal}_j(\mat{P}^{[\Ord]}(\prmtrVec))\}_{j=1}^\fdim$ and 
$\{\tilde{\EigVal}_j(\mat{P}(\prmtrVec))\}_{j=1}^\fdim$ 
be the eigenvalues of $\mat{P}^{[\Ord]}(\prmtrVec)$ and $\mat{P}(\prmtrVec)$ respectively, 
arranged in nondecreasing order and counting multiplicities.
By Weyl's inequality \cite[Sec.~10.3]{Par98}, we know that  
$|\tilde{\EigVal}_j(\mat{P}^{[\Ord]}(\prmtrVec))
-\tilde{\EigVal}_j(\mat{P}(\prmtrVec))|
\leq\|\mat{P}^{[\Ord]}(\prmtrVec)-\mat{P}(\prmtrVec)\|
\leq\eta$ for every $1\leq j\leq\fdim$.
Since $\mat{P}(\prmtrVec)$ is an orthogonal projector,
it has only eigenvalues $1$ and $0$ with respective
multiplicities $M$ and $\fdim-M$.
Hence, we deduce that
\begin{equation}\label{eq:spectrum_P_taylor}
\sigma\bigl(\mat{P}^{[\Ord]}(\prmtrVec)\bigr)
\subseteq
[-\eta,\eta]\cup[1-\eta,1+\eta]
\qquad
\text{for }|\prmtrVec-\prmtrVec_0|\leq\varepsilon,
\end{equation}
and that the second interval contains precisely $M$ eigenvalues, 
counted with multiplicity.

Since $\eta<1/2$, the annulus
$\{\zeta\in\bb{C}:\eta<|\zeta-1|<1/2\}$ is free of the spectrum of
$\mat{P}^{[\Ord]}(\prmtrVec)$.
Let $\Gamma_{\!1}\coloneqq\{\zeta\in\bb{C}:|\zeta-1|=1/2\}$.
This spectral separation is illustrated in
\Cref{fig:spectral_separation_Q_pt}.

\begin{figure}[htbp]
\centering
\tikzexternaldisable
\begin{tikzpicture}[
	    x=3.2cm,
	    y=3.2cm,
	    every node/.style={font=\small,text=black}
]
    \fill[blue!9,even odd rule]
        (1,0) circle[radius=0.5]
        (1,0) circle[radius=0.14];

    \draw[dashed,blue!55!black] (0,0) circle[radius=0.14];
    \draw[dashed,blue!55!black] (1,0) circle[radius=0.14];
	    \draw[->,black!70] (-0.55,0) -- (2.00,0)
	        node[right] {$\operatorname{Re}\zeta$};
    \draw[->,black!70] (0,-0.58) -- (0,0.60)
        node[above] {$\operatorname{Im}\zeta$};

    \draw[
        very thick,
        blue!70!black,
        postaction={decorate},
        decoration={markings,mark=at position 0.16 with {\arrow{>}}}
    ] (1.5,0) arc[start angle=0,end angle=360,radius=0.5];

    \fill[black] (0.5,0) circle[radius=2pt];
    \fill[black] (1.5,0) circle[radius=2pt];

    \foreach \x in {-0.1,-0.08,-0.05,0.03,0.02,0.015,0,0.05,0.07,0.09,0.1}
        \node[text=red!75!black] at (\x,0) {$\times$};
    \foreach \x in {0.9,0.94,1.02,1.09}
        \node[text=red!75!black] at (\x,0) {$\times$};
    \fill[black] (0,0) circle[radius=1.6pt];
    \fill[black] (1,0) circle[radius=1.6pt];

    \node[below=6pt] at (0,0.05) {$0$};
    \node[below=6pt] at (1,0.05) {$1$};
    \node[below=6pt] at (0.425,0.05) {$1/2$};
    \node[below=6pt] at (1.58,0.05) {$3/2$};
	    \node[text=blue!70!black] at (1.52,0.33) {$\Gamma_{\!1}$};
    \node[text=blue!55!black,align=center] at (1,0.25)
        {$\eta<|\zeta-1|<1/2$};
    \node[text=blue!55!black] at (0,-0.20) {$[-\eta,\eta]$};
    \node[text=blue!55!black] at (1,-0.20) {$[1-\eta,1+\eta]$};
\end{tikzpicture}
\tikzexternalenable
\caption{Spectral separation in the proof of
\Cref{lem:closedness_test_space}. 
Dots at $0$ and $1$ mark the eigenvalues
of~$\mat{P}(\prmtrVec)$,
dots at $1/2$ and $3/2$ mark the intersections
of $\Gamma_{\!1}$ with the real axis,
while crosses illustrate the eigenvalues
of~$\mat{P}^{[\Ord]}(\prmtrVec)$.
The contour $\Gamma_{\!1}$ encloses
precisely the cluster near $1$, 
and the shaded annulus is free of spectrum.
}
\label{fig:spectral_separation_Q_pt}
\end{figure}

We define
\begin{equation}\label{eq:def_Q_pt}
\mat{Q}_{\rm{pt}}(\prmtrVec)
\coloneqq
-\frac{1}{2\pi\i}
\oint_{\Gamma_{\!1}}
\bigl(\mat{P}^{[\Ord]}(\prmtrVec)-\zeta\mat{I}\bigr)^{-1}
\,\d{\zeta},
\end{equation}
which is the rank-$M$ orthogonal spectral projector onto
$\ran \mat{P}^{[\Ord]}(\prmtrVec)$ associated with the eigenvalues 
in~$[1-\eta,1+\eta]$.
This proves the first inclusion in \eqref{eq:range_Q_pt},
and the second follows
directly from \eqref{eq:series_P_Pn}.
Also observe that $\mat{P}^{[\Ord]}$ is a polynomial in $\prmtrVec$, 
so $\mat{Q}_{\rm{pt}}$ is real-analytic for $\prmtrVec$ close to $\prmtrVec_0$.

Note that the orthogonal projector $\mat{P}(\prmtrVec)$ has the analogous contour
representation over $\Gamma_{\!1}$.
Applying the resolvent identity
$(\mat{X} - \zeta \mat{I})^{-1}-(\mat{Y}-\zeta \mat{I})^{-1}
=(\mat{X} - \zeta \mat{I})^{-1}(\mat{Y}-\mat{X})(\mat{Y} - \zeta \mat{I})^{-1}$
to the two contour representations gives
\begin{equation}\label{eq:Q_pt_P_difference_int_representation}
\begin{aligned}
\mat{Q}_{\rm{pt}}(\prmtrVec)-\mat{P}(\prmtrVec)
=\frac{1}{2\pi\i}
\oint_{\Gamma_{\!1}}
\bigl(\mat{P}^{[\Ord]}(\prmtrVec)-\zeta\mat{I}\bigr)^{-1}
\bigl(\mat{P}^{[\Ord]}(\prmtrVec)-\mat{P}(\prmtrVec)\bigr)
\bigl(\mat{P}(\prmtrVec)-\zeta\mat{I}\bigr)^{-1}
\,\d{\zeta}.
\end{aligned}
\end{equation}
For $\zeta\in\Gamma_{\!1}$, 
\eqref{eq:spectrum_P_taylor} yields
$\|\bigl(\mat{P}^{[\Ord]}(\prmtrVec)-\zeta\mat{I}\bigr)^{-1}\|
\leq(1/2-\eta)^{-1}$
and 
$\|\bigl(\mat{P}(\prmtrVec)-\zeta\mat{I}\bigr)^{-1}\|=2$.
Since the length of $\Gamma_{\!1}$ is $\pi$, taking norms in
\eqref{eq:Q_pt_P_difference_int_representation} proves
\begin{equation}
\|\mat{Q}_{\rm{pt}}(\prmtrVec)-\mat{P}(\prmtrVec)\|
\leq
(1/2-\eta)^{-1}
\|\mat{P}^{[\Ord]}(\prmtrVec)-\mat{P}(\prmtrVec)\|
\label{eq:Qpt_estimate_by_Pn}. 
\end{equation}

For every $\eta\in(0,1/2)$, the same contour formula
\eqref{eq:def_Q_pt} defines $\mat{Q}_{\rm{pt}}$ on a possibly smaller
neighborhood of $\prmtrVec_0$, 
and \eqref{eq:Qpt_estimate_by_Pn} remains valid.
Dividing both sides of \eqref{eq:Qpt_estimate_by_Pn} 
by $|\prmtrVec-\prmtrVec_0|^{\Ord+1}$, 
taking the limit superior 
(to account for all directions in which $\prmtrVec \to \prmtrVec_0$), 
and letting $\eta\downarrow0$ then proves the first inequality in
\eqref{eq:Q_pt_asymptotic_prefactor}.
The second follows from \eqref{eq:finite_taylor_remainder_P} and
$|(\prmtrVec-\prmtrVec_0)^{\multiidx}|
\leq
|\prmtrVec-\prmtrVec_0|^{\Ord+1}$
for $|\multiidx|=\Ord+1$.
\end{proof}
}

{\begin{remark}[Practical computation of $\mat{Q}_{\rm{pt}}$]
\label{rem:practical_computation_Q_pt}
For a parameter $\prmtrVec$ close to $\prmtrVec_0$,
first assemble $\mat{P}^{[\Ord]}(\prmtrVec)$ 
with the recursive formula \Cref{prop:multivariate_DMPT}. 
By the spectral separation in
\eqref{eq:spectrum_P_taylor}, $\mat{Q}_{\rm{pt}}(\prmtrVec)$ is the
orthogonal projector onto the leading $M$ left singular vectors of
$\mat{P}^{[\Ord]}(\prmtrVec)$. 
Therefore, it can be computed directly by
a rank-$M$ truncated SVD, 
\emph{without} evaluating the contour integral in
\eqref{eq:def_Q_pt}. 
This computation still requires the
ambient-dimensional matrix $\mat{P}^{[\Ord]}(\prmtrVec)$, which is
generally dense.
\end{remark}
}

{We next investigate the difference between the total Ritz-value sums
associated with $\mat{P}(\prmtrVec)$ and
$\mat{Q}_{\rm{pt}}(\prmtrVec)$.
For this purpose, 
we set ${\projDiff}_{\rm{pt}}(\prmtrVec)
\coloneqq
\mat{Q}_{\rm{pt}}(\prmtrVec)-\mat{P}(\prmtrVec)$,
whose scaled derivatives at
$\prmtrVec_0$ satisfy
${\projDiff}_{\rm{pt}}^{\!(\multiidx)}
=
\mat{Q}_{\rm{pt}}^{(\multiidx)}-\mat{P}^{(\multiidx)}$
for $|\multiidx|\leq\Ord+1$.

\begin{lemma}[First nonmatching derivatives]\label{lem:first_defect_Q_pt}
The scaled derivatives of
${\projDiff}_{\rm{pt}}=\mat{Q}_{\rm{pt}}-\mat{P}$ satisfy
\begin{equation}\label{eq:lower_derivatives_Q_pt_P_agree}
{{\projDiff}_{\rm{pt}}^{\!(\multiidxSec)}=\mat{0}
\qquad\text{for }|\multiidxSec|\leq\Ord.}
\end{equation}
For every $|\multiidx|=\Ord+1$,
${\projDiff}_{\rm{pt}}^{\!(\multiidx)}$ is off-diagonal with respect to
$\ran\mat{P}(\prmtrVec_0)\oplus
\ran\mat{P}^{\perp}(\prmtrVec_0)$ and satisfies
\begin{equation}\label{eq:first_defect_Q_pt_norm_rank}
\rank {\projDiff}_{\rm{pt}}^{\!(\multiidx)} \leq 2M
\qquad\text{and}\qquad
\|{\projDiff}_{\rm{pt}}^{\!(\multiidx)}\|
\leq
\|\mat{P}^{(\multiidx)}\|.
\end{equation}
\end{lemma}
For a differential-geometric interpretation of this off-diagonal structure,
we refer to \Cref{rem:tangent_projector_mappings}.

\begin{proof}
By \Cref{lem:closedness_test_space},
${\projDiff}_{\rm{pt}}(\prmtrVec)
=\bigO(|\prmtrVec-\prmtrVec_0|^{\Ord+1})$,
so Taylor's theorem yields
\eqref{eq:lower_derivatives_Q_pt_P_agree}.
{For $|\multiidx|=\Ord+1$, applying
\Cref{lem:first_defect_projector_mappings} with
$\mat{Q}=\mat{Q}_{\rm{pt}}$ gives the stated off-diagonal structure
and rank bound.}

As a spectral projector onto $\ran\mat{P}^{[\Ord]}$, the mapping
$\mat{Q}_{\rm{pt}}$ satisfies
$\mat{Q}_{\rm{pt}}\mat{P}^{[\Ord]}
=\mat{P}^{[\Ord]}\mat{Q}_{\rm{pt}}$
for $\prmtrVec$ close to $\prmtrVec_0$.
Since $\mat{P}^{[\Ord]}$ has no term of order $\Ord+1$, taking the
scaled derivative of order $\multiidx$ at $\prmtrVec_0$ and using
\eqref{eq:lower_derivatives_Q_pt_P_agree} gives
$\mat{Q}_{\rm{pt}}^{(\multiidx)}\mat{P}(\prmtrVec_0)
=
\mat{P}(\prmtrVec_0)\mat{Q}_{\rm{pt}}^{(\multiidx)}.$
Hence,
$\mat{P}^{\perp}(\prmtrVec_0)
\mat{Q}_{\rm{pt}}^{(\multiidx)}
\mat{P}(\prmtrVec_0)=\mat{0}$.
Using the off-diagonal structure and the norm identity from
\Cref{lem:first_defect_projector_mappings}, we obtain
\[
\begin{aligned}
\|{\projDiff}_{\rm{pt}}^{\!(\multiidx)}\|
=
\|{\projDiff}_{\rm{pt}}^{\!(\multiidx)}
\mat{P}(\prmtrVec_0)\| =
\|\mat{P}^{\perp}(\prmtrVec_0)
\mat{P}^{(\multiidx)}
\mat{P}(\prmtrVec_0)\|
\leq
\|\mat{P}^{(\multiidx)}\|,
\end{aligned}
\]
which concludes the proof. 
\end{proof}
}
{To quantify the leading-order Ritz-value error, we define
\begin{equation}\label{eq:def_A_pt_n}
A^{(\Ord+1)}_{\rm{pt}}
\coloneqq
\max_{|\multiidx|=\Ord+1}
\|
\mat{A}(\prmtrVec_0)
{\projDiff}_{\rm{pt}}^{\!(\multiidx)}
\|.
\end{equation}

\begin{theorem}[Total Ritz-value error of the corrected Taylor polynomial]
\label{thm:energy_difference_truncation}
Under the setting of \Cref{lem:closedness_test_space}, we have
\[
\tr \, 
\mat{A}(\prmtrVec)
(\mat{Q}_{\rm{pt}}(\prmtrVec)-\mat{P}(\prmtrVec))
\,=\, 
\bigO\bigl(|\prmtrVec-\prmtrVec_0|^{2\Ord+2}\bigr)
\qquad\text{as }\prmtrVec\to\prmtrVec_0.
\]
More quantitatively, it holds that 
\begin{equation}\label{eq:leading_term_energy_difference}
\begin{aligned}
\limsup_{\prmtrVec\to\prmtrVec_0}
\frac{
\left|
\tr\,
\mat{A}(\prmtrVec)
(\mat{Q}_{\rm{pt}}(\prmtrVec)-\mat{P}(\prmtrVec))
\right|}
{|\prmtrVec-\prmtrVec_0|^{2\Ord+2}}
\, \leq \, 
2M
\binom{\Ord+\Pdim}{\Pdim-1}
A^{(\Ord+1)}_{\rm{pt}}
\!\!
\sum_{|\multiidx|=\Ord+1}
\!\!\!
\|\mat{P}^{(\multiidx)}\|.
\end{aligned}
\end{equation}
\end{theorem}
}

{Notably, this bound contains no global factor
$\|\mat{A}(\prmtrVec_0)\|$, 
and the influence of the unperturbed matrix is
captured more precisely by the quantity
$A_{\rm{pt}}^{(\Ord+1)}$.
Besides, 
the product
$A_{\rm{pt}}^{(\Ord+1)}
\sum_{|\multiidx|=\Ord+1}\|\mat{P}^{(\multiidx)}\|$
reflects the quadratic projector error in
\Cref{lem:energy_difference_general}:
the first factor $A_{\rm{pt}}^{(\Ord+1)}$ 
controls its energy-weighted part, while the second
controls the remaining projector defect.
}

{\begin{proof}
For brevity, set
$E_{\rm{pt}}(\prmtrVec)
\coloneqq
\tr \mat{A}(\prmtrVec){\projDiff}_{\rm{pt}}(\prmtrVec) $.
By \eqref{eq:lower_derivatives_Q_pt_P_agree} and the finite Taylor
expansion of~${\projDiff}_{\rm{pt}}$, we have
\begin{equation}
    \label{eq:finite_taylor_first_defect_Q_pt}
{\projDiff}_{\rm{pt}}(\prmtrVec)
=
\!\!\!
\sum_{|\multiidx|=\Ord+1}
\!\!\!
(\prmtrVec-\prmtrVec_0)^{\multiidx}
{\projDiff}_{\rm{pt}}^{\!(\multiidx)}
+\smallO\bigl(|\prmtrVec-\prmtrVec_0|^{\Ord+1}\bigr).
\end{equation}
Furthermore, \Cref{lem:energy_difference_general} gives
$E_{\rm{pt}}(\prmtrVec)
=
\tr 
\mat{A}(\prmtrVec)
\bigl(\mat{I}-2\mat{P}(\prmtrVec)\bigr)
{\projDiff}_{\rm{pt}}(\prmtrVec)^2
.$
Substituting \eqref{eq:finite_taylor_first_defect_Q_pt} into this
identity and using the continuity of $\mat{A}$ and $\mat{P}$ yields
\begin{equation}\label{eq:leading_expansion_energy_difference}
\begin{aligned}
E_{\rm{pt}}(\prmtrVec)
=
\!\!\!\!\!
\sum_{|\multiidxSec_1|=|\multiidxSec_2|=\Ord+1}
\!\!\!\!
(\prmtrVec-\prmtrVec_0)^{\multiidxSec_1+\multiidxSec_2}
\,\tr \,
\mat{A}(\prmtrVec_0)
\bigl(\mat{I}-2\mat{P}(\prmtrVec_0)\bigr)
{\projDiff}_{\rm{pt}}^{(\multiidxSec_1)}
{\projDiff}_{\rm{pt}}^{(\multiidxSec_2)}
\, + \, 
\smallO\bigl(|\prmtrVec-\prmtrVec_0|^{2\Ord+2}\bigr).
\end{aligned}
\end{equation}
This proves the asserted asymptotic order.

For $|\multiidxSec_1|=|\multiidxSec_2|=\Ord+1$, 
the matrix inequalities
$|\tr \mat{M}_1\mat{M}_2 |
\leq \rank \mat{M}_1 \mat{M}_2 \cdot
\|\mat{M}_1\|\cdot\|\mat{M}_2\|$
and  $\rank \mat{M}_1\mat{M}_2 \leq \rank \mat{M}_2$
applied to 
$\mat{M}_1 \coloneqq \mat{A}(\prmtrVec_0)
(\mat{I}-2\mat{P}(\prmtrVec_0))
{\projDiff}_{\rm{pt}}^{(\multiidxSec_1)}$
and $\mat{M}_2 \coloneqq {\projDiff}_{\rm{pt}}^{(\multiidxSec_2)}$,
together with
\eqref{eq:first_defect_Q_pt_norm_rank}, 
give
\begin{align*}
\bigl|
\tr 
\mat{A}(\prmtrVec_0)
(\mat{I}-2\mat{P}(\prmtrVec_0))
{\projDiff}_{\rm{pt}}^{(\multiidxSec_1)}
{\projDiff}_{\rm{pt}}^{(\multiidxSec_2)}
\bigr|
\leq
2M 
\|\mat{A}(\prmtrVec_0)
(\mat{I}-2\mat{P}(\prmtrVec_0))
{\projDiff}_{\rm{pt}}^{(\multiidxSec_1)}\|
\cdot
\|\mat{P}^{(\multiidxSec_2)}\|.
\end{align*}
Since
$\mat{I}-2\mat{P}(\prmtrVec_0)$ is a unitary operator
commuting with $\mat{A}(\prmtrVec_0)$
and $\|\cdot\|$ is invariant under unitary transform,
we deduce $\|\mat{A}(\prmtrVec_0)
(\mat{I}-2\mat{P}(\prmtrVec_0))
{\projDiff}_{\rm{pt}}^{(\multiidxSec_1)}\|
\leq 
A^{(\Ord+1)}_{\rm{pt}}$
by \eqref{eq:def_A_pt_n}.
In addition, there are
$\binom{\Ord+\Pdim}{\Pdim-1}$ multi-indices of order $\Ord+1$,
and it holds that
$| (\prmtrVec-\prmtrVec_0)^{\multiidx} |
\leq|\prmtrVec-\prmtrVec_0|^{\Ord+1}$
for every such multi-index.
As a consequence of all these observations, 
the leading-order term in
\eqref{eq:leading_expansion_energy_difference} is bounded in absolute
value by
\[
2M
\binom{\Ord+\Pdim}{\Pdim-1}
A^{(\Ord+1)}_{\rm{pt}}
\!\!
\sum_{|\multiidx|=\Ord+1}
\!\!
\|\mat{P}^{(\multiidx)}\| \, 
|\prmtrVec-\prmtrVec_0|^{2\Ord+2}.
\]
Dividing \eqref{eq:leading_expansion_energy_difference} by
$|\prmtrVec-\prmtrVec_0|^{2\Ord+2}$ and taking the limit superior
proves \eqref{eq:leading_term_energy_difference}.
\end{proof}
}

To summarize, \Cref{subsec:truncated_projector_series} suggests 
that it is a good idea 
to approximate $\mat{P}(\prmtrVec)$  
within the range of $\mat{P}^{[\Ord]}(\prmtrVec)$.
This further motivates reducing the approximation problem 
into a lower-dimensional subspace containing the range of $\mat{P}^{[\Ord]}(\prmtrVec)$,
and leads naturally to the Taylor-reduced basis method.

\section{Taylor-Reduced Basis Method}\label{sec:Taylor_rbm}

{This section introduces the Taylor reduced basis method and analyzes its
finite-order convergence in
\Cref{subsec:taylor_rbm_convergence}, before discussing the practical
assembly of the reduced space in
\Cref{subsec:practical_assembly_taylor_rbm}.

For an orthogonal projector-valued mapping $\mat{Q}$, 
we introduce the
parametric notation
$\mat{A}\!\!\upharpoonright_{\mat{Q}}\!\!(\prmtrVec)
\coloneqq
\mat{Q}(\prmtrVec)\mat{A}(\prmtrVec)|_{\ran\mat{Q}(\prmtrVec)}$
and $\mat{Q}^{\perp}(\prmtrVec)\coloneqq
\mat{I}-\mat{Q}(\prmtrVec)$.
}

\subsection{\rev{The Taylor-RBM and its convergence analysis}}
\label{subsec:taylor_rbm_convergence}
{By \Cref{sec:analyticity}, the scaled derivatives 
of~$\mat{P}$ at $\prmtrVec_0$ are available up to order $\Ord+1$.
The finite Taylor approximation from \Cref{subsec:truncated_projector_series}
and the quasi-optimality result in \Cref{thm:Cea_EVP} motivate the
following method.

\begin{definition}[Taylor reduced basis method] \label{def:taylor_rbm}
For a given order $\Ord\in\bb{N}_0$, the \emph{Taylor reduced basis
space (Taylor-RB space)} at $\prmtrVec_0$ is defined as 
\[
\RBspace_{\!\Ord}
\coloneqq
\RBspace_{\!\Ord}(\prmtrVec_0)
\coloneqq
\spanop\bigl\{\ran\mat{P}^{(\multiidx)}:|\multiidx|\leq\Ord\bigr\}.
\]
We denote its orthogonal projector by ${\RBproj}_{\Ord}$ and set
$r_{\!\Ord}\coloneqq\dim\RBspace_{\!\Ord}$.

Let $\mat{V}_{\mkern-5mu\Ord}\in
\bb{C}^{\fdim\times r_{\!\Ord}}$ have orthonormal columns spanning
$\RBspace_{\!\Ord}$. The corresponding \emph{reduced order model
(ROM)} is given by 
\[
\mat{A}^{\!\!\mat{V}_{\mkern-5mu\Ord}}(\prmtrVec)
\coloneqq
\mat{V}_{\mkern-5mu\Ord}^{*}\mat{A}(\prmtrVec)
\mat{V}_{\mkern-5mu\Ord}
\in\bb{C}^{r_{\!\Ord}\times r_{\!\Ord}}.
\]

For $\prmtrVec$ near $\prmtrVec_0$, let
$\mat{Q}^{\!\mat{V}_{\mkern-5mu\Ord}}(\prmtrVec)$ satisfy
\begin{equation} \label{eq:taylor_rbm_reduced_problem_concrete}
\mat{Q}^{\!\mat{V}_{\mkern-5mu\Ord}}(\prmtrVec)
\, \in \!
\argmin_{\mat{M}\in\operatorname{Gr}(M,r_{\!\Ord})}
\tr \,
\mat{A}^{\!\!\mat{V}_{\mkern-5mu\Ord}}(\prmtrVec)\mat{M}
,
\end{equation}
and define
$\mat{Q}_{\rm{rb}}(\prmtrVec)
\coloneqq
\mat{V}_{\mkern-5mu\Ord}
\mat{Q}^{\!\mat{V}_{\mkern-5mu\Ord}}(\prmtrVec)
\mat{V}_{\mkern-5mu\Ord}^{*}.$
The matrix $\mat{Q}_{\rm{rb}}(\prmtrVec)$ is the \emph{Taylor reduced
basis approximation (Taylor-RB approximation)} of
$\mat{P}(\prmtrVec)$.
\end{definition}
Equivalently, $\mat{Q}_{\rm{rb}}(\prmtrVec)$ is the variational subspace approximation in
\eqref{eq:ky_fan_min_principle_subspace} with
${\RBproj}={\RBproj}_{\Ord}$. 
Its well-posedness and regularity are guaranteed by the following lemma. 

\begin{lemma}[Well-posedness and regularity of the Taylor-RBM]
\label{lem:taylor_rbm_well_posedness}
There exists a neighborhood of~$\prmtrVec_0$, 
in which the reduced problem \eqref{eq:taylor_rbm_reduced_problem_concrete} 
has a unique solution. 
Moreover,
$\mat{Q}_{\rm{rb}}(\prmtrVec_0)=\mat{P}(\prmtrVec_0)$ and
$\mat{Q}_{\rm{rb}}$ is of class $\mathrm{C}^{\Ord +1}$ in this
neighborhood.
\end{lemma}
}

\begin{proof}
If $r_{\!\Ord}=M$ for $M$ from \eqref{eq:target_projector}, 
we have $\mat{Q}_{\rm{rb}}(\prmtrVec)={\RBproj}_{\Ord}=\mat{P}(\prmtrVec_0)$ 
for every $\prmtrVec$, and the assertions follow.
Hence, we focus in the following on the case $r_{\!\Ord}>M$.

Consider the compression $\mat{A}\!\!\upharpoonright_{{\RBproj}_\Ord} \!\! (\prmtrVec)$. 
Since $\ran\,\mat{P}(\prmtrVec_0) = \ran \, \mat{P}^{(\vec{0})} \subseteq \RBspace_{\!\Ord}(\prmtrVec_0)$, 
$\EigVal_K(\prmtrVec_0) < \EigVal_{K+1}(\prmtrVec_0)$ and we are in the variational setting, 
$\mat{P}(\prmtrVec_0)$ uniquely solves
\eqref{eq:ky_fan_min_principle_subspace} for $\mat{A} = \mat{A}(\prmtrVec_0)$ and $\RBproj = \RBproj_\Ord$.  
In other words, 
$\mat{Q}_{\rm{rb}}(\prmtrVec_0)$ is uniquely determined 
with $\mat{Q}_{\rm{rb}}(\prmtrVec_0) = \mat{P}(\prmtrVec_0)$.    
This also implies that $\mat{Q}_\rm{rb}(\prmtrVec_0)$ and ${\RBproj}_\Ord \mat{Q}^\perp_\rm{rb}(\prmtrVec_0)$ 
are spectral projectors onto $\ran \mat{A}\!\!\upharpoonright_{{\RBproj}_\Ord} \!\!\!(\prmtrVec_0)$.
In particular, 
the distance between the eigenvalues of $\mat{A}\!\!\upharpoonright_{\mat{Q}_\rm{rb} }\!\!(\prmtrVec_0)$ 
and those of $\mat{A}\!\!\upharpoonright_{{\RBproj}_\Ord \mat{Q}^\perp_\rm{rb} }\!\!(\prmtrVec_0)$ is positive.

The previous considerations and the regularity 
of $\mat{A}\!\!\!\upharpoonright_{{\RBproj}_\Ord} \!\! (\prmtrVec)$ 
in $\prmtrVec$ around $\prmtrVec_0$ 
enable us to apply the perturbation results 
from \Cref{sec:analyticity}.
Consequently, 
the total eigenprojector associated with the first $M$ eigenvalues of $\mat{A}\!\!\upharpoonright_{{\RBproj}_\Ord} \!\! (\prmtrVec)$ 
is uniquely determined and of class $\mathrm{C}^{\Ord +1}$ for $\prmtrVec$ in a small neighborhood of $\prmtrVec_0$, 
which is exactly $\mat{Q}_{\rm{rb}}(\prmtrVec)$.
\end{proof}

{We next analyze the approximation error of the Taylor-RBM.
As in \Cref{thm:Cea_EVP}, we define for~$\prmtrVec$ close to $\prmtrVec_0$
\begin{equation}\label{eq:def_residual_pi_n}
\rm{Res}_{}^2\!(\prmtrVec,{\RBproj}_\Ord)
\coloneqq
\bigl\|{\RBproj}_\Ord\mat{A}(\prmtrVec)
{\RBproj}_\Ord^\perp\bigr\|^2.
\end{equation}
This residual measures the deviation of the entire Taylor-RB space
from an invariant subspace of~$\mat{A}(\prmtrVec)$.
For the case $r_{\!\Ord}>M$, we set 
\begin{equation}\label{eq:def_gamma}
\begin{aligned}
\gamma(\prmtrVec)
\coloneqq
\min\bigl\{
|{\EigVal}-\widehat{\EigVal}|:
{}{\EigVal}\in
\sigma\bigl(
\mat{A}\!\!\upharpoonright_{\mat{P}}\!\!(\prmtrVec)
\bigr),
\widehat{\EigVal}\in
\sigma\bigl(
\mat{A}\!\!\upharpoonright_{({\RBproj}_\Ord-\mat{Q}_{\rm{rb}}) }
\!\!(\prmtrVec)
\bigr)
\bigr\}.
\end{aligned}
\end{equation}
{By \Cref{lem:taylor_rbm_well_posedness} and the min--max principle, the
spectrum associated with $\mat{P}(\prmtrVec)$ lies below that associated with
${\RBproj}_\Ord-\mat{Q}_{\rm{rb}}(\prmtrVec)$. 
Therefore, 
for any fixed $\prmtrVec$ close to $\prmtrVec_0$,
the gap in \eqref{eq:def_gamma} 
agrees with that in \Cref{thm:Cea_EVP} and is positive.
Continuity of the eigenvalues implies that $\gamma$ is continuous in a
sufficiently small neighborhood of~$\prmtrVec_0$.}
Let us finally also define the quasi-optimality factor
\begin{equation}\label{eq:def_C_rb}
C_{\rm{rb}}(\prmtrVec)
\coloneqq
\begin{cases}
\qquad \qquad \qquad 1
& \text{if } r_{\!\Ord}=M,
\\[0.3em]
\displaystyle
\sqrt{
1+
{\gamma(\prmtrVec)^{-2}}{\rm{Res}_{}^2\!(\prmtrVec,{\RBproj}_\Ord)} 
}
&  \text{if } r_{\!\Ord}>M.
\end{cases}
\end{equation}
In either case, $C_{\rm{rb}}$ is continuous near $\prmtrVec_0$.

\begin{theorem}[A priori total eigenprojector error bound of the Taylor-RBM]
\label{thm:Taylor_rbm_projector_error}
For $\prmtrVec$ sufficiently close to $\prmtrVec_0$, it holds that
\begin{equation}
    \bigl\|\mat{Q}_{\rm{rb}}(\prmtrVec)-\mat{P}(\prmtrVec)\bigr\|
\, \leq \, 
C_{\rm{rb}}(\prmtrVec)
\bigl\|{\RBproj}_\Ord^\perp\mat{P}(\prmtrVec)\bigr\|.
\label{eq:Q_rb_ineq}
\end{equation}
In particular, we have  
$\|\mat{Q}_{\rm{rb}}(\prmtrVec)-\mat{P}(\prmtrVec)\|
=\bigO(|\prmtrVec-\prmtrVec_0|^{\Ord+1})$
as $\prmtrVec\to\prmtrVec_0$,
and more quantitatively, 
\begin{equation}\label{eq:leading_projector_error_rb}
\begin{aligned}
\limsup_{\prmtrVec\to\prmtrVec_0}
\frac{
\bigl\|\mat{Q}_{\rm{rb}}(\prmtrVec)-\mat{P}(\prmtrVec)\bigr\|}
{|\prmtrVec-\prmtrVec_0|^{\Ord+1}}
\, \leq \,
C_{\rm{rb}}(\prmtrVec_0)
\!\!\!\!
\sum_{|\multiidx|=\Ord+1}
\!\!
\bigl\|{\RBproj}_\Ord^\perp\mat{P}^{(\multiidx)}\bigr\|.
\end{aligned}
\end{equation}
\end{theorem}

Thus, the Taylor-RBM error is controlled by the projection error of
$\mat{P}(\prmtrVec)$ onto the Taylor-RB space.
In addition, 
$C_{\rm{rb}}(\prmtrVec_0)$ quantifies the
quasi-optimality loss caused by the Rayleigh--Ritz extraction, whereas
the sum involves only the components of the first neglected derivatives
outside~$\RBspace_\Ord$.

\begin{proof}
In the case $r_{\!\Ord}=M$, 
we have $\mat{Q}_{\rm{rb}}={\RBproj}_\Ord$ and
\Cref{lem:proj_diff_equiv} gives
$\bigl\|\mat{Q}_{\rm{rb}}(\prmtrVec)-\mat{P}(\prmtrVec)\bigr\|
=
\bigl\|{\RBproj}_\Ord^\perp\mat{P}(\prmtrVec)\bigr\|.$
If $r_{\!\Ord}>M$, the asserted inequality \eqref{eq:Q_rb_ineq} 
follows from \Cref{lem:proj_diff_equiv}, 
\Cref{thm:Cea_EVP} and the definitions
\eqref{eq:def_residual_pi_n}--\eqref{eq:def_C_rb}.

Let us now prove the asymptotic be havior 
$\|\mat{Q}_{\rm{rb}}(\prmtrVec)-\mat{P}(\prmtrVec)\|
=\bigO(|\prmtrVec-\prmtrVec_0|^{\Ord+1})$
as $\prmtrVec\to\prmtrVec_0$
and \eqref{eq:leading_projector_error_rb}. 
By the definition of $\RBspace_{\!\Ord}$, we have
${\RBproj}_\Ord^\perp\mat{P}^{(\multiidx)}=\mat{0}$ for
$|\multiidx|\leq\Ord$.
Combining this 
with the finite Taylor expansion 
of $\mat{P}$ 
gives
\[
{\RBproj}_\Ord^\perp\mat{P}(\prmtrVec)
\, =
\!\!
\sum_{|\multiidx|=\Ord+1}
\!\!
(\prmtrVec-\prmtrVec_0)^{\multiidx}
{\RBproj}_\Ord^\perp\mat{P}^{(\multiidx)}
+\smallO\bigl(|\prmtrVec-\prmtrVec_0|^{\Ord+1}\bigr).
\]
Then, \Cref{thm:Cea_EVP} and the continuity of $C_{\rm{rb}}$
yield the desired asymptotic behavior. 
Finally, using
$|(\prmtrVec-\prmtrVec_0)^{\multiidx}|
\leq|\prmtrVec-\prmtrVec_0|^{\Ord+1}$
for $|\multiidx|=\Ord+1$, dividing by
$|\prmtrVec-\prmtrVec_0|^{\Ord+1}$, and taking the limit superior
yield \eqref{eq:leading_projector_error_rb}.
\end{proof}
}

{Next, we compare the total Ritz-value sums associated with
$\mat{P}(\prmtrVec)$ and $\mat{Q}_{\rm{rb}}(\prmtrVec)$.
For this purpose, we set
${\projDiff}_{\rm{rb}}(\prmtrVec)
\coloneqq
\mat{Q}_{\rm{rb}}(\prmtrVec)-\mat{P}(\prmtrVec)$,
whose scaled derivatives at $\prmtrVec_0$ are
${\projDiff}_{\rm{rb}}^{\!(\multiidx)}
=
\mat{Q}_{\rm{rb}}^{(\multiidx)}-\mat{P}^{(\multiidx)}$
for $|\multiidx|\leq\Ord+1$.
Let us first collect some properties 
of the derivatives of ${\projDiff}_{\rm{rb}}$ at $\prmtrVec_0$. 

\begin{lemma}[First nonmatching derivatives]
\label{lem:first_defect_Q_rb}
The scaled derivatives of
${\projDiff}_{\rm{rb}}=\mat{Q}_{\rm{rb}}-\mat{P}$ at $\prmtrVec_0$ 
satisfy
\begin{equation}\label{eq:lower_derivatives_Q_rb_P_agree}
{{\projDiff}_{\rm{rb}}^{\!(\multiidxSec)}=\mat{0}
\qquad\text{for }|\multiidxSec|\leq\Ord.}
\end{equation}
For every $|\multiidx|=\Ord+1$,
${\projDiff}_{\rm{rb}}^{\!(\multiidx)}$ is off-diagonal with respect to
$\ran\mat{P}(\prmtrVec_0)\oplus
\ran\mat{P}^{\perp}(\prmtrVec_0)$ and satisfies
\begin{equation}\label{eq:first_defect_Q_rb_norm_rank}
\rank {\projDiff}_{\rm{rb}}^{\!(\multiidx)}\leq2M
\qquad\text{and}\qquad
\|{\projDiff}_{\rm{rb}}^{\!(\multiidx)}\|
\leq
C_{\rm{rb}}(\prmtrVec_0)
\|{\RBproj}_\Ord^\perp\mat{P}^{(\multiidx)}\|.
\end{equation}
\end{lemma}
{The off-diagonal structure and rank bounds in this lemma and
\Cref{lem:first_defect_Q_pt} are both consequences of
\Cref{lem:first_defect_projector_mappings};
for a differential-geometric interpretation,
we refer to \Cref{rem:tangent_projector_mappings}.}
The norm bounds differ: \Cref{lem:first_defect_Q_pt} involves the full
derivative $\mat{P}^{(\multiidx)}$, whereas the Taylor-RBM bound
involves only its component orthogonal to $\RBspace_\Ord$, together with the
quasi-optimality factor $C_{\rm{rb}}(\prmtrVec_0)$ arising from the
Rayleigh--Ritz extraction.
\begin{proof}
By \Cref{thm:Taylor_rbm_projector_error},
${\projDiff}_{\rm{rb}}(\prmtrVec)
=\bigO(|\prmtrVec-\prmtrVec_0|^{\Ord+1})$,
so Taylor's theorem yields
\eqref{eq:lower_derivatives_Q_rb_P_agree}.
{For $|\multiidx|=\Ord+1$, applying
\Cref{lem:first_defect_projector_mappings} with
$\mat{Q}=\mat{Q}_{\rm{rb}}$ gives the stated off-diagonal structure
and rank bound.}

It remains to prove the norm bound.
Set
$\mat{E}_{\mathrm{out}}(\prmtrVec)
\coloneqq {\RBproj}_\Ord^\perp\mat{P}(\prmtrVec)$
and 
$\mat{E}_{\mathrm{in}}(\prmtrVec)
\coloneqq
\bigl({\RBproj}_\Ord-\mat{Q}_{\rm{rb}}(\prmtrVec)\bigr)
\mat{P}(\prmtrVec)$. 
Observe that 
\begin{equation}\label{eq:first_defect_Q_rb_decomposition}
-{\projDiff}_{\rm{rb}}(\prmtrVec)\mat{P}(\prmtrVec)
=
\mat{E}_{\mathrm{out}}(\prmtrVec)
+\mat{E}_{\mathrm{in}}(\prmtrVec).
\end{equation}
By the definition of $\RBspace_\Ord$, 
the scaled derivatives of
$\mat{E}_{\mathrm{out}}$ 
at $\prmtrVec_0$ 
vanish up to order $\Ord$.
Together with \eqref{eq:lower_derivatives_Q_rb_P_agree} and
\eqref{eq:first_defect_Q_rb_decomposition}, 
the scaled derivatives 
of $\mat{E}_{\mathrm{in}}$ 
at $\prmtrVec_0$ 
also vanish up to order $\Ord$. 
Therefore, for $|\multiidx|=\Ord+1$,
taking the scaled derivative of order $\multiidx$
at $\prmtrVec_0$ 
in \eqref{eq:first_defect_Q_rb_decomposition} gives
\begin{equation}\label{eq:first_defect_Q_rb_decomposition_coefficient}
-{\projDiff}_{\rm{rb}}^{\!(\multiidx)}\mat{P}(\prmtrVec_0)
=
{\RBproj}_\Ord^\perp\mat{P}^{(\multiidx)}
+\mat{E}_{\mathrm{in}}^{(\multiidx)}.
\end{equation}

If $r_{\!\Ord}=M$, then
$\mat{E}_{\mathrm{in}}=\mat{0}$, 
and the norm estimate in \eqref{eq:first_defect_Q_rb_norm_rank} 
follows directly from
\eqref{eq:first_defect_Q_rb_decomposition_coefficient} and
$C_{\rm{rb}}(\prmtrVec_0)=1$ in this case. 

Suppose now $r_{\!\Ord}>M$.
To bound $\mat{E}_{\mathrm{in}}^{(\multiidx)}$
with $|\multiidx|=\Ord+1$, 
we consider a Sylvester equation satisfied by~$\mat{E}_{\mathrm{in}}$,
which is derived similarly as in the proof of \Cref{thm:Cea_EVP} 
and reads as
\[
\mat{A}\!\!\upharpoonright_{({\RBproj}_\Ord-\mat{Q}_{\rm{rb}})}
\!\!(\prmtrVec)\mat{E}_{\mathrm{in}}(\prmtrVec)
-\mat{E}_{\mathrm{in}}(\prmtrVec)
\mat{A}\!\!\upharpoonright_{\mat{P}}\!\!(\prmtrVec)
=
-\bigl({\RBproj}_\Ord-\mat{Q}_{\rm{rb}}(\prmtrVec)\bigr)
\mat{A}(\prmtrVec)\mat{E}_{\mathrm{out}}(\prmtrVec).
\]
Differentiating it at $\prmtrVec_0$ with respect to order $\multiidx$,
using the vanishing lower-order
derivatives of $\mat{E}_{\mathrm{out}}$ and $\mat{E}_{\mathrm{in}}$,
and observing $\mat{Q}_{\rm{rb}}(\prmtrVec_0)=\mat{P}(\prmtrVec_0)$,
we arrive at 
\[
\mat{A}\!\!\upharpoonright_{ ( {\RBproj}_\Ord-\mat{P} )}
\!\!(\prmtrVec_0)\mat{E}_{\mathrm{in}}^{(\multiidx)}
-\mat{E}_{\mathrm{in}}^{(\multiidx)}
\mat{A}\!\!\upharpoonright_{\mat{P} }\!\!(\prmtrVec_0)
=
-\bigl({\RBproj}_\Ord-\mat{P}(\prmtrVec_0)\bigr)
\mat{A}(\prmtrVec_0)
{\RBproj}_\Ord^\perp\mat{P}^{(\multiidx)}.
\]
Therefore, the same Sylvester error bound argument and the definition of $\gamma(\prmtrVec_0)$ implies
\begin{equation}\label{eq:first_defect_Q_rb_in_bound}
\|\mat{E}_{\mathrm{in}}^{(\multiidx)}\|
\leq
\frac{
\|{\RBproj}_\Ord\mat{A}(\prmtrVec_0)
{\RBproj}_\Ord^\perp\|}
{\gamma(\prmtrVec_0)}
\|{\RBproj}_\Ord^\perp\mat{P}^{(\multiidx)}\|.
\end{equation}
This observation, 
together with {the norm identity in
\Cref{lem:first_defect_projector_mappings}},
\eqref{eq:first_defect_Q_rb_decomposition_coefficient},
\eqref{eq:first_defect_Q_rb_in_bound} 
and the definition of $C_{\mathrm{rb}}(\prmtrVec_0)$,
yields 
\[
\|{\projDiff}_{\rm{rb}}^{\!(\multiidx)}\|^2
=
\|\projDiff_{\rm{rb}}^{\!(\multiidx)}\mat{P}(\prmtrVec_0)\|^2
\leq
\|{\RBproj}_\Ord^\perp\mat{P}^{(\multiidx)}\|^2
+
\|\mat{E}_{\mathrm{in}}^{(\multiidx)}\|^2
\leq
C_{\rm{rb}}(\prmtrVec_0)^2
\|{\RBproj}_\Ord^\perp\mat{P}^{(\multiidx)}\|^2,
\]
which proves the norm estimate in
\eqref{eq:first_defect_Q_rb_norm_rank}.
\end{proof}
}

{To quantify the leading-order Ritz-value error, we define
\begin{equation}\label{eq:def_A_rb_n}
A^{(\Ord+1)}_{\rm{rb}}
\coloneqq
\max_{|\multiidx|=\Ord+1}
\|
\mat{A}(\prmtrVec_0)
{\projDiff}_{\rm{rb}}^{\!(\multiidx)}
\| ,
\end{equation}
which is the Taylor-RBM counterpart of the quantity
$A_{\rm{pt}}^{(\Ord+1)}$ in \eqref{eq:def_A_pt_n}.

\begin{theorem}[Total Ritz-value error of the Taylor-RBM]
\label{thm:energy_difference_rb}
It holds that 
\[
\tr\,
\mat{A}(\prmtrVec)
(\mat{Q}_{\rm{rb}}(\prmtrVec)-\mat{P}(\prmtrVec))
\, = \, 
\bigO\bigl(|\prmtrVec-\prmtrVec_0|^{2\Ord+2}\bigr)
\qquad\text{as }\prmtrVec\to\prmtrVec_0.
\]
More quantitatively, it holds that
\begin{equation}\label{eq:leading_term_energy_difference_rb}
\begin{aligned}
\limsup_{\prmtrVec\to\prmtrVec_0}
\frac{
\left|
\tr\,
\mat{A}(\prmtrVec)
(\mat{Q}_{\rm{rb}}(\prmtrVec)-\mat{P}(\prmtrVec))
\right|}
{|\prmtrVec-\prmtrVec_0|^{2\Ord+2}}
\,\leq\,
2M
\binom{\Ord+\Pdim}{\Pdim-1}
A^{(\Ord+1)}_{\rm{rb}}
C_{\rm{rb}}(\prmtrVec_0)
\!\!\sum_{|\multiidx|=\Ord+1}
\!\!
\|{\RBproj}_\Ord^\perp\mat{P}^{(\multiidx)}\|.
\end{aligned}
\end{equation}
\end{theorem}
}

{As in \Cref{thm:energy_difference_truncation}, the bound contains
no global factor $\|\mat{A}(\prmtrVec_0)\|$.
The factors $A_{\rm{rb}}^{(\Ord+1)}$ and
$C_{\rm{rb}}(\prmtrVec_0)
\sum_{|\multiidx|=\Ord+1}
\|{\RBproj}_\Ord^\perp\mat{P}^{(\multiidx)}\|$
control, respectively, the energy-weighted and remaining projector
defects in the quadratic Ritz-value error.}

{\begin{proof}
The finite Taylor expansion of ${\projDiff}_{\rm{rb}}$ and
\Cref{lem:energy_difference_general} give the asserted asymptotic order
as in the proof of \Cref{thm:energy_difference_truncation}.
The only difference in the quantitative estimate here is that, 
for $|\multiidxSec_1|=|\multiidxSec_2|=\Ord+1$,
the matrix inequalities
$|\tr \mat{M}_1\mat{M}_2 |
\leq \rank \mat{M}_1 \mat{M}_2 \cdot
\|\mat{M}_1\|\cdot\|\mat{M}_2\|$
and  $\rank \mat{M}_1\mat{M}_2 \leq \rank \mat{M}_2$
applied to $\mat{M}_1 \coloneqq \mat{A}(\prmtrVec_0)
(\mat{I}-2\mat{P}(\prmtrVec_0))
{\projDiff}_{\rm{rb}}^{\!(\multiidxSec_1)}$
and $\mat{M}_2 \coloneqq {\projDiff}_{\rm{rb}}^{\!(\multiidxSec_2)}$,
the fact that $\mat{I}-2\mat{P}(\prmtrVec_0)$ is unitary 
and commutes with $\mat{A}(\prmtrVec_0)$,
the definition of~$A^{(\Ord+1)}_{\rm{rb}}$ in \eqref{eq:def_A_rb_n}, and
\eqref{eq:first_defect_Q_rb_norm_rank} now yield
\[
\begin{aligned}
\bigl|
\tr\,
\mat{A}(\prmtrVec_0)
(\mat{I}-2\mat{P}(\prmtrVec_0))
{\projDiff}_{\rm{rb}}^{\!(\multiidxSec_1)}
{\projDiff}_{\rm{rb}}^{\!(\multiidxSec_2)}
\bigr|
\leq
2M A_{\rm{rb}}^{(\Ord+1)}
C_{\rm{rb}}(\prmtrVec_0)
\|{\RBproj}_\Ord^\perp
\mat{P}^{(\multiidxSec_2)}\|.
\end{aligned}
\]
\eqref{eq:leading_term_energy_difference_rb} then follows 
from the same multi-index count and limit-superior argument.
\end{proof}
}

\rev{Let us close this section with two remarks on the scope of the preceding
construction and convergence results.}

\begin{remark}[Extended subspaces] \label{rem:extended_subspaces}
\rev{The Taylor-RB construction is not tied to the specific space
$\RBspace_{\!\Ord}(\prmtrVec_0)$:} 
the convergence results in
\Cref{thm:Taylor_rbm_projector_error,thm:energy_difference_rb}
remain valid if $\RBspace_{\!\Ord}(\prmtrVec_0)$ is replaced by any fixed
subspace containing it, with the corresponding orthogonal and Ritz
projectors. Thus, the Taylor-RB space may be enriched by spaces obtained
from other model reduction techniques.
\end{remark}

\rev{\begin{remark}[Analytic problem data] \label{rem:analytic_problem_data}
The preceding convergence analysis requires only
$\mathrm{C}^{\Ord+1}$-regularity of $\mat{A}$,
so analyticity is not needed for a fixed order.
If, however, $\mat{A}$ is real-analytic near $\prmtrVec_0$,
then it is of class $\mathrm{C}^{\Ord+1}$ for every $\Ord$,
and the Taylor-RB construction and the error estimates in
\Cref{thm:Taylor_rbm_projector_error,thm:energy_difference_rb}
are available at every finite order.
\end{remark}}

\subsection{{Practical assembly}} \label{subsec:practical_assembly_taylor_rbm}

{By \Cref{rem:multiple_eigenvalues}, the contributions of distinct
eigenvalues can be constructed separately. 
Therefore, we fix one
eigenvalue $\EigVal_0$ of $\mat{A}_0\coloneqq\mat{A}(\prmtrVec_0)$ with
multiplicity $m_0$. 
\rev{Let $\mat{P}_{\!0}(\prmtrVec)$ be the corresponding
$\mathrm{C}^{\Ord +1}$ eigenprojector} 
and set $\mat{P}_{\!0}\coloneqq\mat{P}_{\!0}(\prmtrVec_0)$.
Choose $\mat{U}_0\in\bb{C}^{\fdim\times m_0}$ with orthonormal columns
spanning $\ran\mat{P}_{\!0}$. 
Then, we have 
\[
\mat{U}_0^*\mat{U}_0=\mat{I}_{m_0},\qquad
\mat{P}_{\!0}=\mat{U}_0\mat{U}_0^*,\qquad
\mat{A}_0\mat{U}_0=\EigVal_0\mat{U}_0.
\]
Our \emph{goal} is to construct
$\spanop\{\ran\mat{P}_{\!0}^{(\multiidx)}:|\multiidx|\leq\Ord\}$
from $\EigVal_0$, $\mat{U}_0$, and the scaled derivatives of $\mat{A}$,
without forming the dense projector derivatives. The spaces obtained for
the individual eigenvalues are then combined to form the Taylor-RB space.

\subsubsection{\rev{Recursive assembly for one eigenvalue}}

The recursion in \Cref{prop:multivariate_DMPT} involves dense
$\fdim\times\fdim$ matrices. The following local graph parametrisation
replaces them by thin $\fdim\times m_0$ matrices.
}

{\begin{lemma}[Subspace parametrisation and problem reduction]
\label{lem:projector_parametrization}
For $\prmtrVec$ near $\prmtrVec_0$, define
\begin{equation} \label{eq:def_Vmu_Cmu}
\mat{U}(\prmtrVec)
\coloneqq\mat{P}_{\!0}(\prmtrVec)\mat{U}_0
\quad \text{and} \quad 
\mat{C}(\prmtrVec)
\coloneqq\mat{U}_0^*\mat{P}_{\!0}(\prmtrVec)\mat{U}_0.
\end{equation}
After possibly reducing the neighborhood, the following assertions hold.
\begin{enumerate}
\item \label{item:Cmu_inv}
The mappings $\mat{U}$ and $\mat{C}$ are of class
$\mathrm{C}^{\Ord +1}$, and $\mat{C}(\prmtrVec)$ is invertible.

\item \label{item:ran_Pmu_Xmu}
The restrictions
\[
\mat{P}_{\!0}(\prmtrVec)|_{\ran\mat{P}_{\!0}}:
\ran\mat{P}_{\!0}\to\ran\mat{P}_{\!0}(\prmtrVec)
\quad\text{and}\quad
\mat{P}_{\!0}|_{\ran\mat{P}_{\!0}(\prmtrVec)}:
\ran\mat{P}_{\!0}(\prmtrVec)\to\ran\mat{P}_{\!0}
\]
are invertible.

\item \label{item:ranP0muP0_equals_ranP0mu}
The ranges satisfy
\begin{equation}\label{eq:range_U_P0}
\ran\mat{P}_{\!0}(\prmtrVec)
=\ran\mat{U}(\prmtrVec)
=\ran \mat{P}_{\!0}(\prmtrVec)\mat{U}_0.
\end{equation}

\item \label{item:P0mu_in_terms_of_Umu}
The eigenprojector and the invariant relation can be written as
\begin{align}
\mat{P}_{\!0}(\prmtrVec)
&=\mat{U}(\prmtrVec)\mat{C}(\prmtrVec)^{-1}
\mat{U}(\prmtrVec)^*,
\label{eq:P0mu_in_terms_of_Umu}\\
\mat{A}(\prmtrVec)\mat{U}(\prmtrVec)
&=\mat{U}(\prmtrVec)\mat{C}(\prmtrVec)^{-1}
\mat{U}_0^*\mat{A}(\prmtrVec)\mat{U}(\prmtrVec).
\label{eq:useful_identity_A_Vmu}
\end{align}

\item \label{item:Kmu_definition_first_prop}
The graph matrix
\begin{equation} \label{eq:def_Kmu}
\mat{K}(\prmtrVec)
\coloneqq
\mat{P}_{\!0}^{\perp}\mat{U}(\prmtrVec)
\mat{C}(\prmtrVec)^{-1}
\in\bb{C}^{\fdim\times m_0}
\end{equation}
is of class $\mathrm{C}^{\Ord +1}$ and satisfies
\begin{equation}\label{eq:graph_parametrization}
\mat{U}(\prmtrVec)\mat{C}(\prmtrVec)^{-1}
=\mat{U}_0+\mat{K}(\prmtrVec)
\quad \text{and} \quad 
\ran\mat{P}_{\!0}(\prmtrVec)
=\ran\bigl(\mat{U}_0+\mat{K}(\prmtrVec)\bigr).
\end{equation}
Moreover, $\mat{K}(\prmtrVec_0)=\mat{0}$ and
$\mat{U}_0^*\mat{K}(\prmtrVec)=\mat{0}$.

\item \label{item:define_B_useful_identity}
The effective matrix
\begin{equation}\label{eq:def_Bmu}
\mat{B}(\prmtrVec)
\coloneqq
\mat{U}_0^*\mat{A}(\prmtrVec)\mat{U}(\prmtrVec)
\mat{C}(\prmtrVec)^{-1}
\in\bb{C}^{m_0\times m_0}
\end{equation}
is of class $\mathrm{C}^{\Ord +1}$ and satisfies
\begin{equation} \label{eq:matrix_equation_A_V_C}
\mat{A}(\prmtrVec)\bigl(\mat{U}_0+\mat{K}(\prmtrVec)\bigr)
=\bigl(\mat{U}_0+\mat{K}(\prmtrVec)\bigr)\mat{B}(\prmtrVec).
\end{equation}
Moreover, $\mat{B}(\prmtrVec_0)=\EigVal_0\mat{I}_{m_0}$.
\end{enumerate}
\end{lemma}
}

{\begin{remark}[Connection to many-body perturbation theory]
In many-body perturbation theory,
$\mat{U}_0+\mat{K}(\prmtrVec)$ is an intermediate normalization,
$\mat{K}(\prmtrVec)$ is a correction operator, and
$\mat{B}(\prmtrVec)$ is an effective Hamiltonian. The invariant relation
\eqref{eq:matrix_equation_A_V_C} is the corresponding finite-dimensional
wave-operator equation, 
cf.~\cite{HubW10,DEHKO04}.
\end{remark}
}

{\begin{proof}[Proof of \Cref{lem:projector_parametrization}]
The regularity of $\mat{U}(\prmtrVec)$ and $\mat{C}(\prmtrVec)$ follows from that of
$\mat{P}_{\!0}(\prmtrVec)$. 
Because of $\mat{C}(\prmtrVec_0)=\mat{I}_{m_0}$, 
continuity gives the local invertibility of $\mat{C}(\prmtrVec)$.

The map
$\mat{P}_{\!0}(\prmtrVec)|_{\ran\mat{P}_{\!0}}$ is injective:
if $\vec{x}=\mat{U}_0\vec{z}$ and
$\mat{P}_{\!0}(\prmtrVec)\vec{x}=\vec{0}$, then
$\mat{C}(\prmtrVec)\vec{z}=\vec{0}$ and so $\vec{x}=\vec{0}$.
Its domain and codomain both have dimension $m_0$, so it is invertible.
For $\vec{x}\in\ran\mat{P}_{\!0}$ and
$\vec{y}\in\ran\mat{P}_{\!0}(\prmtrVec)$, we have
$\langle\mat{P}_{\!0}(\prmtrVec)\vec{x},\vec{y}\rangle
=\langle\vec{x},\vec{y}\rangle
=\langle\vec{x},\mat{P}_{\!0}\vec{y}\rangle.$
Thus, 
$(\mat{P}_{\!0}(\prmtrVec)|_{\ran\mat{P}_{\!0}})^*
=
\mat{P}_{\!0}|_{\ran\mat{P}_{\!0}(\prmtrVec)}$ 
is also invertible. 
This proves \ref{item:ran_Pmu_Xmu} and
\eqref{eq:range_U_P0}.

Furthermore,
$\mat{U}(\prmtrVec)^*\mat{U}(\prmtrVec)=\mat{C}(\prmtrVec)$.
The standard formula for the orthogonal projector onto the range of a
full-column-rank matrix gives \eqref{eq:P0mu_in_terms_of_Umu}.
Since $\mat{A}(\prmtrVec)\mat{U}(\prmtrVec)$ belongs to
$\ran\mat{U}(\prmtrVec)$ and
$\mat{U}(\prmtrVec)\mat{C}(\prmtrVec)^{-1}\mat{U}_0^*$ acts as the
identity on this range, we obtain \eqref{eq:useful_identity_A_Vmu}.

The identity $\mat{U}_0^*\mat{U}(\prmtrVec)=\mat{C}(\prmtrVec)$ yields
$\mat{P}_{\!0}\mat{U}(\prmtrVec)\mat{C}(\prmtrVec)^{-1}=\mat{U}_0$.
Splitting with $\mat{I}=\mat{P}_{\!0}+\mat{P}_{\!0}^{\perp}$ proves
\eqref{eq:graph_parametrization}. The remaining assertions about
$\mat{K}$ follow directly from \eqref{eq:def_Kmu}.

Finally, the regularity of $\mat{B}$ follows from \eqref{eq:def_Bmu}.
Multiplying \eqref{eq:useful_identity_A_Vmu} from the right by
$\mat{C}(\prmtrVec)^{-1}$ and using
\eqref{eq:graph_parametrization} gives \eqref{eq:matrix_equation_A_V_C}.
At $\prmtrVec_0$, \eqref{eq:def_Bmu} gives
$\mat{B}(\prmtrVec_0)=\mat{U}_0^*\mat{A}_0\mat{U}_0
=\EigVal_0\mat{I}_{m_0}$.
\end{proof}
}

{The preceding lemma represents 
for one thing the varying invariant subspace by the thin
matrix $\mat{K}(\prmtrVec)\in\bb{C}^{\fdim\times m_0}$, 
and for another the action of
$\mat{A}(\prmtrVec)$ on it by the small matrix
$\mat{B}(\prmtrVec)\in\bb{C}^{m_0\times m_0}$.
By the properties established there,
it holds in particular 
$\mat{K}^{(\vec{0})}=\mat{0}$ and
$\mat{U}_0^*\mat{K}^{(\multiidx)}=\mat{0}$ for
$|\multiidx|\leq\Ord+1$, as well as
$\mat{B}^{(\vec{0})}=\EigVal_0\mat{I}_{m_0}$.
}

{\begin{theorem}[Connection between $\mat{P}_{\!0}^{(\multiidx)}$ and
$\mat{K}^{(\multiidx)}$] \label{thm:connection_P0_Y}
The scaled derivatives at $\prmtrVec_0$ satisfy
\begin{equation} \label{eq:proj_range_relation}
\begin{aligned}
\spanop\bigl\{\ran\mat{P}_{\!0}^{(\multiidx)}:
|\multiidx|\leq\Ord\bigr\}
&=
\spanop\bigl\{\ran \mat{P}_{\!0}^{(\multiidx)}\mat{U}_0 
:
|\multiidx|\leq\Ord\bigr\}\\
&=
\ran\mat{U}_0
+
\spanop\bigl\{\ran\mat{K}^{(\multiidx)}:
1\leq|\multiidx|\leq\Ord\bigr\}.
\end{aligned}
\end{equation}
\end{theorem}
}

{Thus, the desired space can be assembled 
from thin $\fdim\times m_0$ matrices 
instead of the generally dense scaled derivatives. 
A recursion for these thin matrices 
is derived later in \Cref{thm:resursive_Kmu}.

\begin{remark}[Dimension bound] \label{rem:upper_bound_dim_RB}
Since each generator on the last line of \eqref{eq:proj_range_relation}
has at most $m_0$ columns,
we obtain 
\[
\dim\spanop\bigl\{\ran\mat{P}_{\!0}^{(\multiidx)}:
|\multiidx|\leq\Ord\bigr\}
\, \leq \,
m_0\sum_{k=0}^{\Ord}\binom{k+\Pdim-1}{\Pdim-1}
=m_0\binom{\Ord+\Pdim}{\Pdim},
\]
where the last equality follows from the hockey-stick identity.
\end{remark}
}

{\begin{proof}[Proof of \Cref{thm:connection_P0_Y}]
Recall from \eqref{eq:P0mu_in_terms_of_Umu} that 
$\mat{P}_{\!0}(\prmtrVec)
=\mat{U}(\prmtrVec)\mat{C}(\prmtrVec)^{-1}
\mat{U}(\prmtrVec)^*$.
Applying the scaled Leibniz rule \eqref{eq:scaled_multivariate_leibniz}
to this equality gives
\[
\mat{P}_{\!0}^{(\multiidx)}
\, =
\!\!\!
\sum_{\multiidxSec_1+\multiidxSec_2+\multiidxSec_3=\multiidx}
\!\!\!
\mat{U}^{(\multiidxSec_1)}
\bigl(\mat{C}^{-1}\bigr)^{(\multiidxSec_2)}
\bigl(\mat{U}^*\bigr)^{(\multiidxSec_3)}
\qquad 
\text{for }
|\multiidx|\leq\Ord. 
\]
Since $\mat{U}(\prmtrVec)=\mat{P}_{\!0}(\prmtrVec)\mat{U}_0$, we have
$\mat{U}^{(\multiidxSec_1)}
=\mat{P}_{\!0}^{(\multiidxSec_1)}\mat{U}_0$.
Thus, the range of each summand is contained in
$\ran \mat{P}_{\!0}^{(\multiidxSec_1)}\mat{U}_0 $, and
so $\ran\mat{P}_{\!0}^{(\multiidx)}
\subseteq
\spanop\bigl\{\ran 
\mat{P}_{\!0}^{(\multiidxSec)}\mat{U}_0 :
\multiidxSec\leq\multiidx\bigr\}.$
Conversely, right multiplication by~$\mat{U}_0$ cannot enlarge the range,
so
$\ran \mat{P}_{\!0}^{(\multiidx)}\mat{U}_0 
\subseteq\ran\mat{P}_{\!0}^{(\multiidx)}$.
Taking the span over $|\multiidx|\leq\Ord$ proves the reverse inclusion,
and hence  the first identity in \eqref{eq:proj_range_relation}.

For the second identity, fix $\multiidx$ with $|\multiidx|\leq\Ord$
and take the scaled derivative of order $\multiidx$ at $\prmtrVec_0$ in
\[
\mat{P}_{\!0}(\prmtrVec)\mat{U}_0
=\mat{U}(\prmtrVec)
=\bigl(\mat{U}_0+\mat{K}(\prmtrVec)\bigr)\mat{C}(\prmtrVec).
\]
Then, we use $\mat{C}^{(\vec{0})}=\mat{I}_{m_0}$ and
$\mat{K}^{(\vec{0})}=\mat{0}$ to get 
\begin{equation}\label{eq:P0muU_in_terms_of_Kmu_Cmu}
\mat{P}_{\!0}^{(\multiidx)}\mat{U}_0
=\mat{K}^{(\multiidx)}
+\mat{U}_0\mat{C}^{(\multiidx)}
+\!\!\!\sum_{\substack{
\multiidxSec_1+\multiidxSec_2=\multiidx\\
\multiidxSec_1\notin\{\vec{0},\multiidx\}}}
\!\!\!
\mat{K}^{(\multiidxSec_1)}\mat{C}^{(\multiidxSec_2)}.
\end{equation}
This proves 
$\spanop\bigl\{\ran \mat{P}_{\!0}^{(\multiidx)}\mat{U}_0 :
|\multiidx|\leq\Ord\bigr\}
\subseteq
 \ran\mat{U}_0 
+
\spanop\{
\ran\mat{K}^{(\multiidx)}:
1\leq|\multiidx|\leq\Ord
\}$. 
Solving \eqref{eq:P0muU_in_terms_of_Kmu_Cmu} for
$\mat{K}^{(\multiidx)}$ and proceeding by induction on $|\multiidx|$
proves the reverse inclusion, 
because every $\mat{K}^{(\multiidxSec_1)}$ in the sum has
strictly lower derivative order.
\end{proof}
}

{We next derive a finite recursion for the thin derivatives
$\mat{K}^{(\multiidx)}$. With the convention that an empty sum is zero,
define the lower-order interaction terms
\begin{align} \label{eq:def_X_Y_multiidx}
\mat{X}^{(\multiidx)}
\coloneqq
\!\!\!\!\!
\sum_{\substack{
\multiidxSec_1+\multiidxSec_2=\multiidx\\
\multiidxSec_1\neq\vec{0},\;\multiidxSec_2\neq\vec{0}}}
\!\!\!\!\!
\mat{A}^{\!(\multiidxSec_1)}\mat{K}^{(\multiidxSec_2)}
\qquad \text{and} \qquad 
\mat{Y}^{(\multiidx)}
\coloneqq
\!\!\!\!\!
\sum_{\substack{
\multiidxSec_1+\multiidxSec_2=\multiidx\\
\multiidxSec_1\neq\vec{0},\;\multiidxSec_2\neq\vec{0}}}
\!\!\!\!\!
\mat{K}^{(\multiidxSec_1)}\mat{B}^{(\multiidxSec_2)}
\qquad
\text{for }1\leq|\multiidx|\leq\Ord.
\end{align}
All scaled derivatives of $\mat{K}$ and $\mat{B}$ occurring in these terms have
order strictly below $|\multiidx|$.
}

{\begin{theorem}[Recursive formulas for $\mat{K}^{(\multiidx)}$ and
$\mat{B}^{(\multiidx)}$]
\label{thm:resursive_Kmu}
For every $\multiidx$ with $1\leq|\multiidx|\leq\Ord$, it holds that
\begin{equation}
\mat{B}^{(\multiidx)}
=
\mat{U}_0^*\mat{A}^{\!(\multiidx)}\mat{U}_0
+\mat{U}_0^*\mat{X}^{(\multiidx)}
\qquad\text{and} \qquad 
\mat{K}^{(\multiidx)}
=
\mat{S}_0\bigl(
-\mat{A}^{\!(\multiidx)}\mat{U}_0
-\mat{X}^{(\multiidx)}
+\mat{Y}^{(\multiidx)}
\bigr),
\label{eq:recursive_Kmu}
\end{equation}
where $\mat{S}_0$ is the reduced resolvent characterized by
\eqref{eq:reduced_resolvent}.
\end{theorem}

Note that 
\eqref{eq:recursive_Kmu} involves only previously computed scaled derivatives of
$\mat{K}$ and $\mat{B}$.
}

{\begin{proof}
The definition of $\mat{B}$ and \eqref{eq:graph_parametrization} give
\begin{equation} \label{eq:B_mu}
\mat{B}(\prmtrVec)
=
\mat{U}_0^*\mat{A}(\prmtrVec)
\bigl(\mat{U}_0+\mat{K}(\prmtrVec)\bigr).
\end{equation}
Fix $\multiidx$ with $1\leq|\multiidx|\leq\Ord$. Taking the scaled
derivative of order $\multiidx$ in \eqref{eq:B_mu} gives
\[
\mat{B}^{(\multiidx)}
=
\mat{U}_0^*\mat{A}^{\!(\multiidx)}\mat{U}_0
+\mat{U}_0^*
\!\!\!
\sum_{\multiidxSec_1+\multiidxSec_2=\multiidx}
\!\!\!
\mat{A}^{\!(\multiidxSec_1)}\mat{K}^{(\multiidxSec_2)}
=
\mat{U}_0^*\mat{A}^{\!(\multiidx)}\mat{U}_0
+\mat{U}_0^*\mat{A}^{\!(\multiidx)}\mat{K}^{(\vec{0})} 
+\mat{U}_0^*\mat{A}^{\!(\vec{0})}\mat{K}^{(\multiidx)}
+\mat{U}_0^*\mat{X}^{(\multiidx)}.
\]
Note that $\mat{U}_0^*\mat{A}^{\!(\multiidx)}\mat{K}^{(\vec{0})} = \mat{0}$, 
because $\mat{K}^{(\vec{0})}=\mat{0}$. 
Since $\mat{U}_0^*\mat{A}^{\!(\vec{0})}=\mat{U}_0^*\mat{A}_0=\EigVal_0\mat{U}_0^*$ and
$\mat{U}_0^*\mat{K}^{(\multiidx)}=\mat{0}$, 
the term $\mat{U}_0^*\mat{A}_0\mat{K}^{(\multiidx)}$ also vanishes.
This proves the formula for $\mat{B}^{(\multiidx)}$ in
\eqref{eq:recursive_Kmu}.

Next, we derive the formula for $\mat{K}^{(\multiidx)}$.
To this end, 
we multiply \eqref{eq:matrix_equation_A_V_C} from the left by
$\mat{P}_{\!0}^{\perp}$. 
Since
$\mat{P}_{\!0}^{\perp}\mat{U}_0=\mat{0}$ 
and
$\mat{P}_{\!0}^{\perp}\mat{K}(\prmtrVec)=\mat{K}(\prmtrVec)$, 
we obtain
$\mat{P}_{\!0}^{\perp}\mat{A}(\prmtrVec)
\bigl(\mat{U}_0+\mat{K}(\prmtrVec)\bigr)
=\mat{K}(\prmtrVec)\mat{B}(\prmtrVec).$
Taking the scaled derivative of order $\multiidx$ at $\prmtrVec_0$ 
to this identity gives
\begin{equation}
    \mat{P}_{\!0}^{\perp}\mat{A}^{\!(\multiidx)}\mat{U}_0
+
\mat{P}_{\!0}^{\perp}
\!\!\!\!
\sum_{\multiidxSec_1+\multiidxSec_2=\multiidx}
\!\!\!\!
\mat{A}^{\!(\multiidxSec_1)}\mat{K}^{(\multiidxSec_2)}
\;=
\!\!\!\!
\sum_{\multiidxSec_1+\multiidxSec_2=\multiidx}
\!\!\!\!
\mat{K}^{(\multiidxSec_1)}\mat{B}^{(\multiidxSec_2)}.
\label{eq:intermediate_formula_K_beta}
\end{equation}
In the sum on the left, 
the term $\mat{P}_{\!0}^{\perp}\mat{A}^{\!(\multiidx)}\mat{K}^{(\vec{0})}$
vanishes, 
because $\mat{K}^{(\vec{0})}=\mat{0}$. 
Also observe that 
$\mat{P}_{\!0}^{\perp}\mat{A}^{\!(\vec{0})}\mat{K}^{(\multiidx)}
=\mat{P}_{\!0}^{\perp}\mat{A}_0\mat{K}^{(\multiidx)}
=\mat{A}_0\mat{K}^{(\multiidx)}$, 
since $\mat{P}_{\!0}^{\perp}$ commutes with $\mat{A}_0$
and $\mat{P}_{\!0}^{\perp}\mat{K}^{(\multiidx)} = \mat{K}^{(\multiidx)}$
by the definition of~$\mat{K}$.
On the right hand side, 
$\mat{K}^{(\vec{0})} \mat{B}^{(\multiidx)}$ vanishes, 
again due to $\mat{K}^{(\vec{0})}=\mat{0}$. 
Also note that 
$\mat{K}^{(\multiidx)}\mat{B}^{(\vec{0})}
=\EigVal_0\mat{K}^{(\multiidx)}$, 
since $\mat{B}^{(\vec{0})} = \EigVal_0 \mat{I}_{m_0}$.
Consequently, 
\eqref{eq:intermediate_formula_K_beta}
becomes 
\[
\mat{P}_{\!0}^{\perp}\mat{A}^{\!(\multiidx)}\mat{U}_0
+\mat{A}_0\mat{K}^{(\multiidx)}
+\mat{P}_{\!0}^{\perp}\mat{X}^{(\multiidx)}
=\EigVal_0\mat{K}^{(\multiidx)}+\mat{Y}^{(\multiidx)},
\]
and rearranging gives
\begin{equation}
(\mat{A}_0-\EigVal_0\mat{I})\mat{K}^{(\multiidx)}
=
-\mat{P}_{\!0}^{\perp}
\bigl(\mat{A}^{\!(\multiidx)}\mat{U}_0
+\mat{X}^{(\multiidx)}\bigr)
+\mat{Y}^{(\multiidx)}.
\label{eq:final_formula_K_beta}
\end{equation}
Observe from \eqref{eq:def_X_Y_multiidx} 
that $\ran\mat{Y}^{(\multiidx)} \subseteq \ran\mat{P}_{\!0}^{\perp}$. 
To obtain the formula for $\mat{K}^{(\multiidx)}$ 
in \eqref{eq:recursive_Kmu},
we apply $\mat{S}_0$ from the left to \eqref{eq:final_formula_K_beta} and
use $\mat{S}_0(\mat{A}_0-\EigVal_0\mat{I})
=\mat{P}_{\!0}^{\perp}$ from \eqref{eq:reduced_resolvent},
$\mat{P}_{\!0}^{\perp}\mat{K}^{(\multiidx)}=\mat{K}^{(\multiidx)}$,
and $\mat{S}_0\mat{P}_{\!0}^{\perp}=\mat{S}_0$.
\end{proof}
}

\begin{remark}
[Simplification for univariate parameter dependence] \label{rem:univariate_simplification}
In the special case $\Pdim = 1$,
a multi-index $\multiidx \in \bb{N}_0^\Pdim$ reduces to an integer
$n=|\multiidx|\in\{0,\ldots,\Ord\}$.
In other words,
the combinatorial complexity of decomposing multi-indices is reduced to the decomposition of a natural number into two.
This is frequently observed in mathematical physics,
see e.g. \cite{Kato76,McW62}.
A further simplification would be hardly possible even in this univariate case.
The combinatorial complexity in \Cref{thm:resursive_Kmu} is intrinsic to the multivariate setting.
\end{remark}

\subsubsection{Implementation considerations} 
\label{subsubsec:implementation_consideration}
{By \Cref{thm:connection_P0_Y,thm:resursive_Kmu}, 
the desired Taylor-RB space has
the characterization \eqref{eq:proj_range_relation} and can be computed
recursively via \eqref{eq:recursive_Kmu}.
The recursion requires the scaled derivatives
$\mat{A}^{\!(\multiidx)}$, the action of the orthogonal projector
$\mat{P}_{\!0}^\perp$, and the action of the reduced resolvent
$\mat{S}_0$ of $\mat{A}_0$ at $\EigVal_0$.
We address these three ingredients separately.
}

\begin{remark}[Affine decomposition and partial derivatives] \label{rem:affine_decomposition}
In usual RBM applications, 
the matrix-valued function $\mat{A}(\prmtrVec)$ possesses an \emph{affine decomposition structure} in terms of the parameters, 
i.e., there exist parameter-independent Hermitian matrices $\mat{A}_1, \ldots, \mat{A}_\nrAffine$ 
and \rev{functions $\theta_1(\prmtrVec), \ldots,
\theta_\nrAffine(\prmtrVec)$ of class $\mathrm{C}^{\Ord +1}$ near
$\prmtrVec_0$ with explicit expressions},
such that
\begin{align} \label{eq:affine_linear_decomposition}
\mat{A}(\prmtrVec)
\, = \,
\sum_{\idxAffine=1}^\nrAffine \theta_\idxAffine(\prmtrVec) \, \mat{A}_\idxAffine.
\end{align}
\rev{In this case, the scaled derivative
$\mat{A}^{\!(\multiidx)}$ can be computed directly from
\eqref{eq:affine_linear_decomposition} and the derivatives 
of~$\theta_\idxAffine$ at $\prmtrVec_0$. 
For $|\multiidx|\leq\Ord$, this reads }
\begin{align*}
\mat{A}^{\!(\multiidx)}
\, = \,
\sum_{\idxAffine=1}^\nrAffine 
\theta_\idxAffine^{(\multiidx)}(\prmtrVec_0) \, 
\mat{A}_\idxAffine
\quad \text{with} \quad
\theta_\idxAffine^{(\multiidx)}(\prmtrVec_0)
\, \coloneqq \,
\frac{1}{\beta_1! \cdots \beta_\Pdim!} \, 
\frac{\partial^{|\multiidx|} \theta_\idxAffine}{\partial \prmtr_1^{\beta_1} \cdots \partial \prmtr_\Pdim^{\beta_\Pdim}}(\prmtrVec_0).
\end{align*}
\end{remark}

\begin{remark}
[Action of the orthogonal complement projector] \label{rem:orthogonal_complement_projector_action}
{The orthogonal complement projector $\mat{P}_{\!0}^\perp$ need
not be assembled explicitly. Its action on any vector
$\vec{v}\in\bb{C}^\fdim$ is given by}
$\mat{P}_{\!0}^\perp \vec{v}
\, = \,
\vec{v} - \mat{P}_{\!0} \vec{v}
\, = \,
\vec{v} - \mat{U}_0 \mat{U}_0^* \vec{v}.
$
\end{remark}

\begin{remark}
[Action of the reduced resolvent] 
\label{rem:reduced_resolvent_action}
Recall from \eqref{eq:reduced_resolvent} that 
the reduced resolvent $\mat{S}_0$ of $\mat{A}_0$ w.r.t. 
$\EigVal_0$ satisfies the equation $(\mat{A}_0 - \EigVal_0 \mat{I}) \, \mat{S}_0 = \mat{P}_{\!0}^\perp$.
Hence, for any vector $\vec{v} \in \bb{C}^\fdim$, 
we obtain $\mat{S}_0 \vec{v}$ by solving the linear equation system (LES)
\begin{align} \label{eq:linear_system_reduced_resolvent}
(\mat{A}_0 - \EigVal_0 \,\mat{I}) \, \vec{x} 
\, = \,
\mat{P}_{\!0}^\perp  \vec{v}.
\end{align}
Note that the matrix $\mat{A}_0 - \EigVal_0\,\mat{I}$ is singular, 
but the right-hand side $\mat{P}_{\!0}^\perp \, \vec{v}$ is always orthogonal to
{$\ker(\mat{A}_0-\EigVal_0\mat{I})=\ran\mat{P}_{\!0}$}.
Hence, the LES \eqref{eq:linear_system_reduced_resolvent} is still uniquely solvable on $\ran \, \mat{P}_{\!0}^\perp$. 

Practically, 
we can add a regularization term $\alpha \, \mat{P}_{\!0}$ with $\alpha >0$ to the matrix on the left-hand side of \eqref{eq:linear_system_reduced_resolvent}
and then apply a standard (matrix-free) solver on the modified regular LES
\begin{align} 
\label{eq:regularized_linear_system_reduced_resolvent}
(\mat{A}_0 - \EigVal_0 \,\mat{I} + \alpha \, \mat{P}_{\!0}) \, \vec{x} 
\, = \,
\mat{P}_{\!0}^\perp  \vec{v},
\end{align}
such as \texttt{MinRes}, \texttt{GMRES}, or even \texttt{CG} if $\EigVal_0$ is the smallest eigenvalue of $\mat{A}_0$.

By multiplying both sides of \eqref{eq:regularized_linear_system_reduced_resolvent} with $\mat{P}_{\!0}$,
we deduce that the solution of \eqref{eq:regularized_linear_system_reduced_resolvent} is orthogonal to $\ran \, \mat{P}_{\!0}$.
Thus, it is the desired norm-minimizing solution of \eqref{eq:linear_system_reduced_resolvent}. 
Also note that the action of $\mat{P}_{\!0}^\perp \, \vec{v}$ on the right-hand side is computed as in \Cref{rem:orthogonal_complement_projector_action}.
\end{remark}

In addition, 
we also note  
that for each multi-index $\multiidx$ with
\rev{$1\leq|\multiidx|\leq\Ord$},
we can (and should) loop only once over all possible splittings $\multiidxSec_1 + \multiidxSec_2 = \multiidx$ with $\multiidxSec_1, \multiidxSec_2 \neq \vec{0}$ to compute $\mat{X}^{(\multiidx)}$ and $\mat{Y}^{(\multiidx)}$.
Besides, 
the term $\mat{X}^{(\multiidx)}$ defined in \eqref{eq:def_X_Y_multiidx} 
appears in both formulas in \eqref{eq:recursive_Kmu}
for each $\multiidx$,
so we should store it and reuse it for the calculation of $\mat{B}^{(\multiidx)}$.
Also note that $\mat{B}^{(\multiidx)}$ is only needed for $|\multiidx| < \Ord$.
To increase stability, 
we can orthonormalize existing outputs after the computation of $\{\mat{K}^{(\multiidx)} : |\multiidx| = j \}$. 

All in all, the implementation procedure of the Taylor-RBM is summarized in \Cref{alg:assembly_taylor_rbm_space}.

\begin{algorithm}[t]
\begin{algorithmic}[1]
	\REQUIRE
    \rev{A target derivative order $\Ord \in \bb{N}_0$, 
    a Hermitian matrix-valued function
    $\mat{A}(\prmtrVec)\in\bb{C}^{\fdim\times\fdim}$ of class
    $\mathrm{C}^{\Ord +1}$ near $\prmtrVec_0\in\bb{R}^\Pdim$ 
    together with its scaled derivatives $\mat{A}^{\!(\multiidx)}$ 
    for $1\leq|\multiidx|\leq\Ord$,}
    an eigenvalue $\EigVal_0$ of $\mat{A}_0 \coloneqq \mat{A}(\prmtrVec_0)$ of multiplicity $m_0$,
    and a matrix $\mat{U}_0 \in \bb{C}^{\fdim \times m_0}$ 
    whose columns form an ONB of the eigenspace $\ker (\mat{A}_0 - \EigVal_0 \mat{I})$.
	\ENSURE 
    A matrix $\mat{V} \in \bb{C}^{\fdim \times r_{\!\Ord}}$ of orthonormal columns spanning the Taylor-RB space $\RBspace_{\!\Ord}(\prmtrVec_0)$ w.r.t. $\EigVal_0$ in the sense of \Cref{def:taylor_rbm}. 
	\STATE
    \rev{Initialize $\mat{V}\gets\mat{U}_0$,
    $\mat{K}^{(\vec{0})}\gets\mat{0}$,
    $\mat{X}^{(\vec{0})}\gets\mat{0}$,
    $\mat{Y}^{(\vec{0})}\gets\mat{0}$, and
    $\mat{B}^{(\vec{0})}\gets\EigVal_0\mat{I}_{m_0}$.}
    \FOR{{\color{red}$j=1$ \TO $\Ord$}}
    \FORALL{$\multiidx \in \bb{N}_0^\Pdim$ with $|\multiidx| = j$}
    \STATE 
    Compute $\mat{X}^{(\multiidx)}\in \bb{C}^{\fdim \times m_0}$ and $\mat{Y}^{(\multiidx)}\in \bb{C}^{\fdim \times m_0}$ by \eqref{eq:def_X_Y_multiidx}. 
    \STATE 
    Compute $\mat{K}^{(\multiidx)}\in \bb{C}^{\fdim \times m_0}$ by \eqref{eq:recursive_Kmu}. 
    \STATE
    Compute $\mat{B}^{(\multiidx)}\in \bb{C}^{m_0 \times m_0}$ by \eqref{eq:recursive_Kmu} if $|\multiidx| < \Ord$.
    \ENDFOR 
    \STATE 
    Update $\mat{V} \gets \rm{orthonormalize}\{ \mat{V}, \mat{K}^{(\multiidx)} :\, |\multiidx| = j \}$.
    \ENDFOR
    \RETURN $\mat{V}$.
\end{algorithmic}
\caption{Efficient assembly of the Taylor-RB space 
with respect to one eigenvalue of a parameter-dependent Hermitian matrix-valued function.}
\label{alg:assembly_taylor_rbm_space}
\end{algorithm}

{ 
\begin{remark}[Offline computational cost]
\label{rem:offline_computational_cost}
The number of multi-indices with
$1\leq|\multiidx|\leq\Ord$ is
$\binom{\Ord+\Pdim}{\Pdim}-1$.
Thus, \Cref{alg:assembly_taylor_rbm_space} requires this many
applications of $\mat{S}_0$ to matrices with $m_0$ columns. By
\eqref{eq:regularized_linear_system_reduced_resolvent}, these amount to
$\binom{\Ord+\Pdim}{\Pdim}-1$ block linear solves with $m_0$
right-hand sides. 
Since all involed linear systems share the same operator,
cf.~\Cref{rem:reduced_resolvent_action},
it is possible to use a common preconditioner
for acceleration.

The combinatorial factor counts the mixed derivatives of total order at
most $\Ord$. 
Therefore, 
it is intrinsic to a multivariate Taylor construction 
that retains every such derivative rather than to the recursion. 
For fixed $\Pdim$, 
it grows polynomially in $\Ord$ with degree $\Pdim$; 
for fixed $\Ord$, 
it grows polynomially in $\Pdim$ with degree $\Ord$. 
If the parameter dependence is anisotropic or sparse,
smaller problem-adapted multi-index sets may suffice.
\end{remark}

Once the Taylor-RB space has been assembled, approximations at new
parameter values are obtained in the online phase.
\begin{remark}[Online computation of $\mat{Q}_{\rm{rb}}$]
\label{rem:online_assembly_Q_rb}
Given a Taylor-RB basis matrix
\(\mat V_{\mkern-5mu\Ord}\in\mathbb C^{\fdim\times r_{\!\Ord}}\),
the online computation at a query parameter \(\prmtrVec\) 
consists of assembling the reduced matrix 
\(\mat A^{\!\!\mat V_{\mkern-5mu\Ord}}(\prmtrVec)\) 
and solving the corresponding \(r_{\!\Ord}\)-dimensional EVP.
This yields the
reduced spectral projector
$\mat{Q}^{\!\mat{V}_{\mkern-5mu\Ord}}(\prmtrVec)$ from
\Cref{def:taylor_rbm}. 
In typical applications, 
one is interested in the expectation values of
observables rather than in $\mat{Q}_{\rm{rb}}(\prmtrVec)$ itself:
for a given $\mat{M}=\mat{M}^*$, 
one usually aims at
$\tr \mat{M}\mat{P}(\prmtrVec) \approx
\tr \mat{M}\mat{Q}_{\rm{rb}}(\prmtrVec)$;
see e.g.~\cite{HSWR22,BHWRS23}.
Using
$\mat{Q}_{\rm{rb}}(\prmtrVec)
=\mat{V}_{\mkern-5mu\Ord}
\mat{Q}^{\!\mat{V}_{\mkern-5mu\Ord}}(\prmtrVec)
\mat{V}_{\mkern-5mu\Ord}^{*}$
and the cyclicity of the trace, we obtain
\[
\tr \mat{M}\mat{Q}_{\rm{rb}}(\prmtrVec)
=
\tr\bigl(
\mat{V}_{\mkern-5mu\Ord}^{*}\mat{M}\mat{V}_{\mkern-5mu\Ord}
\bigr)
\mat{Q}^{\!\mat{V}_{\mkern-5mu\Ord}}(\prmtrVec).
\]
Thus, after precomputing the
$r_{\!\Ord}\times r_{\!\Ord}$ matrix
$\mat{V}_{\mkern-5mu\Ord}^{*}\mat{M}\mat{V}_{\mkern-5mu\Ord}$
offline, the ambient-dimensional matrix
$\mat{Q}_{\rm{rb}}(\prmtrVec)$ need not be assembled
in the online phase.
The assembly of \(\mat A^{\!\!\mat V_{\mkern-5mu\Ord}}(\prmtrVec)\)
leverages the affine decomposition \eqref{eq:affine_linear_decomposition},
under which the reduced matrices
$\mat{V}_{\mkern-5mu\Ord}^*\mat{A}_\idxAffine
\mat{V}_{\mkern-5mu\Ord}$ can be precomputed offline and combined with
the coefficients $\theta_1(\prmtrVec),\dots,\theta_\nrAffine(\prmtrVec)$ online.
In summary, assembling the reduced matrix and solving its
Hermitian EVP require
$\bigO(\nrAffine r_{\!\Ord}^2+r_{\!\Ord}^3)$ operations per query online, 
independently of the ambient dimension~$\fdim$.
\end{remark}
}

\section{Numerical Experiments}\label{sec:num_exp}

This section aims to illustrate the performance of the Taylor-RBM. 
After introducing the model problem in \Cref{subsec:QSS}, 
we will focus on 
(i)   studying the dimension growth of the Taylor-RB space 
in the derivative order in \Cref{subsec:dim_growth},
(ii)  validating the theoretical error estimates of the Taylor-RBM
in \Cref{subsec:validation}, and
(iii) {comparing the Taylor-RBM and PT approximations} in \Cref{subsec:comparison_PT}. 

The numerical experiments are conducted {in \texttt{Python}}. 
The code and the experiment data are available on \url{https://doi.org/10.18419/DARUS-6559}.

Throughout the experiments, 
full EVPs are solved with the \texttt{scipy.sparse.linalg.eigsh} function, 
and the reduced EVPs are solved with the \texttt{scipy.linalg.eigh} function,
both using machine precision as error tolerance.
The threshold for identifying eigenvalue clusters is $10^{-10}$. 
\rev{For each query parameter, $\mat{Q}_{\rm{rb}}$ is computed by the
 procedure described in \Cref{rem:online_assembly_Q_rb}.}
\rev{Since the experiments primarily assess approximation accuracy
rather than optimized runtimes, we solve the reduced-resolvent systems
\eqref{eq:regularized_linear_system_reduced_resolvent} in the offline
assembly with tight tolerances and without problem-specific
preconditioning.}

\subsection{Model problem: xxz-chain quantum spin systems} \label{subsec:QSS}

For $L,j\in \bb{N}$ with $j \leq L$ and for any matrix $\mat{S} \in \bb{C}^{2 \times 2}$, 
we write  
$\mat{S}^{(L,j)} \vcentcolon=  \mat{I}_{2^{j-1}} \otimes \mat{S} \otimes  \mat{I}_{2^{L-j}}  \in \bb{C}^{2^L \times 2^L}.$
{Here, $\otimes$ denotes the Kronecker product of matrices.}

We will test our method on the 
\emph{xxz-chain model with open boundary conditions},
which is in the affine decomposition form \eqref{eq:affine_linear_decomposition} and defined for $\prmtrVec = (\prmtr_1,\prmtr_2)^{\!\top} \in \bb{R}^2$ by 
\begin{equation}\label{eqn:QSS:xxz:aff}
	\mat{A}(\prmtrVec) = \mat{A}_1+\prmtr_1 \mat{A}_2 - \prmtr_2 \mat{A}_3 \in \bb{C}^{2^L \times 2^L},
\end{equation}
where 
\begin{align*}
 \mat{A}_1 \vcentcolon= \frac{1}{4} \sum_{j=1}^{L-1} \! \Big( \mat{S}_{\rm{x}}^{(L,j)} \mat{S}_{\rm{x}}^{(L,j+1)} + \mat{S}_{\rm{y}}^{(L,j)} \mat{S}_{\rm{y}}^{(L,j+1)} \Big), 
 \quad 
 \mat{A}_2 \vcentcolon= \frac{1}{4} \sum_{j=1}^{L-1}  \mat{S}_{\rm{z}}^{(L,j)} \mat{S}_{\rm{z}}^{(L,j+1)} , 
 \quad 
 \mat{A}_3 \vcentcolon= \frac{1}{2} \sum_{j=1}^{L} \mat{S}_{\rm{z}}^{(L,j)} , 
 \quad 
\end{align*}
with $\mat{S}_{\rm{x}}, \mat{S}_{\rm{y}}, \mat{S}_{\rm{z}}$ 
denoting the \textit{spin-1/2 matrices}, i.e.,
\begin{align*}
    \mat{S}_{\rm{x}} \vcentcolon= 
    \left[
    \begin{array}{cc}
        0 & 1 \\
        1 & 0
    \end{array}
    \right],    \qquad 
    \mat{S}_{\rm{y}} \vcentcolon= 
    \left[
    \begin{array}{cc}
        0 & -\rm{i} \\
        \rm{i} & 0
    \end{array}
    \right],    \qquad 
    \mat{S}_{\rm{z}} \vcentcolon= 
    \left[
    \begin{array}{cc}
        1 & 0 \\
        0 & -1
    \end{array}
    \right]. 
\end{align*}
This is a classical model in quantum many-body physics, 
which has been extensively studied in the physics literature \cite{HSWR22, SRFB2004}.
It exhibits diverse spectral behaviors at different parameter values, 
providing good test cases for studying the Taylor-RBM.
\rev{The parameter dependence in \eqref{eqn:QSS:xxz:aff} 
is affine and hence real-analytic.
Therefore, although the preceding theory is
formulated for a fixed order $\Ord$ 
under $\mathrm{C}^{\Ord +1}$-regularity, 
it applies here for every finite order.}

\subsection{Dimension growth of the Taylor-RB space} \label{subsec:dim_growth}

By \Cref{rem:upper_bound_dim_RB},
the dimension of the Taylor-RB space of order $\Ord$ w.r.t.\! one single eigenvalue $\EigVal_0$ of multiplicity $m_0$ at parameter $\prmtrVec_0 \in \bb{R}^2$ is upper bounded by $m_0 \binom{\Ord + 2}{2} = m_0 {(\Ord + 2)(\Ord + 1)}/{2}$.

To gain {insight into this aspect},
we assemble Taylor-RB spaces with the xxz-chain model of chain length $L=15$ at several parameter points,
i.e., the system matrix is of size $2^{15} \times 2^{15} = 32\,768 \times 32\,768$.
The different parameter points are chosen to be $\prmtrVec_0 \in \{(-1,0)^{\!\top}, (0,0)^{\!\top},(1,1)^{\!\top}, (-1,1)^{\!\top} \}$.
We construct the Taylor-RB spaces up to order $\Ord \leq 10$,
and consider both the Taylor-RB spaces  w.r.t.\! only the smallest eigenvalue cluster ($K=1$) and w.r.t.\! the first two eigenvalue clusters ($K=2$).
During the orthogonalization process, a threshold of $10^{-12}$ is used to identify the linear dependence of the generated vectors.
The results are shown in \Cref{tab:xxz-dimension-analysis}.

\begin{table}[ht]
    \centering
    \small
    \setlength{\tabcolsep}{2.8pt} 

    \begin{subtable}{0.49\textwidth}
        \centering
        \begin{tabular}{@{}c rcc rcc rcc rcc@{}}
            \toprule
            & \multicolumn{2}{c}{$ (-1,0)^{\!\top}$} && \multicolumn{2}{c}{$ (0,0)^{\!\top}$} && \multicolumn{2}{c}{$ (1,1)^{\!\top}$} && \multicolumn{2}{c}{$ (-1,1)^{\!\top}$} \\
            \cmidrule(lr){2-3} \cmidrule(lr){5-6} \cmidrule(lr){8-9} \cmidrule(lr){11-12}
             $\quad \Ord \quad $ & $r$ & UB && $r$ & UB && $r$ & UB && $r$ & UB \\ 
            \midrule
            0  & 16  & 16  && 2  & 2   && 1  & 1  && 1 & 1  \\
            1  & 48  & 48  && 4  & 6   && 3  & 3  && 1 & 3  \\
            2  & 96  & 96  && 6  & 12  && 6  & 6  && 1 & 6  \\
            3  & 160 & 160 && 8  & 20  && 8  & 10 && 1 & 10 \\
            4  & 240 & 240 && 12 & 30  && 13 & 15 && 1 & 15 \\
            5  & 336 & 336 && 14 & 42  && 15 & 21 && 1 & 21 \\
            6  & 448 & 448 && 20 & 56  && 21 & 28 && 1 & 28 \\
            7  & 576 & 576 && 26 & 72  && 27 & 36 && 1 & 36 \\
            8  & 720 & 720 && 33 & 90  && 35 & 45 && 1 & 45 \\
            9  & 880 & 880 && 38 & 110 && 43 & 55 && 1 & 55 \\
            10 & 1054& 1056&& 51 & 132 && 52 & 66 && 1 & 66 \\
            \bottomrule
        \end{tabular}
        \vspace{0.5em}
        \caption{Taylor-RB space dimension $r$ along the derivative order $\Ord$ targeting only the smallest eigenvalue $\EigVal_1$.
                 "UB" is the theoretical upper bound $m_1 \binom{{\Ord}+2}{2}$ of the RB space dimension $r$, 
                 where $m_1$ is the multiplicity of $\EigVal_1$ and {equals} $r$ for $\Ord=0$.}
        \label{tab:xxz-k1}
    \end{subtable}
    \hfill 
    \begin{subtable}{0.49\textwidth}
        \centering
        \begin{tabular}{@{}c rcc rcc rcc rcc@{}}
            \toprule
            & \multicolumn{2}{c}{$ (-1,0)^{\!\top}$} && \multicolumn{2}{c}{$ (0,0)^{\!\top}$} && \multicolumn{2}{c}{$ (1,1)^{\!\top}$} && \multicolumn{2}{c}{$ (-1,1)^{\!\top}$} \\
            \cmidrule(lr){2-3} \cmidrule(lr){5-6} \cmidrule(lr){8-9} \cmidrule(lr){11-12}
            $\quad \Ord \quad $ & $r$ & UB && $r$ & UB && $r$ & UB && $r$ & UB \\ 
            \midrule
            0  & 30  & 30  && 6   & 6   && 2   & 2   && 2  & 2   \\
            1  & 90  & 90  && 12  & 18  && 6   & 6   && 3  & 6   \\
            2  & 180 & 180 && 20  & 36  && 11  & 12  && 4  & 12  \\
            3  & 300 & 300 && 33  & 60  && 16  & 20  && 5  & 20  \\
            4  & 450 & 450 && 47  & 90  && 24  & 30  && 6  & 30  \\
            5  & 630 & 630 && 66  & 126 && 29  & 42  && 7  & 42  \\
            6  & 840 & 840 && 89  & 168 && 41  & 56  && 8  & 56  \\
            7  & 1080& 1080&& 113 & 216 && 52  & 72  && 9  & 72  \\
            8  & 1349& 1350&& 143 & 270 && 66  & 90  && 10 & 90  \\
            9  & 1643& 1650&& 173 & 330 && 81  & 110 && 11 & 110 \\
            10 & 1968& 1980&& 213 & 396 && 100 & 132 && 13 & 132 \\
            \bottomrule
        \end{tabular}
        \vspace{0.5em}
        \caption{Taylor-RB space dimension $r$ along the derivative order $\Ord$ targeting the first two eigenvalues $\EigVal_1$ and $\EigVal_2$.
                 "UB" refers to the theoretical upper bound $M \binom{{\Ord}+2}{2}$ of $r$, 
                 where $M$ is the sum of the multiplicities of $\EigVal_1$ and $\EigVal_2$ and {equals} $r$ for $\Ord=0$.}
        \label{tab:xxz-k2}
    \end{subtable}

    \caption{xxz-model of length $L=15$: List of dimension growth of the Taylor-RB spaces at four different parameter points. 
            The underlying space has dimension $2^L = 32\,768$. }
    \label{tab:xxz-dimension-analysis}
\end{table}

The dimension growth reported in \Cref{tab:xxz-dimension-analysis} reveals different local characteristics of the eigenvalue problem across the parameter space:

\begin{itemize}
\item 
\emph{Maximal growth at $\prmtrVec_0 = (-1,0)^{\!\top}$.} 
This point exhibits high initial degeneracy ($m_1 = 16$ and $m_2 = 14$). 
The subspace dimension grows at the same rate as the theoretical upper bound, 
implying that the ranges of the Taylor-derivatives are almost entirely linearly independent. 
This suggests a complex local structure of the spectral projector, 
where parameter variations explore a high-dimensional subspace in the Hilbert space without redundancy.

\item 
\emph{Moderate growth at $\prmtrVec_0 = (0,0)^{\!\top}$.} 
Although this point also possesses degenerate eigenvalues, 
the dimension growth of the Taylor-RB space is much slower compared to the maximal growth seen at $(-1,0)^{\!\top}$. 
This implies that a significant portion of vectors from ranges of the higher-order derivatives are linearly dependent, 
allowing for a more compact representation.

\item 
\emph{Rapid growth at $\prmtrVec_0 = (1,1)^{\!\top}$ despite a non-degenerate eigenvector.} 
At this point, the eigenvectors are simple ($m_1 = m_2 = 1$), 
yet the dimension of the Taylor-RB space grows rapidly and stays close to the theoretical limit. 
This indicates that while the eigenstate is non-degenerate, 
it is highly sensitive to the parameters $\prmtrVec$.

\item 
\emph{Parameter invariance at $\prmtrVec_0 = (-1,1)^{\!\top}$.} 
For the smallest eigenvalue ($K=1$), 
the dimension remains strictly constant ($r=1$). 
This indicates that the corresponding eigenspace is locally invariant around $\prmtrVec_0$. 
In other words, the higher-order Taylor-derivatives vanish or their ranges reside in the span of the initial eigenvector, 
making the expansion trivial.
\end{itemize}

We conclude that the comparison between the dimension of the Taylor-RB space and the combinatorial upper bound serves as a numerical measure of the local sensitivity or complexity of the total spectral projector.

\subsection{Validation of the subspace error bound} 
\label{subsec:validation}

Let us continue with the xxz-model of length $L=15$,
{i.e.,} the underlying space has dimension $2^L = 2^{15} = 32\,768$. 
Our goal is to examine the convergence rate of the Taylor-RBM.
We consider the parameter point $\prmtrVec_0 = (1,1)^{\!\top}$, 
as it exhibits a non-degenerate smallest eigenvalue and a rapid growth of the Taylor-RB space dimension,
and we target the smallest eigenvalue cluster ($K=1$). 

\subsubsection{Qualitative validation}
For a uniform grid of $71 \times 71 = 5\,041$ parameters from the box 
$ \big\{ \prmtrVec \in \bb{R}^2 : \prmtrVec - (1,1)^{\!\top} \in [-0.08,\, 0.08]^2  \big\} $,
and Taylor-RB spaces of order $\Ord \leq 3$,
the approximation error $\|\mat{Q}_\rm{rb}(\prmtrVec) - \mat{P}(\prmtrVec)\|$
and the Ritz-value error $\rm{tr}\, \mat{A}(\prmtrVec)\big( \mat{Q}_\rm{rb}(\prmtrVec) - \mat{P}(\prmtrVec) \big)$
of the Taylor-RBM are computed. 
Again, a threshold of $10^{-12}$ is used to identify linear dependence during the orthogonalization process.
The dimensions of the Taylor-RB spaces are $r=1$ for $\Ord=0$, $r=3$ for $\Ord=1$, $r=6$ for $\Ord=2$, and $r=8$ for $\Ord=3$.
The results are shown in \Cref{fig:expB_heatmaps}.

\begin{figure}[htbp]
    \centering
    \begin{subfigure}[b]{\textwidth}
        \centering
        \includegraphics[width=\textwidth, trim = 0mm 5mm 0mm 7mm]{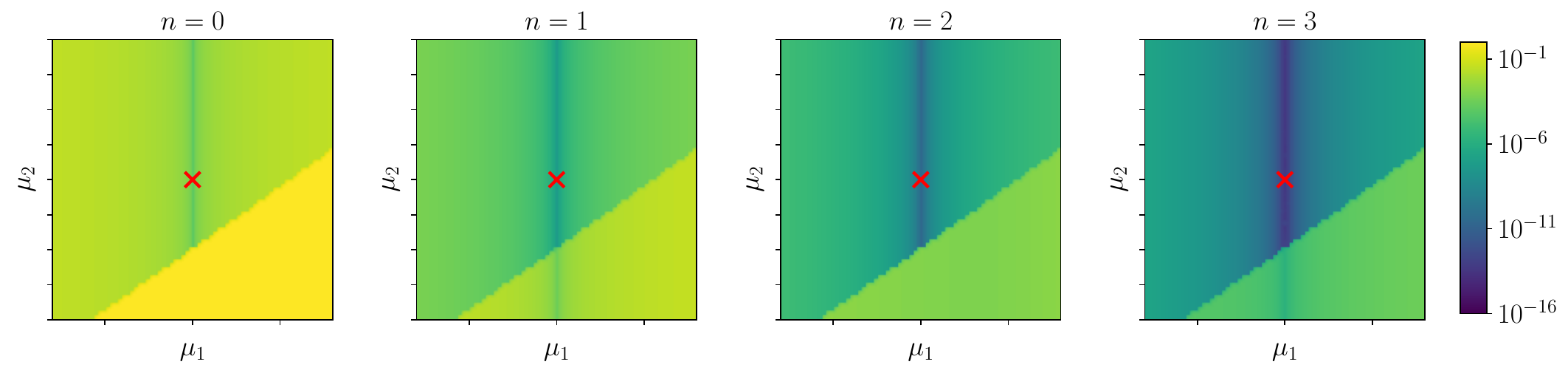}
        \caption{Approximation error $\|\mat{Q}_\rm{rb}(\prmtrVec) - \mat{P}(\prmtrVec)\|$ of the Taylor-RBM of order $\Ord$.}
    \end{subfigure}\\
    \vspace{0.2cm}
    \begin{subfigure}[b]{\textwidth}
        \centering
        \includegraphics[width=\textwidth, trim = 0mm 5mm 0mm 0mm]{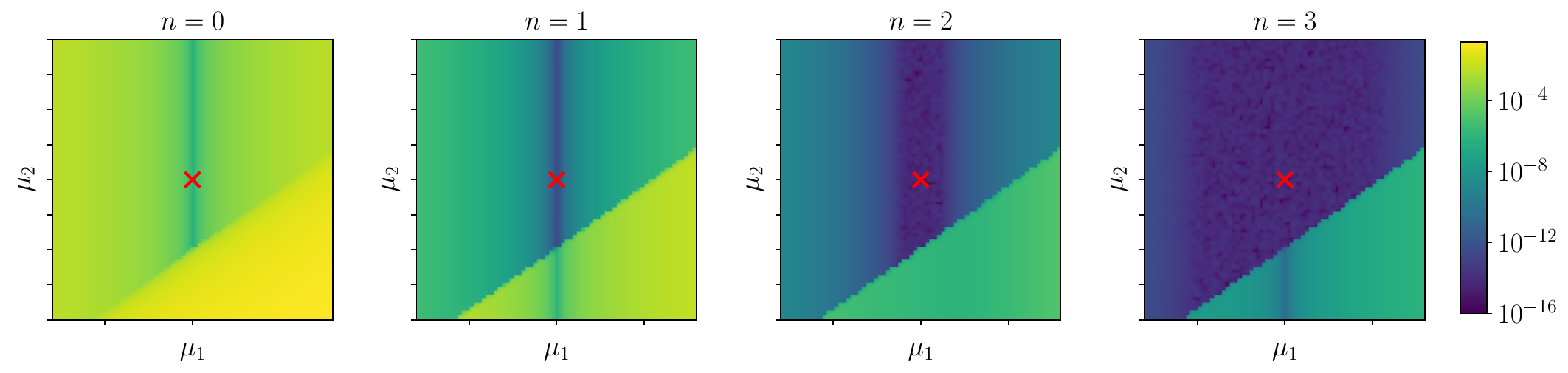}
        \caption{Error of Ritz-value sum $\rm{tr}\, \mat{A}(\prmtrVec)\big( \mat{Q}_\rm{rb}(\prmtrVec) - \mat{P}(\prmtrVec) \big)$  of the Taylor-RBM of order $\Ord$.}
    \end{subfigure}
    \caption{xxz-model of length $L=15$: 
        Approximation error and error of Ritz-value sum of the total eigenspace of rank $m_1 $ at parameters $\prmtrVec$ with $\prmtrVec - \prmtrVec_0 \in [-0.08,\, 0.08]^2$ using Taylor-RBM errors of order~$\Ord$,
        where $\prmtrVec_0 =(1,1)^{\!\top}$ is marked in red cross and has $m_1 = 1$.} 
    \label{fig:expB_heatmaps}
\end{figure}

The heatmaps presented in \Cref{fig:expB_heatmaps} {confirm} the theoretical predictions from \Cref{sec:Taylor_rbm}, 
which states that the approximation quality for both the total eigenprojector and the Ritz-value sum improves as the order $\Ord$ increases in the neighborhood of the expansion point $\prmtrVec_0 = (1,1)^{\!\top}$. 
More specifically, the following observations can be made: 

\begin{itemize}
\item 
\emph{Directional sensitivity.}
The error distributions show a vertical alignment.
This phenomenon is consistent with the physical nature of the model \cite{SRFB2004}
and indicates that the eigenvector is more sensitive to perturbations along the $\prmtr_1$ direction than along the $\prmtr_2$ direction.

\item 
\emph{Comparison of error metrics.} 
{The Ritz-value errors decay at approximately twice the algebraic
rate of the eigenspace errors},
as shown by comparing the upper and lower heatmaps. 
This agrees with \Cref{thm:energy_difference_rb}, which predicts quadratic convergence of eigenvalue errors relative to subspace errors.

\item 
\emph{Behavior under eigenvalue crossings.} 
The {boundary} of the lower-right region of the parameter box corresponds to a point where the spectral gap of $\mat{A}(\prmtrVec)$ vanishes, 
i.e., where eigenvalue crossing occurs.
While this lower-right region exhibits high localized error for $\Ord=0$, 
the Taylor-RBM seems to be able to effectively provide a relatively good approximation as $\Ord$ increases. 
\end{itemize}

\subsubsection{Quantitative validation}
Let us now take a closer look at the error convergence rates of the Taylor-RBM by focusing on the single parameter path from $\prmtrVec_0 = (1,1)^{\!\top}$ to $\prmtrVec = (1.08, 1)^{\!\top}$,
i.e., the horizontal center line starting at $(1,1)^{\!\top}$ of the heatmaps in \Cref{fig:expB_heatmaps}.
The selected results of $\prmtr_1-1 \in [10^{-4},8\cdot 10^{-2}]$ are shown in \Cref{fig:expB_convergence},
from which we make the following observations:
\begin{itemize}
\item 
\emph{Asymptotic convergence rates.} 
In the log-log representation, 
the error curves for different orders $\Ord$ exhibit almost constant slopes for {$\prmtr_1 - 1 < 0.0572$}. 
Specifically in \Cref{subfig:expB_subspace_error_rate},
the empirical slopes of the approximation error calculated with the scatter points are {approximately} 
$0.999$ for $\Ord=0$, $1.998$ for $\Ord=1$, $2.993$ for $\Ord=2$, and $3.975$ for $\Ord=3$.
We see that the Taylor-RBM achieves the expected algebraic convergence rate local to the expansion point before the influence of eigenvalue crossing comes into play.
\item 
\emph{Impact of eigenvalue crossings.} 
Both plots exhibit non-smooth changes of the lines at {$\prmtr_1 - 1 \approx 0.0572$}. 
This point corresponds to the parameter value where the spectral gap 
of~$\mat{A}(\prmtrVec)$ vanishes (eigenvalue crossing),
cf.\! the lower right regions in \Cref{fig:expB_heatmaps}. 
\item 
\emph{Validation of eigenprojector error bounds.} 
\Cref{thm:Taylor_rbm_projector_error} is validated by \Cref{subfig:expB_subspace_error_rate}, 
as the computed approximation error is strictly bounded by the derived estimator across the stable region of the parameter domain. 
We observe that the gap between the estimator and the actual error increases with the order $\Ord$. 
{This is consistent with the numerical observation that the term
involving the squared residual
$\rm{Res}^2(\prmtrVec,{\RBproj}_\Ord)$ increases with $\Ord$ in the
tested range.}
\item 
\emph{Quadratic relationship of spectral errors.} 
\Cref{thm:energy_difference_rb} is supported by this numerical example,
as the Ritz-value convergence rate in \Cref{subfig:expB_ritz_error_rate} is approximately double that of the corresponding subspace convergence in \Cref{subfig:expB_subspace_error_rate}.
\item 
\emph{Robustness of higher-order spaces.} 
Despite the loss of theoretical guarantees beyond the crossing point, 
increasing the order to $\Ord=3$ significantly reduces the absolute error magnitude compared to lower-order approximations. 
This indicates that the Taylor-RB space seems to be an expressive subspace for the approximation of the eigenprojector $\mat{P}(\prmtrVec)$ in the pre-asymptotic regime.
\end{itemize}

\begin{figure}[htbp]
    \centering
    \begin{subfigure}[b]{0.48\textwidth} 
        \centering
        \includegraphics[width=0.9\textwidth, trim = 0mm 3mm 0mm 6mm]{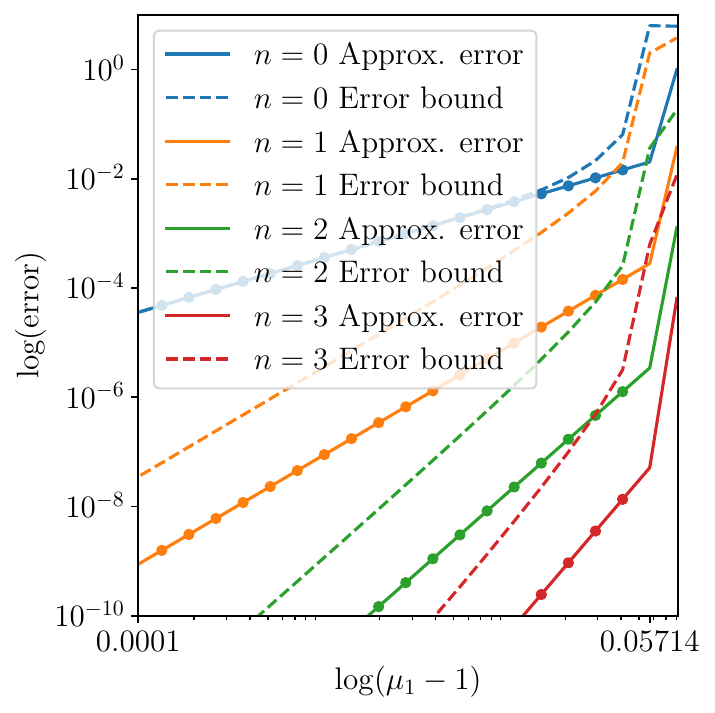}
        \caption{{The left- and right-hand sides} of \Cref{thm:Taylor_rbm_projector_error} of order $\Ord$ along the horizontal parameter path.}
        \label{subfig:expB_subspace_error_rate}
    \end{subfigure}
    \hfill
    \begin{subfigure}[b]{0.48\textwidth}
        \centering
        \includegraphics[width=0.9\textwidth, trim = 0mm 3mm 0mm 6mm]{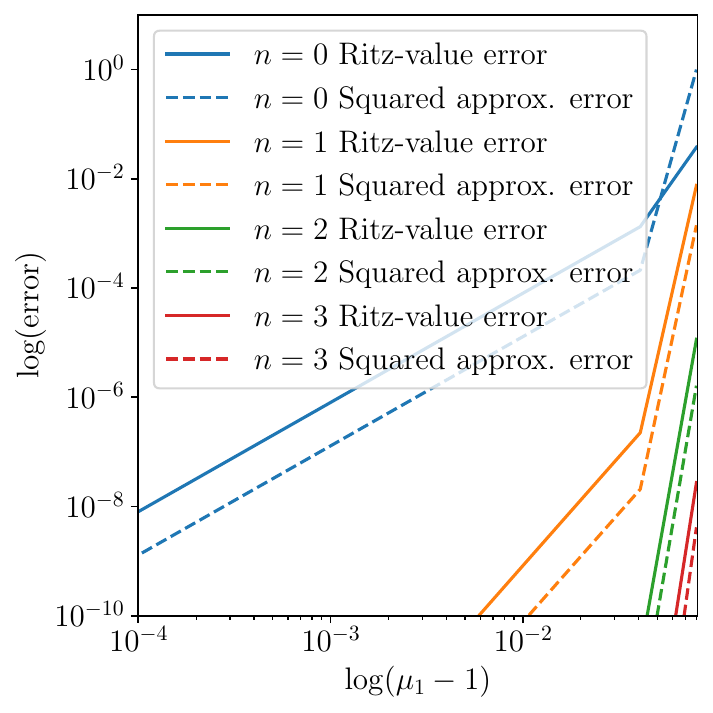}
        \caption{The Ritz-value error and the squared approximation error of order $\Ord$ along the horizontal parameter path.}
        \label{subfig:expB_ritz_error_rate}
    \end{subfigure}
    \caption{xxz-model of length $L=15$: 
             Sections of the error curves of the Taylor-RBM with increasing derivative order $\Ord$ along the parameter path from $(1,1)^{\!\top}$ to $(1.08, 1)^{\!\top}$.
             \rev{In panel~(a), 
             ``Approx. error'' 
             refers to
             $\|\mat{Q}_{\rm{rb}}(\prmtrVec)-\mat{P}(\prmtrVec)\|$,
             and ``Error bound'' to
             $C_{\rm{rb}}(\prmtrVec)
             \|{\RBproj}_\Ord^\perp\mat{P}(\prmtrVec)\|$. 
             In panel~(b), ``Ritz-value error''
              refers to
             $\tr \, \mat{A}(\prmtrVec)
             \bigl(\mat{Q}_{\rm{rb}}(\prmtrVec)-\mat{P}(\prmtrVec)\bigr)$,
             and ``Squared approx. error'' to
             $\|\mat{Q}_{\rm{rb}}(\prmtrVec)-\mat{P}(\prmtrVec)\|^2$.
             }}
    \label{fig:expB_convergence}
\end{figure}

Overall, the numerical experiments confirm the theoretical error estimates and demonstrate the practical effectiveness of the Taylor-RBM for approximating eigenprojectors in parametric EVPs.

\subsection{Comparison with classical perturbation theory} 
\label{subsec:comparison_PT}

{As shown in \Cref{subsec:truncated_projector_series}},
the approximation of eigenspaces using PT requires explicitly building
{the finite Taylor polynomial of the projector, which is generally
a sum of dense matrices of the same size as the original problem}.
This makes the PT approach computationally infeasible for large-scale problems.
Hence, we restrict ourselves to a smaller xxz-model of length $L=10$ (underlying space of dimension $2^{10} = 1\,024$) to perform a direct comparison with Taylor-RBM. 
The smallest eigenvalue cluster ($K=1$) of the parameter point $\prmtrVec_0 = (-1,0)^{\!\top}$ is selected for the comparison, 
as the smallest eigenvalue has multiplicity $11$ at this point.
The tolerance for identifying linear dependence during the orthogonalization process is set to $10^{-16}$.
\rev{The matrix $\mat{P}^{[\Ord]}(\prmtrVec)$ is assembled from the
scaled derivatives of $\mat{P}(\prmtrVec)$ at $\prmtrVec_0$ computed with the recursion
\eqref{eq:recursive_relation_projector_coefficients_appendix}, and
$\mat{Q}_{\rm{pt}}(\prmtrVec)$ is extracted as described in
\Cref{rem:practical_computation_Q_pt}.}

\subsubsection{Qualitative validation}
For a uniform grid of $71 \times 71 = 5\,041$ parameters from the box 
$ \big\{ \prmtrVec \in \bb{R}^2 : \prmtrVec - (-1,0)^{\!\top} \in [-0.03,\, 0.03]^2  \big\} $,
the approximation errors of order $\Ord \leq 3$ using Taylor-RBM $\|\mat{Q}_\rm{rb}(\prmtrVec) - \mat{P}(\prmtrVec)\|$
and PT $\|\mat{Q}_\rm{pt}(\prmtrVec) - \mat{P}(\prmtrVec)\|$ are computed. 
The resulting Taylor-RB spaces have dimensions $11, 33, 66, 110$ for {$\Ord=0,1,2,3$}, respectively. 
It is also worth noting that even in this smaller setting, 
the computational time of the PT approach is at least $45\%$ longer than that of the Taylor-RBM,
and this relative time ratio is observed to grow as $\Ord$ increases. 
The numerical results are shown in \Cref{fig:expC_heatmaps},
and {they demonstrate} the local convergence properties and the global behavior of both approximation schemes. 
Several mathematical features are noteworthy:

\begin{itemize}
\item 
\emph{Spectral gap and convergence domain.} 
The boundary of the characteristic "olive-shaped" region of low error corresponds to the parameter domain, 
where the spectral gap between the 11th and 12th eigenvalues (counting multiplicities) vanishes.
Beyond this boundary, 
we would not expect the classical PT to be applicable anymore. 

\item 
\emph{Verification of local convergence.} 
Within the interior of the "olive-shaped" region, 
both the Taylor-RBM and the PT approach demonstrate systematic error reduction as the order $\Ord$ increases. 
This confirms that the ranges of the Taylor derivatives of the eigenprojector effectively capture the locally varying eigenspace. 
The persistent dark vertical line through $\prmtrVec_0 = (-1,0)^{\!\top}$ indicates a direction of minimal spectral variation, 
which is a physical property of the model that has also been observed in \Cref{fig:expB_heatmaps}.

\item 
\emph{Extrapolation and subspace expressivity.} 
A significant distinction arises outside the convergence region. 
While the PT error stagnates or diverges 
— {as it is restricted by the convergence domain of the analytic power series} — 
the Taylor-RBM continues to reduce the error in certain exterior regions as $\Ord$ increases. 
This again suggests that the Taylor-RB space would possess an expressivity that extends beyond the convergence radius of the series itself. 
\end{itemize}

\begin{figure}[htbp]
    \centering
    \begin{subfigure}[b]{\textwidth}
        \centering
        \includegraphics[width=\textwidth, trim = 0mm 5mm 0mm 10mm]{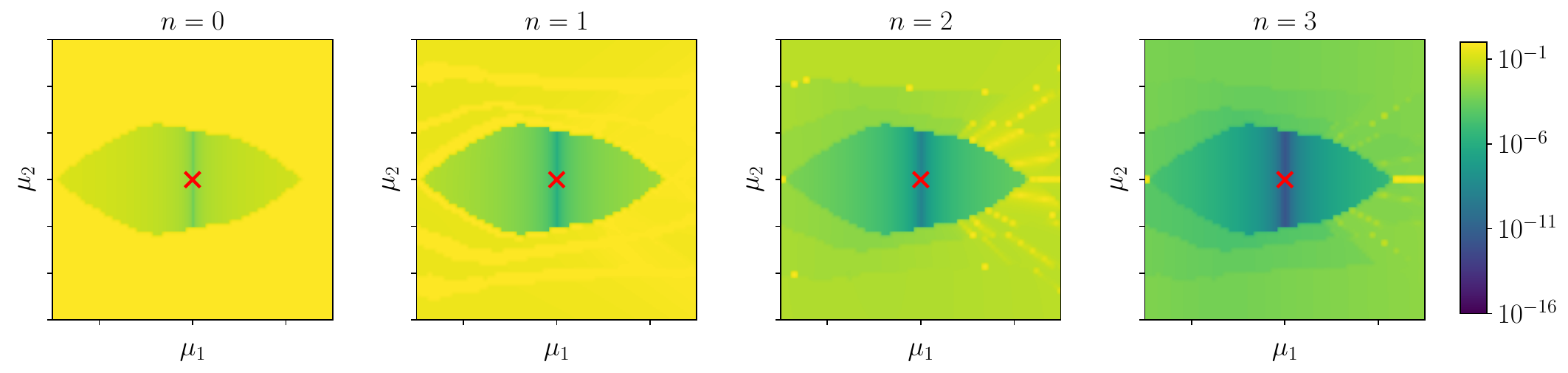}
        \caption{Approximation error $\|\mat{Q}_\rm{rb}(\prmtrVec) - \mat{P}(\prmtrVec)\|$ of the Taylor-RBM of order $\Ord$.}
    \end{subfigure}\\
    \vspace{0.2cm}
    \begin{subfigure}[b]{\textwidth}
        \centering
        \includegraphics[width=\textwidth, trim = 0mm 5mm 0mm 0mm]{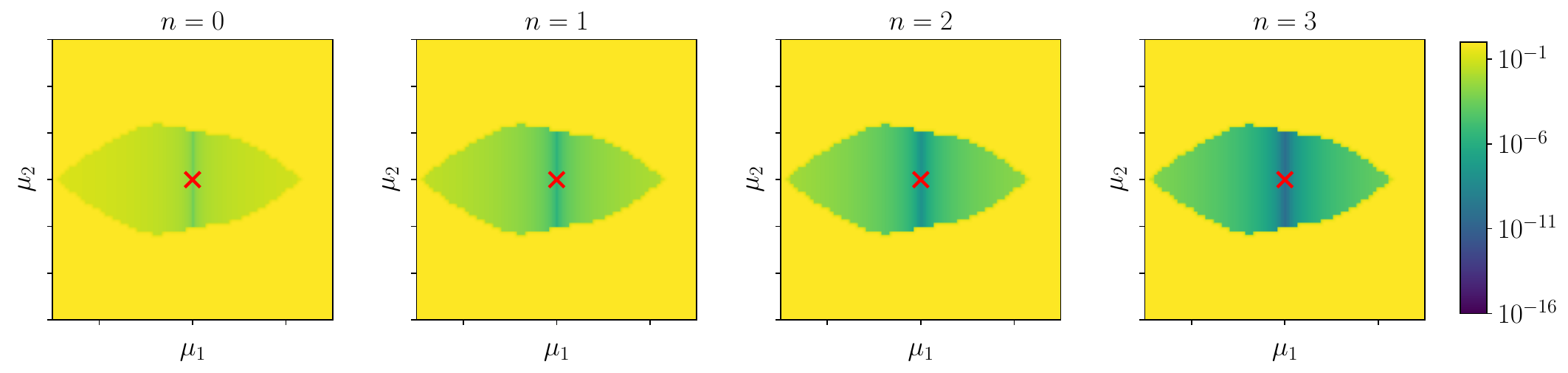}
        \caption{Approximation error $\|\mat{Q}_\rm{pt}(\prmtrVec) - \mat{P}(\prmtrVec)\|$ of the PT approach.}
    \end{subfigure}
    \caption{xxz-model of length $L=10$: 
            Approximation errors of the total eigenspace of rank $m_1 $ at parameters $\prmtrVec$ with $\prmtrVec - \prmtrVec_0 \in [-0.03,\, 0.03]^2$ 
            using Taylor-RBM and PT of order $\Ord \leq 3$ concerning the smallest eigenvalue cluster at $\prmtrVec_0  = (-1,0)^{\!\top}$, 
            which is marked in red and has {multiplicity} $m_1 = 11$. 
            } 
    \label{fig:expC_heatmaps}
\end{figure}

\subsubsection{Quantitative validation}
To further compare the methods, 
we again focus on the parameter path from $\prmtrVec_0 = (-1,0)^{\!\top}$ to $\prmtrVec = (-0.97, 0)^{\!\top}$,
i.e., the horizontal center line starting at $(-1,0)^{\!\top}$ of the heatmaps in \Cref{fig:expC_heatmaps}.
The selected results of $\prmtr_1 + 1 \in [10^{-5},3\cdot 10^{-2}]$ are shown in \Cref{fig:expC_comparison}.

\begin{figure}[htbp]
    \centering

        \centering
        \includegraphics[width=0.9\textwidth, trim = 5mm 8mm 5mm 5mm]{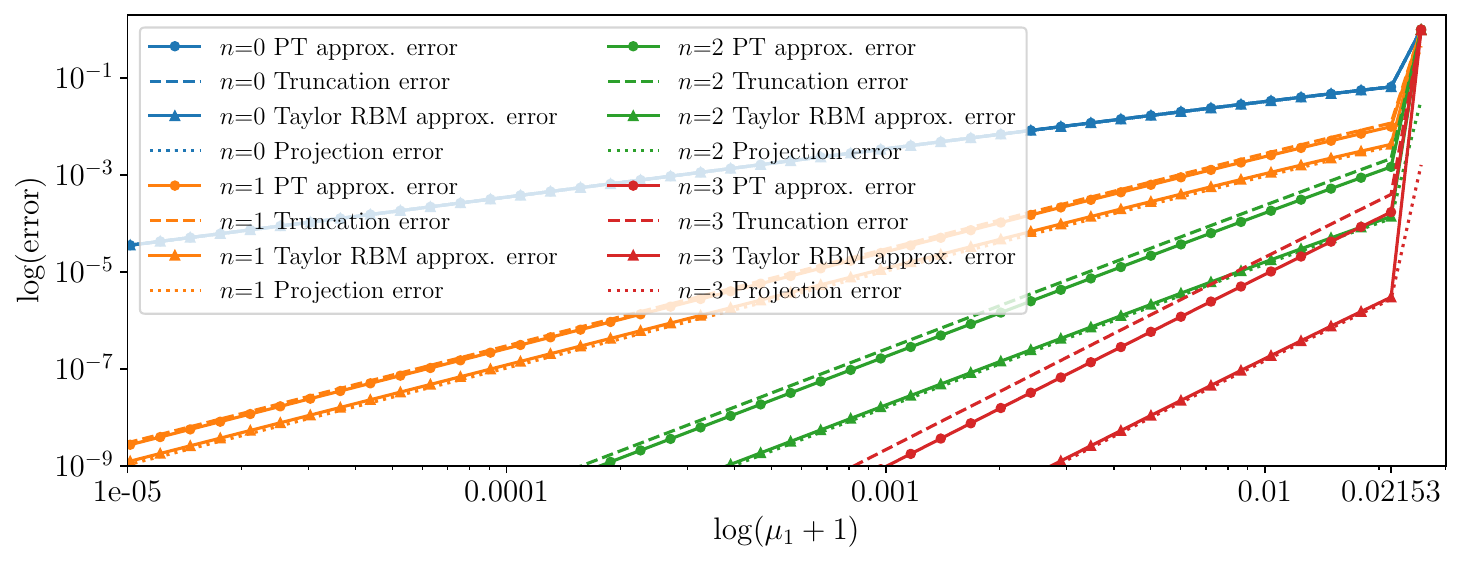}
    \caption{xxz-model of length $L=10$: 
            Sections of the error curves of the Taylor-RBM with increasing derivative order $\Ord$ along the horizontal parameter path from $(-1,0)^{\!\top}$ to $(-0.97, 0)^{\!\top}$.
            "PT approx.\! error" refers to $\| {\mat{Q}_{\rm{pt}}}(\prmtrVec) - \mat{P}(\prmtrVec) \|$, 
            and "truncation error" is $\| \mat{P}^{[\Ord]}(\prmtrVec) - \mat{P}(\prmtrVec) \|$, 
            cf.\! the proof of \Cref{lem:closedness_test_space}.  
            By "Taylor-RBM approx.\! error" and "projection error", 
            we mean $\| {\mat{Q}_{\rm{rb}}}(\prmtrVec) - \mat{P}(\prmtrVec) \|$ and $\|{\RBproj}_\Ord^\perp \mat{P}(\prmtrVec)\|$, respectively,
            cf.\! \Cref{thm:Taylor_rbm_projector_error}.} 
    \label{fig:expC_comparison}
\end{figure}

We make the following observations from the error curves displayed in \Cref{fig:expC_comparison}:
\begin{itemize}
\item 
\emph{Asymptotic convergence rates.} 
For $\prmtr_1 + 1 < 0.0215$, 
the error curves in the log-log plot exhibit constant slopes, 
confirming algebraic convergence. 
The empirical slopes of the Taylor-RBM ($\approx 0.98, 1.96, 2.92, 3.86$) and of the PT approach ($\approx 0.98, 1.97, 2.94, 3.90$) for orders $\Ord = 0, 1, 2, 3$, respectively, 
are in good agreement with the theoretical prediction of
{$\bigO(|\prmtrVec - \prmtrVec_0|^{\Ord+1})$ convergence} for the eigenprojector approximation.

\item 
\emph{Identification of eigenvalue crossing.} 
The sharp change of slope at $\prmtr_1 + 1 \approx 0.0215$ marks the exit from the guaranteed analytical convergence domain (the ``olive-shaped'' region in \Cref{fig:expC_heatmaps}).
This point corresponds to the parameter value where the spectral gap between the target cluster and the rest of the spectrum vanishes.

\item 
\emph{Consistency with the theory.} 
{In this experiment, the PT approximation error remains below the
truncation error, while the Taylor-RBM approximation error remains below
the projection error. This observed ordering is consistent with
\Cref{lem:closedness_test_space} and \Cref{thm:Taylor_rbm_projector_error}, which
control the respective errors by these comparison quantities up to the
stated prefactors.}
Notably, the gap between the truncation error and the projection error widens as $\Ord$ increases. 
This is a non-surprising effect due to the rapid growth of the reduced basis dimension, 
as the projection error automatically decreases with an increasing subspace.

\item 
\emph{Optimality of the reduced basis.} 
The Taylor-RBM approximation error closely tracks the projection error, 
while the PT error tracks the truncation error. 
Since the projection error is by definition the best approximation within the span of the ranges of the Taylor-derivatives, 
we may expect that the Taylor-RBM would tend to outperform the PT approach, 
provided that the error prefactor in \Cref{thm:Taylor_rbm_projector_error} is not too large. 
This performance gap becomes more pronounced at higher orders, 
and this is again aligned with the fact that the dimension of the Taylor-RB space grows fast. 

\end{itemize}

Although not presented, 
the Ritz-value sums for the Taylor-RBM are lower than those of the PT approach, 
as dictated by the variational principle. 
All in all, 
the Taylor-RBM is theoretically much less expensive to assemble for large-scale systems 
and empirically tends to provide superior approximation quality both within and slightly beyond the analytical convergence radius, 
so it is mathematically and computationally preferred over classical PT for eigenspace approximations.


\section{Conclusion and outlook}\label{sec:conclusion}

The Taylor-reduced basis method (Taylor-RBM) for multivariate parametric
eigenvalue problems is a local model order reduction technique that
constructs an approximation space 
\rev{from the ranges of scaled derivatives
of the total eigenprojector at a reference parameter 
up to a prescribed order.}
It combines the Rayleigh--Ritz method for spectral approximation,
finite-order perturbation analysis of total eigenprojectors and the
reduced basis method.

In this work, we have firstly revisited subspace methods for spectral approximations in \Cref{sec:rayleigh_ritz},
where we have in particular emphasized \Cref{thm:Cea_EVP}, 
which bounds the eigenspace approximation error in terms of eigenspace projection error,
as well as \Cref{lem:energy_difference_general},
which concisely shows the squaring effect of the Ritz-value sum approximation error w.r.t.~the eigenspace approximation error. 

\rev{Then, in \Cref{sec:analyticity}, we established 
the finite-order regularity of total eigenprojectors and provided explicit formulae for
computing their scaled derivatives directly or recursively.
We also constructed and analyzed a rank-\(M\) projector approximation 
obtained from the finite Taylor polynomial.}

With all preparations, 
we have defined and analyzed the convergence properties of the Taylor-RBM for pEVPs in \Cref{sec:Taylor_rbm}. 
More importantly, 
we have discussed and derived the computationally efficient algorithm for assembling the Taylor reduced basis space in \Cref{subsec:practical_assembly_taylor_rbm}. 

Finally, we have presented some numerical experiments in \Cref{sec:num_exp} using the xxz-chain model 
to illustrate the Taylor-RBM and to compare it with the spectral approximation using the finite Taylor-approximation of the eigenprojector.
Our numerical results have shown 
that the Taylor-RBM can achieve a much better approximation than 
a rank-\(M\) projector approximation obtained from the finite Taylor polynomial,
even at parameter regime beyond theoretical convergence guarantee. 

\rev{The principal strength of the Taylor-RBM 
is its guaranteed high-order local accuracy 
near the reference parameter: 
an order-\(\Ord\) construction yields 
eigenspace and Ritz-value sum errors 
of orders \(\Ord+1\) and \(2\Ord+2\), 
respectively. 
This makes the method particularly attractive, 
when accurate repeated approximations 
are required near a prescribed parameter. 
Combining Taylor-RB spaces constructed 
at adaptively selected reference parameters 
to cover larger parameter domains 
is a natural direction for future work.}


\subsection*{Acknowledgments}
The authors express their sincere gratitude 
to Prof.~Dr.~Timo Weidl and Prof.~Dr.~Jens Wirth 
for valuable discussions on perturbation theory, 
and to the anonymous reviewers for their careful reading and detailed suggestions, 
which in particular helped sharpen the finite-smoothness analysis
and improve the overall exposition.

Funded by Deutsche Forschungsgemeinschaft (DFG, German Research Foundation) under Germany's Excellence Strategy - EXC 2075 - 390740016. 
We acknowledge the support by the Stuttgart Center for Simulation Science (SimTech).

The authors used generative AI tools during the preparation of this paper for brainstorming,
theorem revision, improving exposition, numerical implementation and language revision. All
AI-assisted output was critically reviewed and revised. The authors take full responsibility for the
mathematical arguments and proofs, numerical results, and scientific conclusions of this work.

\bibliographystyle{macros/plain-doi}
\bibliography{macros/journalabbr, macros/Literature}    

@book{Ant05,
  author    = {Antoulas, Athanasios C.},
  title     = {Approximation of Large-Scale Dynamical Systems},
  publisher = {Society for Industrial and Applied Mathematics (SIAM)},
  year      = {2005},
  series    = {Advances in Design and Control},
  isbn      = {978-0-89871-658-0},
  doi       = {10.1137/1.9780898718713}
}

@Article{AlgBB25,
  author   = {Alghamdi, M. and Boffi, D. and Bonizzoni, F.},
  journal  = CompApplMath,
  title    = {A greedy {MOR} method for the tracking of eigensolutions to parametric elliptic {PDE}s},
  year     = {2025},
  issn     = {0377-0427},
  pages    = {116270},
  volume   = {457},
  doi      = {10.1016/j.cam.2024.116270},
  url      = {https://www.sciencedirect.com/science/article/pii/S0377042724005193},
}

@article{BenZA24,
  author  = {Bendokat, Thomas and Zimmermann, Ralf and Absil, P.-A.},
  title   = {A Grassmann Manifold Handbook: Basic Geometry and
             Computational Aspects},
  journal = {Advances in Computational Mathematics},
  volume  = {50},
  number  = {1},
  pages   = {6},
  year    = {2024},
  doi     = {10.1007/s10444-023-10090-8}
}

@incollection{AndS12,
	address = {Berlin, Heidelberg},
	title = {Sparse {tensor} {approximation} of {parametric} {eigenvalue} {problems}},
	volume = {83},
	isbn = {978-3-642-22060-9 978-3-642-22061-6},
	url = {https://link.springer.com/10.1007/978-3-642-22061-6_7},
	language = {en},
	urldate = {2025-08-03},
	booktitle = {Numerical {Analysis} of {Multiscale} {Problems}},
	publisher = {Springer Berlin Heidelberg},
	author = {Andreev, R. and Schwab, C.},
	editor = {Graham, Ivan G. and Hou, Thomas Y. and Lakkis, Omar and Scheichl, Robert},
	year = {2012},
	doi = {10.1007/978-3-642-22061-6_7},
	note = {Series Title: Lecture Notes in Computational Science and Engineering},
	pages = {203--241},
}

@incollection{BabO91,
  author    = {Babu{\v{s}}ka, Ivo and Osborn, John E.},
  title     = {Eigenvalue problems},
  booktitle = {Handbook of Numerical Analysis},
  volume    = {2},
  editor    = {Ciarlet, P. G. and Lions, J.-L.},
  publisher = {Elsevier (North-Holland)},
  address   = {Amsterdam},
  year      = {1991},
  pages     = {641--787},
  doi       = {10.1016/S1570-8659(05)80042-0}
}

@book{Bau84,
url = {https://doi.org/10.1515/9783112721810},
title = {Analytic Perturbation Theory for Matrices and Operators},
author = {H. Baumgärtel},
publisher = {De Gruyter},
address = {Berlin, Boston},
doi = {doi:10.1515/9783112721810},
isbn = {9783112721810},
year = {1984},
lastchecked = {2025-09-04}
}

@Article{BenGW15,
  author  = {Benner, P. and Gugercin, S. and Willcox, K.},
  journal = SIAMReview,
  title   = {A Survey of Projection-Based Model Reduction Methods for Parametric Dynamical Systems},
  year    = {2015},
  month   = jun,
  pages   = {483--531},
  volume  = {57},
  doi     = {10.1137/130932715},
}

@book{BGTetal20,
  title     = {Snapshot-Based Methods and Algorithms},
  booktitle = {Model Order Reduction},
  editor    = {Benner, Peter and Grivet-Talocia, Stefano and Quarteroni, Alfio and Rozza, Gianluigi and Schilders, Wil and Silveira, Lu{\'i}s Miguel},
  publisher = {De Gruyter},
  location  = {Berlin and Boston},
  date      = {2020},
  year      = {2020},
  volume    = {2},
  doi       = {10.1515/9783110671490},
  isbn      = {9783110671490},
  url       = {https://www.degruyter.com/serial/mor-b/html}
}

@book{BTGetal20total,
  title     = {Model Order Reduction},
  editor    = {Benner, Peter and Grivet-Talocia, Stefano and Quarteroni, Alfio and Rozza, Gianluigi and Schilders, Wil and Silveira, Lu{\'i}s Miguel},
  publisher = {De Gruyter},
  location  = {Berlin and Boston},
  date      = {2020/2021},
  year      = {2020},
  volumes   = {3},
  url       = {https://www.degruyter.com/serial/mor-b/html},
  note      = {Multi-volume handbook. Vol.~1 DOI: 10.1515/9783110498967; Vol.~2 DOI: 10.1515/9783110671490; Vol.~3 DOI: 10.1515/9783110499001}
}

@article{BhaR97,
  author  = {Bhatia, Rajendra and Rosenthal, Peter},
  title   = {How and Why to Solve the Operator Equation {$AX - XB = Y$}},
  journal = {Bulletin of the London Mathematical Society},
  year    = {1997},
  volume  = {29},
  number  = {1},
  pages   = {1--21},
  doi     = {10.1112/S0024609396001828},
  url     = {https://doi.org/10.1112/S0024609396001828}
}

@book{BirS87,
  title     = {Spectral Theory of Self-Adjoint Operators in Hilbert Space},
  author    = {Birman, M. S. and Solomjak, M. Z.},
  year      = {1987},
  publisher = {D. Reidel Publishing Company},
  address   = {Dordrecht},
  series    = {Mathematics and Its Applications (Soviet Series)},
  volume    = {5},
  doi       = {10.1007/978-94-009-4586-9},
  isbn      = {978-94-009-4586-9},
  url       = {https://link.springer.com/book/10.1007/978-94-009-4586-9}
}

@article{Bof10, 
title={Finite element approximation of eigenvalue problems}, 
volume={19}, 
DOI={10.1017/S0962492910000012}, 
journal=ActaNumer, 
author={Boffi, Daniele}, 
year={2010}, 
pages={1-120}
}

@book{BOPRU17,
  editor    = {Benner, Peter and Ohlberger, Mario and Patera, Anthony and Rozza, Gianluigi and Urban, Karsten},
  title     = {Model Reduction of Parametrized Systems},
  publisher = {Springer},
  address   = {Cham},
  year      = {2017},
  series    = {MS\&A (Modeling, Simulation and Applications)},
  volume    = {17},
  isbn      = {978-3-319-58785-1},
  doi       = {10.1007/978-3-319-58786-8}
}

@article{BenR88,
  author  = {Benaroya, Haym and Rehak, Mark},
  title   = {Finite Element Methods in Probabilistic Structural Analysis: A Selective Review},
  journal = {Applied Mechanics Reviews},
  year    = {1988},
  volume  = {41},
  number  = {5},
  pages   = {201--213},
  month   = may,
  doi     = {10.1115/1.3151892},
  url     = {https://doi.org/10.1115/1.3151892}
}

@book{Bra07,
  author    = {Braess, Dietrich},
  title     = {Finite {Elements}: {Theory}, {Fast} {Solvers}, and {Applications} in {Solid} {Mechanics}},
  edition   = {3},
  year      = {2007},
  publisher = {Cambridge University Press},
  address   = {Cambridge},
  isbn      = {9780521705189},
  doi       = {10.1017/CBO9780511618635}
}

@Article{BHWRS23,
  author    = {Brehmer, P. and Herbst, M. F. and Wessel, S. and Rizzi, M. and Stamm, B.},
  journal   = {Phys. Rev. E},
  title     = {Reduced basis surrogates for quantum spin systems based on tensor networks},
  year      = {2023},
  month     = {Aug},
  pages     = {025306},
  volume    = {108},
  doi       = {10.1103/PhysRevE.108.025306},
  issue     = {2},
  numpages  = {14},
  publisher = {American Physical Society},
  url       = {https://link.aps.org/doi/10.1103/PhysRevE.108.025306},
}

@article{BMPetal12,
  author  = {Buffa, Annalisa and Maday, Yvon and Patera, Anthony T. and Prud'homme, Christophe and Turinici, Gabriel},
  title   = {A priori convergence of the greedy algorithm for the parametrized reduced basis method},
  journal = {ESAIM: Mathematical Modelling and Numerical Analysis},
  year    = {2012},
  volume  = {46},
  number  = {3},
  pages   = {595--603},
  publisher = {EDP Sciences},
  doi     = {10.1051/m2an/2011056},
  url     = {https://doi.org/10.1051/m2an/2011056}
}

@article{CJT21,
  author  = {Castrill{\'o}n-Cand{\'a}s, Julio E. and Nobile, Fabio and Tempone, Ra{\'u}l F.},
  title   = {A Hybrid Collocation-Perturbation Approach for {PDE}s with Random Domains},
  journal = {Advances in Computational Mathematics},
  year    = {2021},
  volume  = {47},
  number  = {3},
  eid     = {40},
  doi     = {10.1007/s10444-021-09859-6},
  url     = {https://doi.org/10.1007/s10444-021-09859-6}
}

@article{DEFK23,
  author  = {Duguet, Thomas and Ekstr{\"o}m, Andreas and Furnstahl, Richard J. and K{\"o}nig, Sebastian and Lee, Dean},
  title   = {Colloquium: Eigenvector continuation and projection-based emulators},
  journal = {Reviews of Modern Physics},
  volume  = {96},
  number  = {3},
  pages   = {031002},
  year    = {2024},
  doi     = {10.1103/RevModPhys.96.031002},
}

@article{DEHKO04,
title = {Effective interactions and the nuclear shell-model},
journal = {Progress in Particle and Nuclear Physics},
volume = {53},
number = {2},
pages = {419-500},
year = {2004},
issn = {0146-6410},
doi = {https://doi.org/10.1016/j.ppnp.2004.05.001},
url = {https://www.sciencedirect.com/science/article/pii/S0146641004000912},
author = {D.J. Dean and T. Engeland and M. Hjorth-Jensen and M.P. Kartamyshev and E. Osnes}
}

@article{DoeE24,
   title={On Uncertainty Quantification of Eigenvalues and Eigenspaces with Higher Multiplicity},
   volume={62},
   ISSN={1095-7170},
   url={http://dx.doi.org/10.1137/22M1529324},
   DOI={10.1137/22m1529324},
   number={1},
   journal={SIAM Journal on Numerical Analysis},
   publisher={Society for Industrial & Applied Mathematics (SIAM)},
   author={Dölz, Jürgen and Ebert, David},
   year={2024},
   month=feb, pages={422–451} }

@article{Fan49,
 ISSN = {00278424, 10916490},
 author = {Ky Fan},
 journal = {Proceedings of the National Academy of Sciences of the United States of America},
 number = {11},
 pages = {652--655},
 publisher = {National Academy of Sciences},
 title = {On a Theorem of {Weyl} Concerning Eigenvalues of Linear Transformations. I},
 urldate = {2025-12-01},
 volume = {35},
 year = {1949}
}

@article{Fin83,
  author  = {Fink, J. P. and Rheinboldt, W. C.},
  title   = {On the Error Behavior of the Reduced Basis Technique for Nonlinear Finite Element Approximations},
  journal = {ZAMM -- Journal of Applied Mathematics and Mechanics / Zeitschrift f{\"u}r Angewandte Mathematik und Mechanik},
  year    = {1983},
  volume  = {63},
  number  = {1},
  pages   = {21--28},
  doi     = {10.1002/zamm.19830630105},
  url     = {https://onlinelibrary.wiley.com/doi/10.1002/zamm.19830630105}
}

@article{FHIetal18,
  author  = {Frame, Dillon and He, Rongzheng and Ipsen, Ilse and Lee, Daniel and Lee, Dean and Rrapaj, Ermal},
  title   = {Eigenvector Continuation with Subspace Learning},
  journal = {Physical Review Letters},
  year    = {2018},
  volume  = {121},
  number  = {3},
  pages   = {032501},
  month   = jul,
  doi     = {10.1103/PhysRevLett.121.032501},
  url     = {https://doi.org/10.1103/PhysRevLett.121.032501},
  publisher = {American Physical Society}
}

@article{FMPV16,
  author  = {Fumagalli, Ivan and Manzoni, Andrea and Parolini, Nicola and Verani, Marco},
  title   = {Reduced Basis Approximation and a Posteriori Error Estimates for Parametrized Elliptic Eigenvalue Problems},
  journal = {ESAIM: Mathematical Modelling and Numerical Analysis},
  year    = {2016},
  volume  = {50},
  number  = {6},
  pages   = {1857--1885},
  doi     = {10.1051/m2an/2016009}
}

@misc{GarS24,
  title={On reduced basis methods for eigenvalue problems, and on its coupling with perturbation theory}, 
  author={Louis Garrigue and Benjamin Stamm},
  year={2024},
  eprint={2408.11924},
  archivePrefix={arXiv},
  primaryClass={math-ph},
  note={arXiv:2408.11924},
  url={https://arxiv.org/abs/2408.11924}, 
}

@book{Gou12,
  author    = {Gould, Sydney Henry},
  title     = {Variational Methods for Eigenvalue Problems: An Introduction to the Methods of {Rayleigh}, {Ritz}, {Weinstein}, and {Aronszajn}},
  publisher = {Dover Publications},
  year      = {2012},
  isbn      = {9780486165806}
}

@article{GSH23,
	title = {Stochastic collocation method for computing eigenspaces of parameter-dependent operators},
	volume = {153},
	issn = {0029-599X, 0945-3245},
	url = {https://link.springer.com/10.1007/s00211-022-01339-3},
	doi = {10.1007/s00211-022-01339-3},
	language = {en},
	number = {1},
	urldate = {2024-09-24},
	journal = NumerMath,
	author = {Grubišić, Luka and Saarikangas, Mikael and Hakula, Harri},
	month = jan,
	year = {2023},
	pages = {85--110},
}

@incollection{Haa17,
  author    = {Haasdonk, Bernard},
  title     = {Reduced Basis Methods for Parametrized PDEs --- A Tutorial Introduction for Stationary and Instationary Problems},
  booktitle = {Model Reduction and Approximation: Theory and Algorithms},
  editor    = {Benner, Peter and Cohen, Albert and Ohlberger, Mario and Willcox, Karen},
  publisher = {SIAM},
  location  = {Philadelphia, PA},
  year      = {2017},
  chapter   = {2},
  pages     = {65--136},
  doi       = {10.1137/1.9781611974829.ch2},
  isbn      = {9781611974812}
}

@Article{HSWR22,
  author    = {Herbst, M. F. and Stamm, B. and Wessel, S. and Rizzi, M.},
  journal   = {Phys. Rev. E},
  title     = {Surrogate models for quantum spin systems based on reduced-order modeling},
  year      = {2022},
  month     = {Apr},
  pages     = {045303},
  volume    = {105},
  doi       = {10.1103/PhysRevE.105.045303},
  issue     = {4},
  numpages  = {13},
  publisher = {American Physical Society},
  url       = {https://link.aps.org/doi/10.1103/PhysRevE.105.045303},
}

@Book{HesRS16,
  author = {Hesthaven, J. and Rozza, G. and Stamm, B.},
  title  = {Certified Reduced Basis Methods for Parametrized Partial Differential Equations}, 
  publisher = {Springer},
  year   = {2016},
  month  = jan,
  doi    = {10.1007/978-3-319-22470-1},
  issn   = {978-3-319-22470-1},
}

@book{HubW10,
  author    = {Huba{\v{c}}, Ivan and Wilson, Stephen},
  title     = {Brillouin--Wigner Methods for Many-Body Systems},
  year      = {2010},
  publisher = {Springer Dordrecht},
  series    = {Progress in Theoretical Chemistry and Physics},
  volume    = {21},
  doi       = {10.1007/978-90-481-3373-4},
  isbn      = {978-90-481-3372-7},
  isbn      = {978-90-481-3373-4},
}

@article{HWD17,
  author  = {Horger, Thomas and Wohlmuth, Barbara and Dickopf, Thomas},
  title   = {Simultaneous Reduced Basis Approximation of Parameterized Elliptic Eigenvalue Problems},
  journal = {ESAIM: Mathematical Modelling and Numerical Analysis},
  year    = {2017},
  volume  = {51},
  number  = {2},
  pages   = {443--465},
  doi     = {10.1051/m2an/2016025}
}

@article{KanMMM18,
  author  = {Kangal, F. and Meerbergen, K. and Mengi, E. and Michiels, W.},
  journal = SIAMMatrix,
  title   = {A Subspace Method for Large Scale Eigenvalue Optimization},
  year    = {2017},
  number  = {1},
  pages   = {48-82},
  volume  = {39},
  url     = {https://doi.org/10.1137/16M1070025},
}

@book{Kato76,
  author        = "Kato, Tosio",
  title         = "{Perturbation theory for linear operators; 2nd ed.}",
  publisher     = "Springer",
  address       = "Berlin",
  series        = "Grundlehren der mathematischen Wissenschaften: a series of comprehensive studies in mathematics",
  year          = "1976",
}

@book{Kna16,
  author    = {Knapp, Anthony W.},
  title     = {Basic Real Analysis},
  edition   = {Digital Second},
  publisher = {Published by the Author},
  year      = {2016},
  isbn      = {978-1-4297-9999-7},
  url       = {https://www.math.stonybrook.edu/~aknapp/download/b2-realanal-clickable.pdf}
}

@article{KMX13,
  author  = {Knyazev, Andrew and Mehrmann, Volker and Xu, Jinchao},
  title   = {Numerical Solution of {PDE} Eigenvalue Problems},
  journal = {Oberwolfach Reports},
  volume  = {10},
  number  = {4},
  year    = {2013},
  pages   = {3221--3304},
  doi     = {10.4171/OWR/2013/56},
  url     = {https://publications.mfo.de/handle/mfo/3387}
}

@Article{KreV14,
  author   = {Kressner, D. and Vandereycken, B.},
  journal  = SIAMMatrix,
  title    = {Subspace Methods for Computing the Pseudospectral Abscissa and the Stability Radius},
  year     = {2014},
  number   = {1},
  pages    = {292-313},
  volume   = {35},
  doi      = {10.1137/120869432},
  eprint   = {https://doi.org/10.1137/120869432},
  url      = {https://doi.org/10.1137/120869432},
}

@Article{MacMOPR00,
  author   = {Machiels, L. and Maday, Y. and Oliveira, I. B. and Patera, A. and Rovas, D. V.},
  journal  = {Comptes Rendus de l'Académie des Sciences - Series I - Mathematics},
  title    = {Output bounds for reduced-basis approximations of symmetric positive definite eigenvalue problems},
  year     = {2000},
  issn     = {0764-4442},
  number   = {2},
  pages    = {153-158},
  volume   = {331},
  doi      = {https://doi.org/10.1016/S0764-4442(00)00270-6},
  url      = {https://www.sciencedirect.com/science/article/pii/S0764444200002706},
}

@article{McW62,
  author  = {McWeeny, R.},
  title   = {Perturbation Theory for the {Fock}-{Dirac} Density Matrix},
  journal = {Physical Review},
  year    = {1962},
  volume  = {126},
  number  = {3},
  pages   = {1028--1034},
  month   = may,
  doi     = {10.1103/PhysRev.126.1028},
  url     = {https://doi.org/10.1103/PhysRev.126.1028}
}

@article{MSZ25,
  author  = {Manucci, Mattia and Stamm, Benjamin and Zeng, Zhuoyao},
  title   = {Certified model order reduction for parametric {Hermitian} eigenproblems},
  journal = {Mathematics of Computation},
  year    = {2026},
  doi     = {10.1090/mcom/4204},
  url     = {https://doi.org/10.1090/mcom/4204},
  note    = {Published electronically ahead of print},
}

@article{MMMV17,
	Author = {Meerbergen, K. and Mengi, E. and Michiels, W. and Van Beeumen, R.},
	Journal = {IMA J. Numer. Anal.},
	Number = {4},
	Pages ={1831-1863},
	Title = {Computation of Pseudospectral Abscissa for Large-scale Nonlinear Eigenvalue Problems},
	Volume ={37},
	Year = {2017},
    url = {https://doi.org/10.1093/imanum/drw065},
}

@article{MDFGZ22,
  author  = {Melendez, J. A. and Drischler, Christian and Furnstahl, Richard J. and Garcia, A. J. and Zhang, Xilin},
  title   = {Model reduction methods for nuclear emulators},
  journal = {Journal of Physics G: Nuclear and Particle Physics},
  volume  = {49},
  number  = {10},
  pages   = {102001},
  year    = {2022},
  doi     = {10.1088/1361-6471/ac83dd},
}

@article{NooP80,
  author  = {Noor, A. K. and Peters, J. M.},
  title   = {Reduced Basis Technique for Nonlinear Analysis of Structures},
  journal = {AIAA Journal},
  year    = {1980},
  volume  = {18},
  number  = {4},
  pages   = {455--462},
  doi     = {10.2514/3.50778},
  url     = {https://doi.org/10.2514/3.50778}
}

@article{Ovt06I,
  author  = {Ovtchinnikov, Evgueni},
  title   = {Cluster robust error estimates for the {Rayleigh}--{Ritz} approximation {I}: Estimates for invariant subspaces},
  journal = {Linear Algebra and its Applications},
  year    = {2006},
  volume  = {415},
  number  = {1},
  pages   = {167--187},
  month   = may,
  doi     = {10.1016/j.laa.2005.06.040},
  url     = {https://doi.org/10.1016/j.laa.2005.06.040},
  issn    = {0024-3795}
}

@book{Par98,
  address = {Philadelphia},
  author = {Parlett, Beresford N.},
  isbn = {0898714028},
  publisher = {Society for Industrial and Applied Mathematics},
  refid = {37675815},
  title = {The Symmetric Eigenvalue Problem},
  year = 1998,
  url = {https://epubs.siam.org/doi/book/10.1137/1.9781611971163},
}

@article{PorL87,
  author  = {Porsching, T. A. and Lee, M. Lin},
  title   = {The Reduced Basis Method for Initial Value Problems},
  journal = {SIAM Journal on Numerical Analysis},
  year    = {1987},
  volume  = {24},
  number  = {6},
  pages   = {1277--1287},
  month   = dec,
  doi     = {10.1137/0724083},
  url     = {https://doi.org/10.1137/0724083}
}

@Article{PraB24,
  author   = {Pradovera, D. and Borghi, A.},
  journal  = JComputPhy,
  title    = {Match-based solution of general parametric eigenvalue problems},
  year     = {2024},
  issn     = {0021-9991},
  pages    = {113384},
  volume   = {519},
  doi      = {10.1016/j.jcp.2024.113384},
  url      = {https://www.sciencedirect.com/science/article/pii/S0021999124006326},
}

@Article{PruRVP02,
  author    = {Prud'homme, C. and Rovas, D. V. and Veroy, K. and Patera, A. T.},
  journal   = E2MNA,
  title     = {A mathematical and computational framework for reliable real-time solution of parametrized partial differential equations},
  year      = {2002},
  number    = {5},
  pages     = {747--771},
  volume    = {36},
  doi       = {10.1051/m2an:2002035},
  language  = {en},
  publisher = {EDP-Sciences},
  url       = {http://www.numdam.org/articles/10.1051/m2an:2002035/},
  zbl       = {1024.65104},
}

@techreport{Rel54,
  author      = {Rellich, Franz},
  title       = {Perturbation Theory of Eigenvalue Problems},
  institution = {Institute of Mathematical Sciences, New York University},
  address     = {New York},
  year        = {1954},
  month       = aug,
  number      = {IMM-NYU 212},
  type        = {Technical Report},
  note        = {Research in the Field of Perturbation Theory and Linear Operators, Technical Report No. 1; prepared under sponsorship of the Army Office of Ordnance Research},
  url         = {https://archive.org/details/perturbationtheo00rell},
}

@article{Rhe93,
  author  = {Rheinboldt, Werner C.},
  title   = {On the Theory and Error Estimation of the Reduced Basis Method for Multi-Parameter Problems},
  journal = {Nonlinear Analysis: Theory, Methods \& Applications},
  year    = {1993},
  volume  = {21},
  number  = {11},
  pages   = {849--858},
  month   = dec,
  doi     = {10.1016/0362-546X(93)90050-3},
  url     = {https://doi.org/10.1016/0362-546X(93)90050-3}
}

@book{SVR08,
  editor    = {Schilders, Wilhelmus H. A. and van der Vorst, Henk A. and Rommes, Joost},
  title     = {Model Order Reduction: Theory, Research Aspects and Applications},
  publisher = {Springer},
  address   = {Berlin, Heidelberg},
  year      = {2008},
  series    = {Mathematics in Industry},
  volume    = {13},
  isbn      = {978-3-540-78840-9},
  doi       = {10.1007/978-3-540-78841-6}
}

@Article{SirK16,
  author  = {Sirkovic, P. and Kressner, D.},
  journal = SIAMMatrix,
  title   = {Subspace Acceleration for Large-Scale Parameter-Dependent {Hermitian} Eigenproblems},
  year    = {2016},
  number  = {2},
  pages   = {1323--1353},
  volume  = {37},
  url = {https://doi.org/10.1137/15M1017181},
}

@article{ShiA72,
  author  = {Shinozuka, Masanobu and Astill, Clifford J.},
  title   = {Random Eigenvalue Problems in Structural Analysis},
  journal = {AIAA Journal},
  year    = {1972},
  volume  = {10},
  number  = {4},
  pages   = {456--462},
  doi     = {10.2514/3.50119},
  url     = {https://doi.org/10.2514/3.50119}
}

@article{Sun91,
  author  = {Sun, Ji-guang},
  title   = {Eigenvalues of {Rayleigh} quotient matrices},
  journal = {Numerische Mathematik},
  year    = {1991},
  volume  = {59},
  number  = {6},
  pages   = {603--614},
  month   = dec,
  doi     = {10.1007/BF01385798},
  url     = {https://doi.org/10.1007/BF01385798}
}

@article{Wed83,
    author = "Wedin, Per Ake",
    title = "{On Angles between subspaces of a finite dimensional inner product space}",
    journal = "Lect. Notes Math.",
    volume = "973",
    pages = "263--285",
    year = "1983",
    url = {https://link.springer.com/content/pdf/10.1007/BFb0062089.pdf}
}

@book{Wei74,
  author    = {Weinberger, H. F.},
  title     = {Variational Methods for Eigenvalue Approximation},
  series    = {CBMS-NSF Regional Conference Series in Applied Mathematics},
  volume    = {15},
  publisher = {SIAM},
  year      = {1974},
  doi       = {10.1137/1.9781611970531},
  isbn      = {9780898710120},
  eisbn     = {9781611970531}
}

@book{Tao12,
  author    = {Tao, Terence},
  title     = {Topics in Random Matrix Theory},
  publisher = {American Mathematical Society},
  year      = {2012},
  series    = {Graduate studies in mathematics ; v. 132},
  isbn      = {978-0-8218-7430-1}
}

@book{SRFB2004,
title = "Quantum Magnetism",
editor = "U Schollw{\"o}ck and J Richter and DJJ Farnell and RF Bishop",
year = "2004",
doi = "10.1007/b96825",
language = "English",
isbn = "3-540-21422-4",
series = "Lecture Notes in Physics",
publisher = "Springer Nature",
address = "United States",
}

@STRING{ActaNumer = {Acta Numer.}}

@STRING{CompApplMath = {J. Comput. Appl. Math.}}

@STRING{Computing = {Computing}}

@STRING{E2MNA = {ESAIM: Math. Model. Numer. Anal.}}

@STRING{JComputPhy = {J. Comput. Phys.}}

@STRING{NumerMath = {Numer. Math.}}

@STRING{SIAM = {Society for Industrial and Applied Mathematics}}

@STRING{SIAMMatrix = {{SIAM} J. Matrix Anal. Appl.}}

@STRING{SIAMReview = {{SIAM} Rev.}}

@STRING{Springer = {Springer-Verlag}}

\end{document}